\documentclass{amsart}

\usepackage{xparse}
\usepackage{etoolbox}
\usepackage{aliascnt}

\usepackage[backend=biber, style=alphabetic, url=false, maxbibnames=99, sorting=nyt]{biblatex}
\RequirePackage{doi}

\DeclareFieldFormat{postnote}{#1}
\DeclareFieldFormat{multipostnote}{#1}

\usepackage{courier}
\usepackage[T1]{fontenc}

\usepackage[scaled]{helvet}
\usepackage[mathbf]{euler}
\usepackage[driver=pdftex,margin=3cm,heightrounded=true,centering]{geometry}

\usepackage{enumerate}
\usepackage{amssymb}
\usepackage{amsopn}
\usepackage{amsmath}
\usepackage{mathtools}

\usepackage{tikz}
\usepackage{tikz-cd}
\usepackage{ifthen}
\usepackage{graphics}
\usepackage{hyperref}
\hypersetup{
	colorlinks=true,
	linkcolor=blue,
	filecolor=blue,
	citecolor = blue,
	urlcolor=blue,
	bookmarksdepth=3,
	breaklinks=true,
}

\usepackage{amsthm}
\usepackage{topthm}

\usepackage[textsize = footnotesize]{todonotes}
\usepackage{comment}

\author{Thorben Kastenholz}
\thanks{TK was supported by the DFG (German Research
  Foundation) – SPP 2026, Geometry at Infinity - Project 73: Geometric Chern
  characters in p-adic equivariant K-theory}
\address{University of G\"ottingen, Mathematisches Institut, Bunsenstrasse 3-5, 37073 G\"ottingen, Germany}
\email{thorben.kastenholz@mathematik.uni-goettingen.de}

\author{Robin J. Sroka}
\thanks{RJS was supported by NSERC Discovery Grant A4000 in connection with a
Postdoctoral Fellowship at McMaster University, by the Swedish Research Council
under grant no.\ 2016-06596 while in residence at Institut Mittag-Leffler in
Djursholm, Sweden during the semester \emph{Higher algebraic structures in
algebra, topology and geometry}, and by the European Research Council (ERC
grant agreement No.772960) and the Danish National Research Foundation (DNRF92,
DNRF151) as a PhD Fellow at the University of Copenhagen.}
\address{University of Münster, Mathematisches Institut, Einsteinstrasse 62, 48149 Münster, Germany}
\email{robinjsroka@uni-muenster.de}

\title{Simplicial bounded cohomology and stability}

\begin{document}

\newcommand{\introduce}[1]
  {\emph{#1}}
\newcommand{\tk}[1]
  {\todo[size=\tiny,color=green!40]{TK: #1}}
\newcommand{\rs}[1]
  {\todo[size=\tiny,color=blue!40]{RS: #1}}
\newcommand\blfootnote[1]{
  \begingroup
  \renewcommand\thefootnote{}\footnote{#1}
  \addtocounter{footnote}{-1}
  \endgroup
}

\newcommand{\on}[1]{\operatorname{#1}}
\newcommand{\apply}[2]
  {{#1}\!\left({#2}\right)}
\newcommand{\at}[2]
  {\left.{#1}\right\rvert_{#2}}
\newcommand{\Identity}
  {\mathrm{Id}}
\newcommand{\NaturalNumbers}
  {\mathbb{N}}
\newcommand{\Integers}
  {\mathbb{Z}}
\newcommand{\Rationals}
  {\mathbb{Q}}
\newcommand{\Reals}
  {\mathbb{R}}
\newcommand{\ComplexNumbers}
  {\mathbb{C}}
\newcommand{\Field}
  {\mathbb{K}}
\newcommand{\SubVectorSpace}
  {V}
\newcommand{\AbstractProjection}[1]
  {p_{#1}}
\newcommand{\Norm}[1]
  {\left|\left| #1 \right|\right|}
\newcommand{\AbsoluteValue}[1]
  {\left| #1 \right|}
\newcommand{\Inclusion}
  {\iota}
\newcommand{\Factorial}[1]
  {#1!}
\newcommand{\ellone}
  {\ell^{1}}
\newcommand{\Retraction}
  {r}
\newcommand{\im}
  {Im}
\newcommand{\Colim}
  {\mathrm{colim}}

\newcommand{\Ring}
  {R}
\newcommand{\StableRank}
  {sr}
\newcommand{\Module}
  {M}
\newcommand{\Group}
  {\Gamma}
\newcommand{\Subgroup}
  {H}
\newcommand{\GroupElement}
  {g}
\newcommand{\Genus}
  {g}
\newcommand{\QuadraticModule}
  {\mathbf{M}}
\newcommand{\WittIndex}[1]
  {\apply{\Genus}{#1}}
\newcommand{\StableWittIndex}[1]
  {\apply{\overline{\Genus}}{#1}}
\newcommand{\ChainContraction}[1]
  {H_{#1}}
\newcommand{\GLnShort}[1]
  {\mathrm{GL}_{#1}}
\newcommand{\GLn}[2]
  {\apply{\GLnShort{#1}}{#2}}
\newcommand{\Hom}
  {\mathrm{Hom}}
\newcommand{\HomSpace}[2]
  {\apply{\Hom}{#1, #2}}
\newcommand{\BoundaryChainComplex}[1]
  {\partial_{#1}}
\newcommand{\AbstractChainComplex}[1]
  {C_{#1}}

\newcommand{\SectionSetup}
  {\operatorname{s}}
\newcommand{\EpimorphismSetup}
  {\pi}
\newcommand{\Stabilizer}
  {H}
\newcommand{\CompatibilityFunction}
  {\tau}
\newcommand{\AcyclicityFunction}
  {\gamma}

\newcommand{\HomologyClass}
  {\alpha}
\newcommand{\HomologyOfSpaceObject}[3]
  {\apply{H_{#1}}{#2 ; #3}}
\newcommand{\ReducedHomologyOfSpaceObject}[3]
  {\apply{\widetilde{H}_{#1}}{#2 ; #3}}
\newcommand{\CohomologyOfSpaceObject}[3]
  {\apply{H^{#1}}{#2 ; #3}}
\newcommand{\BoundedCohomologyOfSpaceObject}[3]
  {\apply{H^{#1}_{\text{b}}}{#2 ; #3}}
\newcommand{\ContinuousBoundedCohomologyOfSpaceObject}[3]
  {\apply{H^{#1}_{\text{b,c}}}{#2 ; #3}}
\newcommand{\BoundedCohomologySymbol}{H^{*}_{\text{b}}}
\newcommand{\BoundedCohomologyOfSimplicialObject}[3]
  {\apply{H^{#1}_{\text{b, s}}}{#2 ; #3}}
\newcommand{\ReducedBoundedCohomologyOfSimplicialObject}[3]
  {\apply{\widetilde{H}^{#1}_{\text{b, s}}}{#2 ; #3}}
\newcommand{\HomologyOfSpaceMorphism}[1]
  {{#1}_{\ast}}
\newcommand{\HomologyOfGroupObject}[3]
  {\apply{H_{#1}}{#2; #3}}
\newcommand{\HomologyOfGroupMorphism}[1]
  {{#1}_{\ast}}
\newcommand{\HomologyOfSpacePairObject}[3]
  {\apply{H_{#1}}{{#2},{#3}}}
\newcommand{\Multiple}
  {\lambda}
\newcommand{\CoefficientModule}
  {R}
\newcommand{\Spectralsequence}[3]
  {E^{#1,#2}_{#3}}
\newcommand{\SSDifferential}[3]
  {d^{#1,#2}_{#3}}
\newcommand{\BoundedCohomologyChainComplex}[3]
  {\apply{H^{#1}_{\text{b, s}}}{#2 ; #3}}

\newcommand{\TopologicalSpace}
  {X}
\newcommand{\Point}
  {\ast}
\newcommand{\ContinuousMap}
  {f}
\newcommand{\ContinuousMapALT}
  {g}
\newcommand{\maps}
  {\ensuremath{\text{maps}}}
\newcommand{\HomotopyGroupOfObject}[3]
  {\apply{\pi_{#1}}{{#2},{#3}}}
\newcommand{\HomotopyGroupOfPairObject}[4]
  {\apply{\pi_{#1}}{{#2},{#3},{#4}}}
\newcommand{\HomotopyGroupMorphism}[1]
  {{#1}_{\ast}}
\newcommand{\EMSpace}[2]
  {\apply{K}{{#1},{#2}}}
\newcommand{\ClassifyingSpace}[1]
  {B#1}
\newcommand{\EG}[1]
  {E#1}
\newcommand{\UniversalCovering}[1]
  {\widetilde{#1}}
\newcommand{\UniversalCoveringMap}[1]
  {\widetilde{#1}}
\newcommand{\HomotopyQuotient}[2]
  {#1 \times_{#2} \EG{#2}}

\newcommand{\SimplicialComplex}
  {X}
\newcommand{\SimplicialComplexS}
  {S}
\newcommand{\AuxSimplicialComplex}
  {K}
\newcommand{\Subcomplex}
  {Y}
\newcommand{\AuxSubcomplex}
  {L}
 \newcommand{\ChainComplexWithCoeffients}[3]
  {\apply{C_{#1}}{#2; #3}}
\newcommand{\ChainComplex}[2]
  {\apply{C_{#1}}{#2}}
\newcommand{\ReducedChainComplex}[2]
{\apply{\widetilde{C}_{#1}}{#2}}
\newcommand{\CoChainComplex}[2]
  {\apply{C^{#1}}{#2}}
\newcommand{\BoundedCoChainComplex}[2]
  {\apply{C_{b}^{#1}}{#2}}
\newcommand{\Simplex}[1]
  {\sigma_{#1}}
\newcommand{\Linkfunction}
  {\text{Lk}}
\newcommand{\Link}[2]
  {\apply{\Linkfunction_{#1}}{#2}}
\newcommand{\Starfunction}
  {\text{St}}
\newcommand{\Star}[2]
  {\apply{\Starfunction_{#1}}{#2}}
\newcommand{\BoundaryIndexSimplex}[2]
  {\apply{\partial_{#1}}{#2}}
\newcommand{\BoundarySimplex}
  {\partial}
\newcommand{\FaceMap}
  {d}
\newcommand{\Coboundary}
  {\delta}
\newcommand{\StandardSimplex}[1]
  {\Delta^{#1}}
\newcommand{\Horn}[2]
  {\Lambda_{#1}^{#2}}
\newcommand{\vertex}
  {v}
\newcommand{\GeometricRealization}[1]
  {\left\lvert #1 \right\rvert}
\newcommand{\BoundHomotopy}
  {N}
\newcommand{\Chain}
  {\sigma}
\newcommand{\BoundingChain}
  {\rho}
\newcommand{\Cone}[1]
  {C#1}
\newcommand{\Homotopy}
  {H}
\newcommand{\BarycentricSubdivion}[1]
  {\apply{\operatorname{Sd}}{#1}}

\newcommand{\UBCConstant}[2]
  {K^{#1}_{#2}}
\newcommand{\UBCConstantSuspension}[3]
  {\apply{K^{\Sigma}}{#1, #2, #3}}
\newcommand{\UBCConstantCellAttachment}[3]
  {\apply{K^{\on{Cell}}}{#1, #2, #3}}
\newcommand{\UBCConstantMorseTheory}[4]
  {\apply{K^{\on{MT}}}{#1, #2, #3, #4}}
\newcommand{\UBCConstantMayerVietoris}[4]
  {\apply{K^{\on{MV}}}{#1, #2, #3, #4}}
\newcommand{\UBCConstantTwoOutOfThreeThree}[3]
  {\apply{K^{\on{Three}}}{#1, #2, #3}}
\newcommand{\UBCConstantTwoOutOfThreeTwo}[3]
  {\apply{K^{\on{Two}}}{#1, #2, #3}}
\newcommand{\UBCConstantTwoOutOfThreeOne}[3]
  {\apply{K^{\on{One}}}{#1, #2, #3}}
\newcommand{\UBCConstantFactorThrough}[3]
  {\apply{K^{\on{Fact}}}{#1, #2, #3}}
\newcommand{\UBCConstantOrdConstruction}[2]
  {\apply{K^{\on{ord}}}{#1, #2}}
\newcommand{\UBCConstantOrdConstructionStar}[2]
  {\apply{K^{\on{ord}}_{\on{St}}}{#1, #2}}
\newcommand{\UBCConstantTechnicalLemmaI}[3]
  {\apply{K^{\on{TLI}}}{#1, #2, #3}}
\newcommand{\UBCConstantTechnicalLemmaII}[2]
  {\apply{K^{\on{TLII}}}{#1, #2}}
\newcommand{\UBCConstantIntermediateGLI}[5]
  {\apply{K^{\on{GLI}}}{#1, #2, #3, #4, #5}}
\newcommand{\UBCConstantIntermediateGLIplus}[5]
  {\apply{K^{\on{GLI^+}}}{#1, #2, #3, #4, #5}}
\newcommand{\UBCConstantIntermediateGLIplusP}[6]
  {\apply{K_{P_{#1}}^{\on{GLI^+}}}{#2, #3, #4, #5, #6}}
\newcommand{\UBCConstantIntermediateGLIplusQ}[6]
  {\apply{K_{Q_{#1}}^{\on{GLI^+}}}{#2, #3, #4, #5, #6}}
\newcommand{\UBCConstantIntermediateGLIplusLkQ}[6]
  {\apply{K_{\on{Lk}_{Q_{#1}}}^{\on{GLI^+}}}{#2, #3, #4, #5, #6}}
\newcommand{\UBCConstantIntermediateGLIplusLkQplus}[6]
  {\apply{K_{\on{Lk}_{Q_{#1}}^+}^{\on{GLI^+}}}{#2, #3, #4, #5, #6}}
\newcommand{\UBCConstantIntermediateGLII}[4]
  {\apply{K^{\on{GLII}}}{#1, #2, #3, #4}}
\newcommand{\UBCConstantIntermediateGLIIPOne}[5]
  {\apply{K_{P_{#1},(1)}^{\on{GLII}}}{#2, #3, #4, #5}}
\newcommand{\UBCConstantIntermediateGLIILk}[5]
  {\apply{K_{\on{Lk}_{P_{#1}}}^{\on{GLII}}}{#2, #3, #4, #5}}
\newcommand{\UBCConstantIntermediateGLIIOne}[4]
  {\apply{K^{\on{GLII}}_{(1)}}{#1, #2, #3, #4}}
\newcommand{\UBCConstantIntermediateGLIIPTwo}[5]
  {\apply{K_{P_{#1}, (2)}^{\on{GLII}}}{#2, #3, #4, #5}}
\newcommand{\UBCConstantIntermediateGLIITwo}[4]
  {\apply{K_{(2)}^{\on{GLII}}}{#1, #2, #3, #4}}
\newcommand{\UBCConstantGLOne}[2]
  {\apply{K^{(1)}}{#1, #2}}
\newcommand{\UBCConstantGLTwo}[3]
  {\apply{K^{(2)}}{#1, #2, #3}}
\newcommand{\UBCConstantGLThree}[2]
  {\apply{K^{(3)}}{#1, #2}}
\newcommand{\UBCConstantGLFour}[3]
  {\apply{K^{(4)}}{#1, #2, #3}}

\newcommand{\Poset}
  {F}
\newcommand{\Subposet}
  {S}
\newcommand{\ElementsBelow}[2]
  {#1_{\leq#2}}
\newcommand{\ElementsAbove}[2]
  {#1_{\geq#2}}
\newcommand{\PosetElement}
  {x}
\newcommand{\OrderedSequences}[1]
{\mathcal{O}(#1)}
\newcommand{\UnimodularSequences}[1]
  {\mathcal{U}(#1)}
\newcommand{\TitsBuilding}[2]
  {T^{#1}_{#2}}

\newcommand{\CMub}{\operatorname{CM_{ub}}}
\newcommand{\lCMub}{\operatorname{lCM_{ub}}}
\newcommand{\Ord}[2]{#1^{\operatorname{ord}}_{#2}}

\newcommand{\SpaceOfBoundedOperators}[2]
  {\apply{B}{#1,#2}}

\newcommand{\U}[4]{\apply{\mathrm{U}^{#4}_{#1}}{#2, #3}}
\newcommand{\Aut}[4]{\apply{\mathrm{Aut}^{#4}_{#1}}{#2, #3}}
\newcommand{\FormParameter}{(\epsilon, \Lambda)}
\newcommand{\BilinearForm}{\lambda}
\newcommand{\QuadraticFrom}{\mu}
\newcommand{\KComplex}[1]{\apply{\mathscr{K}^a}{#1}}
\newcommand{\HyperbolicModule}{\mathcal{H}}
\newcommand{\UBCUnitaryGroups}{K^{\on{Aut}}}

\begin{abstract}
	We introduce a set of combinatorial techniques for studying the simplicial
	bounded cohomology of semi-simplicial sets, simplicial complexes and
	posets. We apply these methods to prove several new
	\emph{bounded} acyclicity results for semi-simplicial sets appearing in the
	homological stability literature. Our strategy is to recast classical
	arguments (due to Bestvina, Maazen, van der Kallen, Vogtmann, Charney and, recently,
	Galatius--Randal-Williams) in the setting of bounded cohomology using
	uniformly bounded refinements of well-known simplicial tools. Combined with
	ideas developed by Monod and De la Cruz Mengual--Hartnick, we deduce
	slope-$1/2$ stability results for the bounded cohomology of two
	large classes of linear groups: general linear groups over any ring with
	finite Bass stable rank and certain automorphism groups of quadratic modules over
	the integers or any field of characteristic zero. We expect that many other results in the
	literature on homological stability admit bounded cohomological analogues
	by applying the blueprint provided in this work.
\end{abstract}

\maketitle
\setcounter{tocdepth}{1}
\tableofcontents

\section{Introduction}
Bounded cohomology is a ``norm-enriched'' version of group cohomology. It
assigns to a group $\Group$ a sequence of real vector spaces
$\{\BoundedCohomologyOfSpaceObject{q}{\Group}{\Reals}\}_{q \in \NaturalNumbers}$,
each equipped with a seminorm. There is a natural comparison map
\begin{equation}
	\label{eq:comparison-map-groups}
\BoundedCohomologyOfSpaceObject{q}{\Group}{\Reals} \to \CohomologyOfSpaceObject{q}{\Group}{\Reals}
\end{equation}
to the usual group cohomology of $\Group$ with trivial $\Reals$-coefficients. For
discrete groups bounded cohomology was introduced by Johnson and Trauber
\cite{johnson1972}, and further pioneered by Gromov \cite{gromov1982}, who used
it to study characteristic classes of flat $\on{GL}_n$-bundles and the volume
of manifolds. For (locally compact) topological groups, the theory was
developed by Burger--Monod \cite{burgermonod2002} and is referred to as
continuous bounded cohomology. Since then (continuous) bounded cohomology has
found many applications in geometric topology, e.g.~ \cite{dupont1978,
matsumotomorita1985boundedcohomologyofcertaingroupsofhomeomorphisms,
bucherkarlsson2007,MonodStabilization, MonodSemiSimple, NarimanMonod}.

Computing the (continuous) bounded cohomology of a group $\Group$ is a challenging
problem and only few results are contained in the literature (see e.g.\
\cite{monod2006}). The known results usually come in two flavours: acyclicity
theorems stating that $\BoundedCohomologyOfSpaceObject{q}{\Group}{\Reals} = 0$ if $q
\neq 0$ \cite{johnson1972, gromov1982,
matsumotomorita1985boundedcohomologyofcertaingroupsofhomeomorphisms, loeh2017,
fournierfacioloehmoraschini2022} or non-vanishing results in low cohomological
degrees $q$ \cite{brooks1981, grigorchuk1995, teruhiko1997,
delacruzthirdbounded, DCMThesis, buchermonod2019,
franceschinifrigeriopozzettisisto2019}. Monod \cite{MonodStabilization,
MonodSemiSimple} and, recently, De la Cruz Mengual--Hartnick
\cite{DCMThesis, HartnickDeLaCruzQuillen, HartnickDeLaCruzStability} introduced homological
stability ideas to the field, which allows to leverage low-degree computations
and have a long history as fruitful tools for computations in the setting of
usual group (co-)homology (see e.g.\ \cite{wahl2022}). Related techniques enabled
Monod--Nariman \cite{NarimanMonod} to achieve the first computation of a
non-trivial bounded cohomology ring.

The first goal of this work is to complement the recent development of
computation tools by introducing a set of simplicial and combinatorial
techniques crafted for the study of bounded cohomology. The second is to
showcase some applications of these tools within the realm of homological
stability ideas.

The combinatorial tools, that we introduce, are bounded cohomological
refinements of standard techniques that have long and successfully been
employed in the study of classical group (co-)homology (see e.g.\
\cite{quillen1978, Björner_1995, kozlov2008}).
This is a feature. Indeed, we show that many of these techniques
can be adapted to the setting of bounded cohomology using small variations of
one simple idea. The theoretical upshot of this work is therefore a blueprint
that explains how well-established simplicial methods in the literature on
group (co-)homology can be adapted to the ``norm-enriched'' setting of bounded
(co-)homology. The authors expect that this approach has many other
applications. The present article serves as a proof of concept.

\subsection{Uniformly bounded simplicial methods and acyclicity results}
\label{sec:introduction-simplicial-methods}

Simplicial and combinatorial geometries enter the ``norm-enriched'' theory
(among others \cite{IvanovSimplicialMappingTheorem}) via the following approach
for studying the bounded cohomology of a discrete group $\Group$ (compare e.g.\
\cite{MonodStabilization, monod2007, NarimanMonod,
HartnickDeLaCruzQuillen, HartnickDeLaCruzStability}):
Consider the action of $\Group$ on a semi-simplicial set or ordered simplicial
complex $\SimplicialComplex_\bullet$. There is a notion of bounded cohomology
for $\SimplicialComplex_\bullet$ called \emph{simplicial bounded cohomology},
which we denoted by
$\BoundedCohomologyOfSimplicialObject{q}{\SimplicialComplex_\bullet}{\Reals}$.
If the reduced simplicial bounded cohomology of $\SimplicialComplex_\bullet$
vanishes in a range of degrees, one obtains an isotropy spectral sequence
which converges to the bounded cohomology of $\Group$ in a range.

Hence, a careful
analysis of this spectral sequence and the bounded cohomology of the
isotropy groups of the $\Group$-action on $\SimplicialComplex_\bullet$ yields a
strategy for obtaining information about or even computing
$\BoundedCohomologyOfSpaceObject{q}{\Group}{\Reals}$ in a range of cohomological
degrees. This is completely analogous to a standard approach for studying
the classical group (co-)homology of $\Group$ (see e.g.\ \cite[Chapter
VII]{Brown_1982}).

The set of combinatorial techniques that we introduce in this work can be used
to prove \emph{bounded acyclicity} results for semi-simplicial sets
$X_\bullet$. In particular, this toolbox can be used to check the condition
on which the strategy for studying the bounded cohomology of a discrete group
$\Group$ in the previous paragraph relies: that the reduced simplicial bounded
cohomology
$\ReducedBoundedCohomologyOfSimplicialObject{q}{\SimplicialComplex_\bullet}{\Reals}$
 ought to be zero in a range of degrees. We illustrate how these methods can be
 applied
 by establishing new bounded acyclicity results for
several well-known combinatorial geometries attached to a large class of
discrete groups: general linear groups $\GLn{n}{\Ring}$ over any ring $\Ring$
of finite Bass stable rank as well as, if $\Ring$ is $\Integers$ or any field
of characteristic zero, certain automorphism groups of quadratic
modules $\Aut{n}{\Ring}{\Lambda}{\epsilon}$, i.e.\ essentially the symplectic
groups $\on{Sp}_{2n}(\Ring)$ and orthogonal groups $\on{O}_{n,n}(\Ring)$.

\begin{theorem}
	\label{thm:general-connectivity}
	If $\SimplicialComplex_\bullet$ is one of the following semi-simplicial sets, then its reduced simplicial bounded cohomology vanishes, $\ReducedBoundedCohomologyOfSimplicialObject{q}{\SimplicialComplex_\bullet}{\Reals} = 0$,
	in a range of degrees $q \leq c(\SimplicialComplex_\bullet)$.
	\begin{enumerate}
		\item \textbf{$\on{GL}_n$-complexes:}
		\begin{enumerate}
			\item \label{item:solomon-tits-an} the Tits building of type $\mathtt{A}_{n-1}$ of any field $\Field$ 
			(see \autoref{def:tits-building}) with $c(\SimplicialComplex_\bullet) = n - 2$.
			\item \label{item:stability-complex-gl} the complex of $R$-split
			injections into $\Ring^n$ for any ring $\Ring$ (see
			\autoref{def:ComplexOfSplitInjections}) with
			$c(\SimplicialComplex_\bullet) = n - \StableRank(R) - 1$, where
			$\StableRank(R) \in [1, \infty]$ is the Bass stable rank of $R$ (see \autoref{def:stable-rank}).
		\end{enumerate}
		\item \textbf{$\on{Aut}_n^\epsilon$-complexes:}
		\begin{enumerate}
			\item \label{item:solomon-tits-bncn} the Tits building of type
			$\mathtt{C}_{n}$ of any field $\Field$ (see 
			\autoref{def:symplectic-tits-building}) with
			$c(\SimplicialComplex_\bullet) = n - 1$.
			\item \label{item:stability-complex-sp} the complex of hyperbolic
			split injections into certain quadratic modules
			$\HyperbolicModule^n = \Ring^{2n}$ for
			$\Ring = \Integers$ or any field of characteristic zero (see
			\autoref{def:ComplexOfHyperbolicSplitInjections}) with
			$c(\SimplicialComplex_\bullet) = \lfloor \frac{n-4}{2} \rfloor$.
		\end{enumerate}
	\end{enumerate}
\end{theorem}

\autoref{item:solomon-tits-an} and \autoref{item:solomon-tits-bncn} of
\autoref{thm:general-connectivity} are bounded cohomology analogues of the
celebrated Solomon--Tits theorem
\cite{solomon1969thesteinbergcharacterofafinitegroupwithbnpair}; 
a fundamental result in the theory of buildings 
which is important in the study of algebraic groups (see e.g.\
\cite{borelserre1973}). The classical Solomon--Tits theorem implies that these
complexes have the homotopy type of a bouquet of spheres. Similarly,
\autoref{thm:general-connectivity} states that their reduced simplicial bounded
cohomology is also concentrated in a single degree, $q =
c(\SimplicialComplex_\bullet)$. A version of this result for continuous bounded cohomology 
was previously established and used by Monod (see \cite[Theorem~3.9]{MonodSemiSimple}).
The complexes in
\autoref{item:stability-complex-gl} and \autoref{item:stability-complex-sp} are
well-known to experts interested in homological stability ideas. They have been
studied by van der Kallen
\cite{vanderkallen1980homologystabilityforlineargroups} and Maazen
\cite{maazen79, MaazenStability}, Vogtmann \cite{VogtmannStability}, Charney \cite{CharneyStability} and, recently,
Galatius--Randal-Williams \cite{GalatiusRandalWilliamsStability}. In
particular, van der Kallen and Galatius--Randal-Williams established that their
homotopy groups are trivial, $\pi_{q}(\SimplicialComplex_\bullet) = 0$, in 
the range of degrees $q \leq c(\SimplicialComplex_\bullet)$ stated in \autoref{thm:general-connectivity}. Again, our result is an analogue
proving that the same holds for their reduced simplicial bounded cohomology.

The similarity of \autoref{thm:general-connectivity} to such classical results is not a coincidence: The aim of this work is to convince the reader that many of the existing acyclicity arguments in the literature admit refinements to the setting of simplicial bounded cohomology. In particular, \autoref{thm:general-connectivity} is obtained by recasting simplicial and poset theoretic arguments due to Bestvina \cite{BestvinaPLMorseTheory}, van der Kallen \cite{vanderkallen1980homologystabilityforlineargroups} and Galatius--Randal-Williams \cite{GalatiusRandalWilliamsStability} using the combinatorial tools we developed.

The central idea behind our \emph{uniformly bounded simplicial methods} can be summarized as follows: Let $\SimplicialComplex_\bullet$ be a semi-simplicial set or ordered simplicial complex. Similar to \autoref{eq:comparison-map-groups}, there is a comparison map between the simplicial bounded cohomology of $\SimplicialComplex_\bullet$ and its ordinary simplicial cohomology
\begin{equation}
	\label{eq:comparison-map-simplicial}
	\BoundedCohomologyOfSpaceObject{q}{\SimplicialComplex_\bullet}{\Reals} \to \CohomologyOfSpaceObject{q}{\SimplicialComplex_\bullet}{\Reals}.
\end{equation}
The results in the literature (i.e.\ for \autoref{thm:general-connectivity}
e.g.\ \cite{solomon1969thesteinbergcharacterofafinitegroupwithbnpair,
BestvinaPLMorseTheory, maazen79, MaazenStability,
vanderkallen1980homologystabilityforlineargroups, VogtmannStability, CharneyStability,
GalatiusRandalWilliamsStability}) imply that
$\CohomologyOfSpaceObject{q}{\SimplicialComplex_\bullet}{\Reals} = 0$ for $q
\leq c(\SimplicialComplex_\bullet)$. Therefore, it suffices to prove that the
comparison map in \autoref{eq:comparison-map-simplicial} is injective for $q
\leq c(\SimplicialComplex_\bullet)$ to obtain a bounded acyclicity theorem. The
seminal work of Matsumoto--Morita
\cite{matsumotomorita1985boundedcohomologyofcertaingroupsofhomeomorphisms}
introduced a condition on the ordinary simplicial $\Reals$-chain complex
$\ChainComplexWithCoeffients{*}{\SimplicialComplex_\bullet}{\Reals}$ of the
semi-simplicial set $\SimplicialComplex_\bullet$, which implies the desired
injectivity. This condition is called the \emph{uniform boundary condition},
short UBC. Our strategy is thus to check that this conditions holds in the
relevant range of degrees $q$. The key point is this:
$\ChainComplexWithCoeffients{*}{\SimplicialComplex_\bullet}{\Reals}$ is exactly
the same chain complex that one needs to investigate if one is interested in
the ordinary simplicial homology,
$\CohomologyOfSpaceObject{q}{\SimplicialComplex_\bullet}{\Reals} = 0$ for $q
\leq c(\SimplicialComplex_\bullet)$, of $\SimplicialComplex_\bullet$. In
particular, there is a vast amount of literature in algebraic topology
containing strategies for studying it (e.g.\ \cite{Björner_1995, kozlov2008});
ranging from the homotopy theory of posets (e.g.\ \cite{quillen1978}),
subdivision techniques (e.g.\ \cite{rourkesanderson1971}) and discrete Morse
theory (e.g.\ \cite{BestvinaPLMorseTheory} and \cite{forman1998}) to standard
tools such as the Mayer--Vietoris sequence and the long exact sequence of a
pair. The uniformly bounded simplicial methods introduced in this work are
refinements of a selection of these standard tools, with an extra layer of
complexity that keeps track of the UBC. These ``norm-enriched'' combinatorial
tools then allow us to inductively check that UBC holds, while running
connectivity arguments due to Bestvina \cite{BestvinaPLMorseTheory}, van der
Kallen \cite{vanderkallen1980homologystabilityforlineargroups} and
Galatius--Randal-Williams \cite{GalatiusRandalWilliamsStability}.

We expect that similar ideas can be used to adapt many other simplicial or homotopy theoretic techniques in the literature to the ``norm-enriched''` setting of bounded cohomology.

\subsection{Stability patterns in the bounded cohomology of linear groups}

Homological stability ideas originate from Quillen's work on algebraic K-theory \cite{quillen1971}, and have proven to be an extremely fruitful tool for computations in the setting of group (co-)homology (see e.g.\ \cite{wahl2022}). The homology of many important sequences of groups are known to satisfy stability patterns, for example: Symmetric groups \cite{NakaokaSymmetricGroups}, mapping class groups of surfaces \cite{HarerStability, wahl2008}, automorphism groups of free groups
\cite{hatchervogtmann1998, HatcherVogtmann, HatcherVogtmannWahl, hatcherwahl2005}, general linear groups \cite{maazen79, MaazenStability, Charney1980, vanderkallen1980homologystabilityforlineargroups, friedrich2016homologicalstabilityofautomorphismgroupsofquadraticmodulesandmanifolds}, automorphism groups
of formed spaces e.g.\ $\apply{\mathrm{O}_{n,n}}{\Ring}$,
$\apply{\mathrm{Sp}_{2n}}{\Ring}$ and $\apply{\mathrm{U}_{n}}{\Ring}$
\cite{VogtmannStability, CharneyStability, mirzaiivanderkallen2002, vanderkallenlooijenga2011, friedrich2016homologicalstabilityofautomorphismgroupsofquadraticmodulesandmanifolds}, (discrete) diffeomorphism groups of certain high-dimensional manifolds \cite{GalatiusRandalWilliamsStability, friedrich2016homologicalstabilityofautomorphismgroupsofquadraticmodulesandmanifolds, NarimanStability} and Coxeter groups \cite{HepworthStability}.

Monod \cite{MonodStabilization, MonodSemiSimple} and, recently, De la
Cruz Mengual--Hartnick \cite{DCMThesis, HartnickDeLaCruzQuillen,
HartnickDeLaCruzStability} introduced
homological stability techniques to the realm of (continuous) bounded
cohomology. We
say that a nested sequence of groups
\[
\Group_0
\xhookrightarrow{}
\Group_1
\xhookrightarrow{}
\ldots
\xhookrightarrow{}
\Group_n
\xhookrightarrow{\iota_n}
\Group_{n+1}
\xhookrightarrow{}
\ldots
\]
satisfies \emph{bounded cohomological stability} with slope $s \in \Rationals_{\geq 0}$ and offset $c \in \Reals$ if the inclusion induced map $\BoundedCohomologyOfSpaceObject{q}{\iota_n}{\Reals}$ in the sequence obtained by passing to $q$-th bounded cohomology
\[
\BoundedCohomologyOfSpaceObject{q}{\Group_0}{\Reals}
\xleftarrow{}
\BoundedCohomologyOfSpaceObject{q}{\Group_1}{\Reals}
\xleftarrow{}
\ldots
\xleftarrow{}
\BoundedCohomologyOfSpaceObject{q}{\Group_n}{\Reals}
\xleftarrow{\BoundedCohomologyOfSpaceObject{q}{\iota_n}{\Reals}}
\BoundedCohomologyOfSpaceObject{q}{\Group_{n+1}}{\Reals}
\xleftarrow{}
\ldots
\]
is an isomorphism whenever $s \cdot n + c \geq q$. In other words, the inclusion maps eventually induce isomorphisms in bounded cohomology.

De la Cruz Mengual--Hartnick \cite{HartnickDeLaCruzQuillen} introduced a
framework, which can be seen as a bounded cohomology
analogue of Randal-Williams--Wahl's work for classical group (co-)homology
\cite{randalwilliamswahl2017homologicalstabilityforautomorphismgroups}, within
which bounded cohomological stability questions can be formulated and studied.
Working in this framework, the most difficult part for establishing a bounded
cohomology stability result for a nested sequence of discrete groups
$\{\Group_n\}_{n \in \NaturalNumbers}$ is the proof of a bounded acyclicity
theorem of the type outlined at the beginning of
\autoref{sec:introduction-simplicial-methods}: For each group $\Group_n$ one
needs to construct semi-simplicial set $\SimplicialComplex^n_\bullet$ on which
$\Group_n$ acts. Among others, the sequence of semi-simplicial sets
$\{\SimplicialComplex^n_\bullet\}_{n \in \NaturalNumbers}$ ought to be highly
boundedly acyclic (increasingly as $n$ goes to $\infty$).

\begin{remark}
	\cite{HartnickDeLaCruzQuillen} only state their framework for countable
	discrete groups. This is an artefact of them primarily being interested in
	the	continuous bounded cohomology of locally compact second countable
	topological	groups. In \autoref{rem:CountabilityQuillen}, we
	explain why their result holds for arbitrary discrete groups.
\end{remark}

The second goal of this work is to explain how \autoref{thm:general-connectivity} implies the following bounded cohomological stability results.

\begin{theorem} \label{thm:BoundedCohomologicalStabilityForLinearGroups}
  Let $\Ring$ be an unital ring and consider the sequence of general linear groups
  \[
    \GLn{0}{\Ring}
    \hookrightarrow
    \GLn{1}{\Ring}
    \hookrightarrow
    \ldots
    \hookrightarrow
    \GLn{n}{\Ring}
    \xhookrightarrow{\iota_n}
    \GLn{n+1}{\Ring}
    \hookrightarrow
    \ldots
  \]
  Then the inclusion $\GLn{n}{\Ring} \xhookrightarrow{\iota_n} \GLn{n+1}{\Ring}$ induces an isomorphism
  \[
    \BoundedCohomologyOfSpaceObject{q}{\GLn{n}{\Ring}}{\Reals}
    \xleftarrow{\BoundedCohomologyOfSpaceObject{i}{\iota_n}{\Reals}}
    \BoundedCohomologyOfSpaceObject{q}{\GLn{n+1}{\Ring}}{\Reals}
  \]
  whenever $q \leq \frac{n-\StableRank(\Ring)}{2}$, and an injection whenever $q \leq \frac{n-\StableRank(\Ring)}{2} + 1$.
  Here, $\StableRank(\Ring) \in [1, \infty]$ denotes the Bass stable rank of 
  the ring $\Ring$ (see \autoref{def:stable-rank}).
\end{theorem}

Previously, this result was only known for the integers $\Ring = \Integers$ and
the Gaussian integers $\Ring = \Integers[i]$ by work of Monod (combining
\cite[Theorem 1.1]{MonodStabilization} and \cite[Note on p.685]{monod2007}, see
also \cite[Example 1.2]{HartnickDeLaCruzQuillen}, with
\cite[Corollary 4.8]{MonodSemiSimple}).
\autoref{thm:BoundedCohomologicalStabilityForLinearGroups} is the bounded
cohomology analogue of van der Kallen's classical homological stability theorem
for general linear groups
\cite{vanderkallen1980homologystabilityforlineargroups} (see
\cite{friedrich2016homologicalstabilityofautomorphismgroupsofquadraticmodulesandmanifolds}
 for a generalization), and follows from \autoref{item:stability-complex-gl} of
\autoref{thm:general-connectivity}. It applies to a vast set of discrete rings
$R$: any Dedekind domain $R$ satisfies $\StableRank(R) \leq 2$
\cite[4.1.11]{hahnomeara1989} (this includes e.g.\ discrete valuation rings
such as the $p$-adic integers $\Integers_p^\delta$, all euclidean rings such as
the integers $\Integers$ or Gaussian integers $\Integers[i]$, more generally all
principal ideal domains and all rings of integers of a number field
$\Field/\Rationals$); any semi-local ring satisfies $\StableRank(R) = 1$
\cite[4.1.17]{hahnomeara1989} (this includes e.g.\ all discrete fields such as
$\Reals^\delta, \ComplexNumbers^\delta$ and the discrete field of $p$-adic
numbers $\Rationals_p^\delta$ as well as all local rings); more generally any
commutative Noetherian ring of finite Krull dimension $d$ satisfies
$\StableRank(R) \leq d+1$ \cite[4.1.15]{hahnomeara1989}. We remark that e.g.\
the bounded cohomology of $\GLn{n}{\Reals^\delta}$ contains all characteristic
classes of flat $\GLn{n}{\Reals}$-bundles by a result of Gromov
\cite{gromov1982}.

The next stability result concerns automorphism groups
$\Aut{n}{\Ring}{\Lambda}{\epsilon}$ of certain quadratic modules in the sense
of Bak \cite{bak1969, bak1981} over the integers and any
field of characteristic zero $\Ring$. In our setting (see
\autoref{sec:automorphism-groups-of-quadratic-modules}), these are essentially
the symplectic groups $\on{Sp}_{2n}(\Ring)$ and the orthogonal groups
$\on{O}_{n,n}(\Ring)$.

\begin{theorem} \label{thm:BoundedCohomologicalStabilityForUnitaryGroups}
	Let $\Ring$ be the integers $\Integers$ or any field $\Field$ of characteristic zero. Let 
	$\Aut{n}{\Ring}{\Lambda}{\epsilon}$ denote the automorphisms of the $\FormParameter$-quadratic module $\HyperbolicModule^{\oplus n}$ as in \autoref{sec:automorphism-groups-of-quadratic-modules}. Consider the 
	sequence
	\[
	\Aut{0}{\Ring}{\Lambda}{\epsilon}
	\hookrightarrow
	\Aut{1}{\Ring}{\Lambda}{\epsilon}
	\hookrightarrow
	\ldots
	\hookrightarrow
	\Aut{n}{\Ring}{\Lambda}{\epsilon}
	\xhookrightarrow{\iota_n}
	\Aut{n+1}{\Ring}{\Lambda}{\epsilon}
	\hookrightarrow
	\ldots
	\]
	Then the inclusion 
	$
	\Aut{n}{\Ring}{\Lambda}{\epsilon}
	\xhookrightarrow{\iota_n}
	\Aut{n+1}{\Ring}{\Lambda}{\epsilon}
	$
	induces an isomorphism
	\[
	\BoundedCohomologyOfSpaceObject{q}{\Aut{n}{\Ring}{\Lambda}{\epsilon}}{\Reals}
	\xleftarrow{\BoundedCohomologyOfSpaceObject{i}{\iota_n}{\Reals}}
	\BoundedCohomologyOfSpaceObject{q}{\Aut{n+1}{\Ring}{\Lambda}{\epsilon}}{\Reals}
	\]
	whenever $q \leq \frac{n-3}{2}$, and an injection whenever $q \leq \frac{n-3}{2} + 1$.
\end{theorem}

Previously, De la Cruz Mengual--Hartnick \cite{HartnickDeLaCruzStability}
obtained similar stability results for lattices in certain simple Lie groups,
including e.g.\ $\operatorname{Sp}_{2n}(\Ring)$ and
$\operatorname{O}_{n,n}(\Ring)$ for the integers $\Ring = \Integers$ and the
Gaussian integers $\Ring = \Integers[i]$. However, their stable range is
logarithmic in $n$. Our theorem improves this to a slope-$1/2$ range for
$\operatorname{Sp}_{2n}(\Integers)$ and $\operatorname{O}_{n,n}(\Integers)$.
For (discrete) fields of characteristic zero, the results seems to be
new. \autoref{thm:BoundedCohomologicalStabilityForUnitaryGroups} is 
a bounded cohomology analogue of results due to Vogtmann \cite{VogtmannStability}, Charney
\cite{CharneyStability}, Panin \cite{panin1990}, Mirzaii--van der Kallen
\cite{mirzaiivanderkallen2002} and van der Kallen--Looijenga \cite{vanderkallenlooijenga2011}; it follows from
\autoref{item:stability-complex-sp} of \autoref{thm:general-connectivity}.

We remark that Mirazaii--van der Kallen \cite{mirzaiivanderkallen2002} and Friedrich
\cite{friedrich2016homologicalstabilityofautomorphismgroupsofquadraticmodulesandmanifolds}
proved generic slope-$1/2$ homological stability patterns for sequences of
automorphism groups of quadratic modules, conditioned on the \emph{unitary 
stable rank} of $\Ring$ (i.e.\ similar to role of the Bass stable rank
in \autoref{thm:BoundedCohomologicalStabilityForLinearGroups}). In light of
this, we expect that a much more general version of
\autoref{thm:BoundedCohomologicalStabilityForUnitaryGroups} holds true. More
precisely, Friedrich's and Mirazaii--van der Kallen's work should have an
analogue in the setting of bounded cohomology following the blueprint we
provide in this article: Their work builds on ideas and poset techniques that
van der Kallen developed in
\cite{vanderkallen1980homologystabilityforlineargroups}. This suggests that
more general version of \autoref{item:stability-complex-sp} of
\autoref{thm:general-connectivity} should be obtainable. Its proof should be in
the style of \autoref{item:stability-complex-gl} of
\autoref{thm:general-connectivity}, following the strategy of \cite[Section
3]{friedrich2016homologicalstabilityofautomorphismgroupsofquadraticmodulesandmanifolds}
instead of \cite{vanderkallen1980homologystabilityforlineargroups}. Because of
this similarity, we decided to not implement this here and instead refined an
argument due to Galatius--Randal-Williams
\cite{GalatiusRandalWilliamsStability}, which uses a quite different set of
combinatorial techniques. The general case is currently investigated by a
student at the University of Göttingen.

\subsection{Continuous bounded cohomology and previous works}
\label{subsec:bounded-continuous-cohomology}
The theory of continuous bounded cohomology was developed by Burger--Monod
\cite{burgermonod2002} and incorporates the natural topology on groups such as
$\GLn{n}{\Reals}$, $\GLn{n}{\Rationals_p}$ or
$\operatorname{Sp}_{2n}(\ComplexNumbers)$ into bounded cohomology. The
previously known stability results for discrete groups mentioned above have all
been obtained by passing through continuous bounded cohomology: In
\cite{MonodStabilization} and \cite{HartnickDeLaCruzStability}, the authors
first establish a stability theorem for continuous bounded cohomology of an
ambient, topologized Lie groups and then invoke a comparison isomorphism
established by Monod \cite[Corollary 4.8]{MonodSemiSimple} to pass to a
sequence of lattices.
Since the comparison isomorphism works both ways, the continuous bounded
cohomology of $\{\GLn{n}{\Reals}\}_{n \in \NaturalNumbers}$,
$\{\GLn{n}{\ComplexNumbers}\}_{n \in \NaturalNumbers}$, and
$\{\Aut{n}{\Reals}{\Lambda}{\epsilon}\}_{n \in \NaturalNumbers}$ (as in
\autoref{thm:BoundedCohomologicalStabilityForUnitaryGroups})
satisfy stability in continuous bounded cohomology in the slope-$1/2$ range stated in
\autoref{thm:BoundedCohomologicalStabilityForLinearGroups} and
\autoref{thm:BoundedCohomologicalStabilityForUnitaryGroups} (using
$\StableRank(\Integers), \StableRank(\Integers[i]) \leq 2$, see e.g.\ \cite[4.1.11]{hahnomeara1989},
for general linear groups).
For the two sequences of
general linear groups, this recovers the Archimedean case of a result of Monod
\cite{MonodStabilization}, see also \cite[Note on p.685]{monod2007} and
\cite[Example 1.2]{HartnickDeLaCruzQuillen}. For
$\{\Aut{n}{\Reals}{\Lambda}{\epsilon}\}_{n \in
\NaturalNumbers}$ (as in
\autoref{thm:BoundedCohomologicalStabilityForUnitaryGroups}),
this recovers results of De la Cruz Mengual--Hartnick \cite{HartnickDeLaCruzStability} and
improves their stable range from logarithmic in $n$ to a slope-$1/2$ range.

\begin{corollaryi}
	\label{thm:ContinuousBoundedCohomologicalStability}
	Consider the sequence of automorphism groups
	$\{\Aut{n}{\Reals}{\Lambda}{\epsilon}\}_{n \in \NaturalNumbers}$ as in
	\autoref{thm:BoundedCohomologicalStabilityForUnitaryGroups} equipped with
	their standard topologies.
	Then the inclusion map \[\iota_n: \Aut{n}{\Reals}{\Lambda}{\epsilon}
	\hookrightarrow \Aut{n+1}{\Reals}{\Lambda}{\epsilon}\]
	induces an isomorphism in continuous bounded cohomology
	$$
	\ContinuousBoundedCohomologyOfSpaceObject{q}{\Aut{n}{\Reals}{\Lambda}{\epsilon}}{\Reals}
	\xleftarrow{\ContinuousBoundedCohomologyOfSpaceObject{i}{\iota_n}{\Reals}}
	\ContinuousBoundedCohomologyOfSpaceObject{q}{\Aut{n+1}{\Reals}{\Lambda}{\epsilon}}{\Reals}
	$$
	whenever $q \leq \frac{n-3}{2}$, and an injection whenever $q \leq
	\frac{n-3}{2} +1$.
\end{corollaryi}

\subsection{Mapping class groups and classifying spaces of diffeomorphism groups of high-dimensional manifolds}

A more geometric reason to be interested in 
\autoref{thm:BoundedCohomologicalStabilityForLinearGroups} and
\autoref{thm:BoundedCohomologicalStabilityForUnitaryGroups} stems 
from applications to the study of mapping class groups, $\pi_0(\on{Diff}(M))$, and classifying spaces of diffeomorphism groups, $\on{BDiff}_\partial(M)$, 
of high-dimensional manifolds $M$ following ideas of Galatius--Randal-Williams
\cite{GalatiusRandalWilliamsStability}, Kreck \cite{kreck1979}, Hatcher \cite{hatcher1978} and Hsiang--Sharpe \cite{hsiangsharpe1976}.
Indeed, our proof of \autoref{thm:BoundedCohomologicalStabilityForUnitaryGroups} generalizes the
``algebra'' part of \cite{GalatiusRandalWilliamsStability} to the
``norm-enriched'' setting. Combined with work of Kreck \cite{kreck1979},
\autoref{thm:BoundedCohomologicalStabilityForUnitaryGroups} immediately implies a bounded 
cohomology analogue of the main theorem of \cite{GalatiusRandalWilliamsStability} for the
``high-dimensional genus $g$ surfaces'' $W_{g,1}^{2m} \coloneqq D^{2m} \# (S^m \times S^m)^{\# g}$. 
Combined with work of Hatcher \cite{hatcher1978} and Hsiang--Sharpe \cite{hsiangsharpe1976},
\autoref{thm:BoundedCohomologicalStabilityForLinearGroups} yields a similar result for high-dimensional tori $T^m = (S^1)^{\times m}$.

To state these consequences, let $\on{Diff}(M)$ denote the diffeomorphism group of a smooth, compact, orientable manifold $M$ equipped with the $C^\infty$-topology.
If the boundary of $M$ is nonempty, we write $\on{Diff}_\partial(M) \subseteq \on{Diff}(M)$ for the subgroup of (orientation-preserving) diffeomorphisms fixing a neighbourhood of the boundary $\partial M$ pointwise.

\begin{corollaryi}
	\label{thm:stability-for-mcg-and-bdiff-of-high-dimensional-manifolds}
    Firstly, assume that $2m \geq 6$. Then there are isometric isomorphisms
	\[
	\BoundedCohomologyOfSpaceObject{q}{\apply{\on{BDiff}_\partial}{W_{g,1}^{2m}}}{\Reals} \cong 
	\BoundedCohomologyOfSpaceObject{q}{\apply{\pi_0}{\apply{\on{Diff}_\partial}{W_{g,1}^{2m}}}}{\Reals} \cong
	\BoundedCohomologyOfSpaceObject{q}{\Aut{g}{\Integers}{\Lambda}{\epsilon}}{\Reals},
	\]
	where $\Aut{g}{\Integers}{\Lambda}{\epsilon}$ is isomorphic to $\on{O}_{g,g}(\Integers)$ if $m$ is even,
	$\on{Sp}_{2g}(\Integers)$ if $m \in \{3, 7\}$, and $\on{Sp}_{2g}^0(\Integers)$, 
	the subgroup of $\on{Sp}_{2g}(\Integers)$ fixing a quadratic refinement of Arf invariant zero, otherwise.
	Therefore the inclusion maps (i.e.\ extending by the identity)
	\[
	\apply{\on{BDiff}_\partial}{W_{g,1}^{2m}} \hookrightarrow \apply{\on{BDiff}_\partial}{W_{g+1,1}^{2m}} \text{ and }
	\apply{\pi_0}{\apply{\on{Diff}_\partial}{W_{g,1}^{2m}}} \hookrightarrow \apply{\pi_0}{\apply{\on{Diff}_\partial}{W_{g+1,1}^{2m}}}
	 \]
	induce isomorphisms in the $q$-th bounded cohomology of spaces and groups, respectively, whenever $q \leq \frac{g-3}{2}$, and injections whenever $q \leq
	\frac{g-3}{2} +1$.
	
	Secondly, assume that $m \geq 6$. Then there are isometric isomorphisms
	\[
	\BoundedCohomologyOfSpaceObject{q}{\apply{\on{BDiff}}{T^m}}{\Reals} \cong 
	\BoundedCohomologyOfSpaceObject{q}{\apply{\pi_0}{\apply{\on{Diff}}{T^m}}}{\Reals} \cong
	\BoundedCohomologyOfSpaceObject{q}{\GLn{m}{\Integers}}{\Reals}.
	\]
	Therefore the inclusion maps (i.e.\ extending by the identity)
	\[
	\apply{\on{BDiff}}{T^m} \hookrightarrow \apply{\on{BDiff}}{T^{m+1}} \text{ and }
	\apply{\pi_0}{\apply{\on{Diff}}{T^m}} \hookrightarrow \apply{\pi_0}{\apply{\on{Diff}}{T^{m+1}}}
	\]
	induce isomorphisms in the $q$-th bounded cohomology of spaces and groups, respectively, if $m \geq 6$ and $q \leq \frac{m-2}{2}$, and injections if $m \geq 6$ and $q \leq \frac{m-2}{2} +1$.
\end{corollaryi}

\begin{remark}
	In the setting of \autoref{thm:stability-for-mcg-and-bdiff-of-high-dimensional-manifolds} and as explained by Krannich in the first paragraph of \cite[Section 1.3]{krannich2020},
	Kreck's work \cite[Proposition 3]{kreck1979} yields two exact sequences 
	$1 \to \on{T}_{g,1}^{m} \to \apply{\pi_0}{\apply{\on{Diff}_\partial}{W_{g,1}^{2m}}} \to \Aut{g}{\Integers}{\Lambda}{\epsilon} \to 1$
	and 
	$1 \to \Theta_{2m+1} \to \on{T}_{g,1}^{m} \to \Hom(\HomologyOfSpaceObject{m}{W_{g,1}^{2m}}{\Integers}, \on{S}\pi_m(\on{SO}(m)) \to 1,$
	where $\Theta_{2m+1}$ and $\on{S}\pi_m(\on{SO}(m))$ are abelian groups. Therefore, \cite[Theorem 3.3 and Proposition 3.4]{frigerio2017boundedcohomologyofdiscretegroups} imply
	that $\on{T}_{g,1}^{m}$ is amenable. Hence, Gromov's theorem \cite{gromov1982} (see also \cite{IvanovSimplicialMappingTheorem} and \autoref{thm:gromov-ivanov}) and the fact that $\pi_1(\apply{\on{BDiff}_\partial}{W_{g,1}}) \cong \apply{\pi_0}{\apply{\on{Diff}_\partial}{W_{g,1}}}$ yield the first set of isometric isomorphisms
	in \autoref{thm:stability-for-mcg-and-bdiff-of-high-dimensional-manifolds}. The second set follows similarly from a work of
	Hatcher \cite[Theorem 4.1]{hatcher1978} and Hsiang--Sharpe \cite[Theorem 2.5]{hsiangsharpe1976}, which establishes an exact sequence
	$1 \to A_m \to \apply{\pi_0}{\apply{\on{Diff}}{T^m}} \to \GLn{m}{\Integers} \to 1$ if $m \geq 6$ and where  $A_m$ is abelian.
\end{remark}

In on-going work, the authors use a ``uniform''
version of homotopy theoretic techniques to study the bounded cohomology of 
$\operatorname{Diff(M)}^\delta$, the discrete diffeomorphisms groups of
manifolds, building on ideas contained in \cite{NarimanStability} and \cite{GalatiusRandalWilliamsStability}.

\subsection{Limitations, instability, stable cohomology and outlook}

One might hope that all homological stability results for sequences of discrete
groups $\{\Group_{n}\}_{n \in \NaturalNumbers}$ in the literature can be
adapted to the setting of bounded cohomology using uniformly bounded methods.
This is not the case: As explained beforehand, one usually uses a sequence of
highly (boundedly) acyclic simplicial $\Group_n$-spaces
$\{\SimplicialComplex_\bullet^n\}_{n \in \NaturalNumbers}$ to prove stability
results. It is easily seen that any infinite diameter simplicial complex
has non-vanishing first simplicial bounded cohomology (see
\autoref{exm:SBCofR}), so for our approach to bear fruit one would need that
the complexes $\SimplicialComplex_{\bullet}^n$ have finite diameter.

In contrast to \autoref{thm:stability-for-mcg-and-bdiff-of-high-dimensional-manifolds}, it is known that 
the bounded cohomology of mapping class groups of 2-dimensional orientable surfaces does not stabilize:
It follows from work of Bestvina and Fujiwara \cite{BestvinaFujiwaraQuasimorphisms} that
the kernel of the comparison map
between bounded and ordinary cohomology in \autoref{eq:comparison-map-groups}
is nontrivial in cohomological degree $q = 2$. In \cite{KastenholzInstability}, one of the
authors explains how Bestvina--Fujiwara's construction can be used to show
that the second bounded cohomology of mapping class groups does not stabilize.

In particular, it follows (from \autoref{thm:stability-for-quillen-families})
that the complexes $\SimplicialComplex^n_\bullet$ that can be used to prove
classical homological
stability for mapping class groups necessarily have infinite diameter if $n$ is
sufficiently large.
The complexes appearing in \cite{randalwilliamswahl2017homologicalstabilityforautomorphismgroups},
for example, share similarities with curve and arc complexes that are
well-known to have infinite diameter (see \cite{MasurMinsky}).

If a sequence of groups  $\{\Group_{n}\}_{n \in \NaturalNumbers}$ satisfies
stability in bounded cohomology, the next natural question to ask is the
following:

\begin{question}
	How is the stable bounded cohomology $\operatorname{lim}_n \BoundedCohomologyOfSpaceObject{q}{\Group_{n}(\Ring)}{\Reals}$ related to the bounded cohomology of the colimit $\Group_{\infty} = \operatorname{colim}_n \Group_{n}$ of the sequence of groups?
\end{question}
In the setting of classical group cohomology, they are isomorphic and this fact
allows for the proof of many powerful results; in particular it leads to
computations (see e.g.\ \cite{borel1974, madsenweiss2007, szymikwahl2019}). In bounded cohomology the
relationship seems to be more complicated. One of the authors showed in
\cite{KastenholzInstability} that while the bounded cohomology of
$\operatorname{Sp}_{2n}(\Integers)$ stabilizes, there are
stable classes whose norm grows to infinity as one increases $n$. Similarly,
De la Cruz Mengual \cite{DCMThesis} proved via explicit computations that the
third bounded cohomology of $\on{Sp}_{2n}(\Integers[i])$ contains classes whose norm tends
to infinity. Hence, these classes cannot come from the infinite symplectic groups
$\operatorname{Sp}_{\infty}(\Integers)$ and $\operatorname{Sp}_{\infty}(\Integers[i])$, respectively. This illustrates that the map
$
  \BoundedCohomologyOfSpaceObject{*}{\Group_{\infty}}{\Reals}
  \to
  \operatorname{lim}_i
  \BoundedCohomologyOfSpaceObject{*}{\Group_i}{\Reals}
$
will in general fail to be surjective.
\begin{question}
	\label{qstn:Colim}
	Is the map
	\[
  	\BoundedCohomologyOfSpaceObject{*}{\Group_{\infty}}{\Reals}
  	\to
  	\operatorname{lim}_i
  	\BoundedCohomologyOfSpaceObject{*}{\Group_i}{\Reals}
	\]
	injective? If so, does its image coincide with the set of classes whose norm
	is uniformly bounded on all $\Group_{i}$ i.e. bounded independently of $i$?
\end{question}
In Section~4 of \cite{fournierfacioloehmoraschini2022} the related question
whether the colimit of a sequence of boundedly acyclic groups is boundedly
acyclic was investigated.

Since bounded cohomology comes equipped with a semi-norm, one might further ask
whether the inclusions induce semi-norm preserving isomorphisms
$\BoundedCohomologyOfSpaceObject{q}{\Group_{n}(\Ring)}{\Reals} \leftarrow
\BoundedCohomologyOfSpaceObject{q}{\Group_{n+1}(\Ring)}{\Reals}$ in the stable
range. The example of symplectic groups also shows that one can in general not
hope that this is true (see \cite{KastenholzInstability} or \cite{DCMThesis}).
Since the norm is the most important feature of bounded cohomology
answering \autoref{qstn:Colim} and finding criteria for the stabilization to
occur via isometries would strengthen the possible results enormously.

\subsection{Outline} 
In \autoref{scn:SBC}, we introduce simplicial bounded cohomology and explain
some of its elementary properties. In \autoref{scn:Setup}, we discuss the 
stability framework for bounded cohomology of De la Cruz Mengual--Hartnick \cite{HartnickDeLaCruzQuillen, HartnickDeLaCruzStability}
and provide some intuition for it. In
\autoref{scn:GLnSetup} and \autoref{sec:automorphism-groups-of-quadratic-modules},
we show how this framework can be applied to general linear groups and automorphism groups of quadratic modules 
to deduce \autoref{thm:BoundedCohomologicalStabilityForLinearGroups} and \autoref{thm:BoundedCohomologicalStabilityForUnitaryGroups}
assuming the key ingredient: two bounded acyclicity results.
We then start developing our uniformly bounded methods and work towards proving these acyclicity theorems.
In \autoref{scn:PosetsBackground}, we collect background material on
partially ordered sets, simplicial complexes, and semi-simplicial sets 
that is used in subsequent sections.
\autoref{sec:toolbox} contains the technical heart of this work: It introduces the 
notion of uniform acyclicity and establishes a toolbox of simplicial techniques that can be employed to check
that simplicial complexes are highly boundedly acyclic.
In \autoref{scn:SolomonTits} and as a warmup for the more 
involved acyclicity arguments in last two sections, 
we apply our toolbox to prove \autoref{item:solomon-tits-an} and 
\autoref{item:solomon-tits-bncn} of \autoref{thm:general-connectivity} 
using ideas contained in \cite{BestvinaPLMorseTheory} . 
Finally, \autoref{scn:BoundedAcyclicityProofs} follows an argument in 
\cite{vanderkallen1980homologystabilityforlineargroups} to prove
\autoref{item:stability-complex-gl} of \autoref{thm:general-connectivity},
finishing the proof of
\autoref{thm:BoundedCohomologicalStabilityForLinearGroups}, while
\autoref{scn:BoundedAcyclicityProofsForUnitaryGroups} follows an argument by
\cite{GalatiusRandalWilliamsStability} to prove
\autoref{item:stability-complex-sp} of \autoref{thm:general-connectivity},
finishing the proof of
\autoref{thm:BoundedCohomologicalStabilityForUnitaryGroups}.

\subsection{Acknowledgments}
We are grateful to Jens Reinhold with whom we started a reading course on
\cite{HartnickDeLaCruzStability} out of which this work eventually grew. 
It is a pleasure to thank Alexander Kupers and Ismael Sierra for a conversation leading to \autoref{thm:stability-for-mcg-and-bdiff-of-high-dimensional-manifolds}.  
We are grateful to Mladen Bestvina for helpful correspondence about \cite{BestvinaFujiwaraQuasimorphisms},
to Nicolas Monod for a clarifying exchange about \cite{MonodSemiSimple, MonodStabilization} and general feedback on a draft of this article,
as well as to Carlos De la Cruz Mengual and Tobias Hartnick for helpful comments on an early version of this paper and for clarifying a question about their previous work \cite{HartnickDeLaCruzQuillen,HartnickDeLaCruzStability}.
RJS would like to thank his PhD advisor, Nathalie Wahl, and Sam Nariman for helpful conversations while this project was developed.
\section{Simplicial Bounded Cohomology}
\label{scn:SBC}
In this section, we introduce the simplicial bounded cohomology of a
semi-simplicial set. We then briefly recall the definitions of the bounded
cohomology
of a topological space and a discrete group, and collect some of their
well-known
properties. In the final subsection, we discuss to which extend simplicial
versions of Eilenberg--Steenrod axioms hold for the bounded cohomology of
semi-simplicial sets. The observations collected thereby are the starting point
of the ``uniformly bounded simplicial methods'' which we develop in
\autoref{sec:toolbox}.

\subsection{Bounded cohomology of semi-simplicial sets}

Let $\SimplicialComplex_{\bullet}$ be a semi-simplicial set (e.g.\ an
ordered simplicial complex). The usual \introduce{real simplicial chain complex} $\ChainComplex{*}{\SimplicialComplex_{\bullet}}$ of $\SimplicialComplex_\bullet$ is defined as follows: The $q$-th chain module
\[
  \ChainComplex{q}{\SimplicialComplex_{\bullet}}
  =
  \Reals[\SimplicialComplex_{q}]
\]
is the vector space over $\Reals$ with the set of $q$-simplices
$\SimplicialComplex_q$ of $\SimplicialComplex_\bullet$ as a basis and the
$q$-th boundary map $\BoundaryChainComplex{q}$ is defined as
\[
  \BoundaryChainComplex{q} \Simplex{}
  =
  \sum_{i=0}^{q+1}
  (-1)^{i}
  \apply{\FaceMap_{i}}{\Simplex{}}
  \in \ChainComplex{q-1}{\SimplicialComplex_{\bullet}},
\]
where $\apply{\FaceMap_{i}}{\Simplex{}} \in \SimplicialComplex_{q-1}$
denotes the $i$-th face of the $q$-simplex $\sigma \in \SimplicialComplex_q$ of
the semi-simplicial set $\SimplicialComplex_{\bullet}$. The \introduce{reduced
simplicial chain complex of $\SimplicialComplex_{\bullet}$}, denoted by
$\ReducedChainComplex{*}{\SimplicialComplex_{\bullet}}$, is defined by
extending the simplicial chain complex by $\Reals$ in degree $-1$ and
defining the differential $\BoundaryChainComplex{0}$ to map every vertex to
$1$. If $\Subcomplex_\bullet \subset
\SimplicialComplex_{\bullet}$
denotes a semi-simplicial subset, then the quotient complex
$
  \ChainComplex{*}{\SimplicialComplex_{\bullet}}
  /
  \ChainComplex{*}{\Subcomplex_\bullet}
$
is denoted by
$\ChainComplex{*}{\SimplicialComplex_{\bullet},\Subcomplex_\bullet}$
and called the \introduce{relative simplicial chain complex}.

The chain modules $\ChainComplex{*}{\SimplicialComplex_{\bullet}}$
can be equipped with an $\ellone$-norm, where
\[
  \Norm{\sum_k a_k \sigma_k} = \sum_k \AbsoluteValue{a_k}
\]
and by defining the norm on
$\ReducedChainComplex{-1}{\SimplicialComplex_{\bullet}}$ to be the absolute
value on $\Reals$, one can also equip
$\ReducedChainComplex{*}{\SimplicialComplex_{\bullet}}$ with a norm.
Similarly
the relative simplicial chain complex $\ChainComplex{*}{\SimplicialComplex_{\bullet},\Subcomplex_\bullet}$ can be equipped with the quotient
norm (Since $\ChainComplex{*}{\Subcomplex_{\bullet}}$ is closed in
$\ChainComplex{*}{\SimplicialComplex_{\bullet}}$ the quotient norm is indeed a
norm). More generally, this is encapsulated in the following definition.
\begin{definition}
	\label{def:normed-chain-complex}
  A chain complex $\AbstractChainComplex{*}$, where all
  $\AbstractChainComplex{q}$ are equipped with a norm and the
  boundary maps $\BoundaryChainComplex{q}$ are all bounded maps with respect to
  these norms is called a \introduce{normed chain complex}.
\end{definition}
In this work all simplicial chain complexes are equipped with the aforementioned norms. We record two facts about these normed chain complexes:
Firstly, the $q$-th differential has norm at most $q+1$. Secondly, every chain $\Chain$ in $\ChainComplex{q}{\SimplicialComplex_{\bullet}}$ has a unique decomposition $\Chain = \Chain_{\Subcomplex_q} + \Chain_{\SimplicialComplex_q \setminus \Subcomplex_{q}}$, where $\Chain_{\Subcomplex_q} \in \ChainComplex{q}{\Subcomplex_{\bullet}}$ and $\Chain_{\SimplicialComplex_q \setminus \Subcomplex_{q}} \in \Reals[\SimplicialComplex_{q} \setminus \Subcomplex_q]$. In particular, there is an isomorphism $\ChainComplex{q}{\SimplicialComplex_{\bullet},\Subcomplex_{\bullet}} \cong \Reals[\SimplicialComplex_{q} \setminus \Subcomplex_q]$ mapping $\Chain + \ChainComplex{q}{\Subcomplex_{\bullet}}$ to $\Chain_{\SimplicialComplex_q \setminus \Subcomplex_{q}}$. Under this identification the quotient norm of a chain $\Chain + \ChainComplex{q}{\Subcomplex_{\bullet}}$ in $\ChainComplex{q}{\SimplicialComplex_{\bullet},\Subcomplex_{\bullet}}$ is exactly the $\ellone$-norm of $\Chain_{\SimplicialComplex_q \setminus \Subcomplex_{q}}$,
\[
\Norm{\Chain+ \ChainComplex{q}{\Subcomplex_{\bullet}}} = \Norm{\Chain_{\SimplicialComplex_q \setminus \Subcomplex_{q}}}_{\ellone}.
\]
This leads us to the following observation, which we shall frequently use later.
\begin{observation}
	\label{obs:norm-preserving-splitting}
	Let $(\SimplicialComplex_\bullet, \Subcomplex_\bullet)$ be a pair of semi-simplicial sets and let $q \geq 0$. Then the map
	$$\ChainComplex{q}{\SimplicialComplex_{\bullet},\Subcomplex_{\bullet}} \to \Reals[\SimplicialComplex_{q} \setminus \Subcomplex_q] \hookrightarrow \ChainComplex{q}{\SimplicialComplex_{\bullet}}: \Chain+ \ChainComplex{q}{\Subcomplex_{\bullet}} \mapsto \Chain_{\SimplicialComplex_q \setminus \Subcomplex_{q}}$$
	is a norm preserving splitting of the quotient map $\ChainComplex{q}{\SimplicialComplex_{\bullet}} \to \ChainComplex{q}{\SimplicialComplex_{\bullet},\Subcomplex_{\bullet}}$.
\end{observation}

\begin{definition}
  For a normed chain complex $\AbstractChainComplex{*}$ and a normed
  $\Reals$-module $\Module$, we
  define the \introduce{bounded cohomology of $\AbstractChainComplex{*}$ with coefficients in $\Module$},
  denoted by
  $
    \BoundedCohomologyChainComplex{*}{\AbstractChainComplex{*}}{\Module}
  $,
  to be the cohomology of the subcomplex
  $(\BoundedCoChainComplex{*}{\AbstractChainComplex{*}}, \Coboundary^*)$ of
  the cochain complex
  $(\CoChainComplex{*}{\AbstractChainComplex{*}}, \Coboundary^*)$ of
  $\AbstractChainComplex{*}$ consisting of all bounded cochains
  \[
    \BoundedCoChainComplex{q}{\AbstractChainComplex{*}}
    \coloneqq
    \SpaceOfBoundedOperators
    {\AbstractChainComplex{q}}
    {\Module}.
  \]

  In the case where $\AbstractChainComplex{*}$ is the real simplicial chain complex
  of a semi-simplicial set $\SimplicialComplex_{\bullet}$, we denote its
  bounded cohomology by
  $
    \BoundedCohomologyOfSimplicialObject
      {*}
      {\SimplicialComplex_{\bullet}}
      {\Module}
  $
  and call it the \introduce{simplicial bounded cohomology of
  $\SimplicialComplex_{\bullet}$}. Similarly, the bounded cohomology of the
  reduced simplicial chain complex of $\SimplicialComplex_{\bullet}$ is denoted
  by
  $
    \ReducedBoundedCohomologyOfSimplicialObject
      {*}
      {\SimplicialComplex_{\bullet}}
      {\Module}
  $
  and called the \introduce{reduced simplicial bounded cohomology of
  $\SimplicialComplex_{\bullet}$}. Finally we call bounded cohomology of the
  relative simplicial complex the \introduce{relative
  simplicial bounded cohomology} and denote it by
  $
    \BoundedCohomologyOfSimplicialObject
      {*}
      {\SimplicialComplex_{\bullet},\Subcomplex_\bullet}
      {\Module}
  $.
\end{definition}

We close this subsection with a simple observation.

\begin{example}
	\label{expl:finiteness-condition}
	If $\SimplicialComplex_{\bullet}$ is a semi-simplicial set such that the set of $q$-simplices $\SimplicialComplex_{q}$ is finite for all $q$, then the simplicial bounded cohomology agrees with usual simplicial cohomology.
\end{example}

Hence, simplicial bounded cohomology is of particular interest if for some $q$ the set of simplices $\SimplicialComplex_{n}$ is infinite.

\subsection{Bounded cohomology of topological spaces and discrete groups}

The notion of simplicial bounded cohomology introduced in the previous section
can be used to define the ordinary bounded cohomology of topological spaces and
discrete groups: If $X$ denotes a topological space and
$\apply{\mathrm{Sing}}{X}_\bullet$ is its singular complex, then the
\introduce{bounded cohomology of $X$}, denoted by
$\BoundedCohomologyOfSpaceObject{*}{X}{\Reals}$, is defined as the simplicial
bounded cohomology of $\apply{\mathrm{Sing}}{X}_\bullet$. If $\Group$ is a
discrete group and $X$ is a $\EMSpace{\Group}{1}$-space, then the
\introduce{bounded cohomology of $\Group$} is defined as that of the space $X$,
$\BoundedCohomologyOfSpaceObject{*}{\Group}{\Reals} \coloneqq
\BoundedCohomologyOfSpaceObject{*}{X}{\Reals}$. This is well-defined, because
Gromov \cite{gromov1982} and subsequently Ivanov \cite{ivanov2020} showed that the bounded
cohomology of a path-connected space only depends on its fundamental group:

\begin{theoremnum}[\cite{gromov1982},\cite{ivanov2020}]
	\label{thm:gromov-ivanov}
	Let $f: X \to Y$ be a continuous map between path-connected spaces that induces a surjection with amenable kernel between the fundamental groups. Then the induced map on bounded cohomology $\BoundedCohomologyOfSpaceObject{q}{f}{\Reals}: \BoundedCohomologyOfSpaceObject{q}{Y}{\Reals} \to \BoundedCohomologyOfSpaceObject{q}{X}{\Reals}$
	is an isometric isomorphism for all $q \geq 0$.
\end{theoremnum}

One can easily see that bounded cohomology of spaces is still a homotopy
invariant.
In contrast, it is well-known that bounded cohomology of spaces
does not admit a Mayer--Vietoris sequence: Combined with
\autoref{thm:gromov-ivanov} such a sequence would imply that the bounded
cohomology of $S^1 \vee S^1 = S^1 \cup S^1$ is trivial in positive degrees,
contradicting the fact that the second and third bounded cohomology of $S^1
\vee S^1$ are infinite dimensional \cite{brooks1981, grigorchuk1995,
teruhiko1997}. Hence, bounded cohomology of spaces does not satisfy excision
either. The unavailability of these important computational tools is one of the
reasons why calculating bounded cohomology of a path-connected space or
discrete group is difficult. In the next subsection, we explain why the
situation is fundamentally different in the semi-simplicial setting.

\subsection{Properties of simplicial bounded cohomology}

The main results of \cite{IvanovSimplicialMappingTheorem} shows that the simplicial bounded cohomology of semi-simplicial set $\SimplicialComplex_\bullet$, which can be promoted to a connected Kan simplicial set, is isometrically isomorphic to that of its geometric realization. This can used to give another definition of the bounded cohomology of a discrete group, as explained in the next example.

\begin{example}
	\label{expl:def-via-bar-resolution}
  Let $\ClassifyingSpace{\Group}_\bullet$ be the nerve of a discrete group
  $\Group$ viewed as a category with one object. Then
  $\ClassifyingSpace{\Group}_\bullet$ is a semi-simplicial set that can be
  promoted to a connected Kan simplicial set and therefore the simplicial
  bounded cohomology of $\ClassifyingSpace{\Group}_\bullet$ agrees with the
  bounded cohomology of its geometric realization
  $\GeometricRealization{\ClassifyingSpace{\Group}_\bullet}$. Additionally, its
  simplicial complex is the so called Bar resolution of $\Group$.

  Since $\GeometricRealization{\ClassifyingSpace{\Group}_\bullet}$ is a
  $\EMSpace{\Group}{1}$-space and invoking \autoref{thm:gromov-ivanov}, this
  means that defining the bounded cohomology of a group $\Group$ via the bar
  resolution or in terms of the singular complex of a
  $\EMSpace{\Group}{1}$-space yields isometrically isomorphic groups.
\end{example}

In the above, the Kan property is crucial: In contrast to the classical fact that the simplicial cohomology
$\CohomologyOfSpaceObject{*}{\SimplicialComplex_{\bullet}}{\Reals}$ of a semi-simplicial set
$\SimplicialComplex_{\bullet}$ agrees with the singular cohomology of its
geometric realization
$
  \CohomologyOfSpaceObject
    {*}
    {\GeometricRealization{\SimplicialComplex_{\bullet}}}
    {\Reals}
$,
the simplicial bounded cohomology of $\SimplicialComplex_\bullet$, in general, does not agree with the
bounded cohomology of its geometric realisation as a topological space.
This is illustrated in the next \autoref{exm:SBCofR}.
\begin{example}
\label{exm:SBCofR}
  Consider the following semi-simplicial set $\SimplicialComplex_{\bullet}$
  that realizes the usual simplicial structure on $\Reals$: The vertex set $\SimplicialComplex_0$ and
  edge set $\SimplicialComplex_1$ are both given by $\Integers$ and the two faces of an edge are given
  by $\apply{\FaceMap_{0}}{n} = n$ and $\apply{\FaceMap_{1}}{n} = n + 1$. For all $q\geq 2$ we set $X_{q} = \emptyset$.
  Now consider the function
  $
    \phi
    \colon
    \ChainComplex{0}{\SimplicialComplex_{\bullet}}
    \to
    \Reals
  $
  that sends $n$ to $\AbsoluteValue{n}=\apply{d}{0,n}$.
  Since the diameter of $\Reals$ is infinite, this function is unbounded and defines an unbounded cochain $\phi \in \CoChainComplex{0}{\SimplicialComplex_{\bullet}}$.
  However, its boundary $\apply{\Coboundary^{0}}{\phi}$ takes values in $\{-1,1\}$. In particular, $\apply{\Coboundary^{0}}{\phi} \in \BoundedCoChainComplex{1}{\SimplicialComplex_{\bullet}}$ is a bounded cochain and, since $(\Coboundary^{1} \circ \Coboundary^{0}) = 0$, a bounded cocycle $\apply{\Coboundary^{0}}{\phi} \in \ker \Coboundary^1$. We claim that $[\apply{\Coboundary^{0}}{\phi}] \in \BoundedCohomologyOfSimplicialObject{1}{\SimplicialComplex_{\bullet}} {\Reals}$ is a nontrivial class. To see this, observe that any bounded cochain $\eta \in \BoundedCoChainComplex{0}{\SimplicialComplex_{\bullet}}$ satisfying $\apply{\Coboundary^{0}}{\eta} = \apply{\Coboundary^{0}}{\phi}$
  has to be of the form $\eta = \phi + \psi$, where $\psi$ is in the kernel of $\Coboundary^{0}$. But the kernel of $\Coboundary^{0}$ consists only of constant functions. This implies that no bounded cochain $\eta \in \BoundedCoChainComplex{0}{\SimplicialComplex_{\bullet}}$ satisfying $\apply{\Coboundary^{0}}{\eta} = \apply{\Coboundary^{0}}{\phi}$ exists and we conclude that $[\apply{\Coboundary^{0}}{\phi}] \neq 0$ is a non-zero class in
  $
    \BoundedCohomologyOfSimplicialObject
      {1}
      {\SimplicialComplex_{\bullet}}
      {\Reals}
  $.
  On the other hand, the topological space
  $\GeometricRealization{\SimplicialComplex_{\bullet}} \cong \Reals$ is
  contractible and hence it immediately follows that
  $\BoundedCohomologyOfSpaceObject{q}{\GeometricRealization{\SimplicialComplex_{\bullet}}}{\Reals}
   = 0$ for all $q > 0$. In particular, this shows that the simplicial bounded
  cohomology of $\SimplicialComplex_{\bullet}$ does not agree with the bounded
  cohomology of its geometric realization,
  \[
    0
    =
    \BoundedCohomologyOfSpaceObject{1}{\GeometricRealization{\SimplicialComplex_{\bullet}}}{\Reals}
    \neq
    \BoundedCohomologyOfSimplicialObject
      {1}
      {\SimplicialComplex_{\bullet}}
      {\Reals}.
  \]
\end{example}
Note that in this example the only property used about the real line was its
infinite diameter as a semi-simplicial set. In particular, any
semi-simplicial set with infinite diameter has non-vanishing simplicial
bounded cohomology in degree one. In contrast, it is a well-known fact that the
first bounded cohomology of any path-connected space or discrete group is
trivial (combine e.g.\ \cite[Section
2.1]{frigerio2017boundedcohomologyofdiscretegroups} and
\autoref{thm:gromov-ivanov}). Ivanov's result
\cite{IvanovSimplicialMappingTheorem} therefore implies that the first bounded
cohomology of any connected Kan simplicial set is trivial as well. We highlight
that these have diameter at most one. In \autoref{sec:toolbox} we introduce the
uniform boundary condition of Matsumoto--Morita
\cite{matsumotomorita1985boundedcohomologyofcertaingroupsofhomeomorphisms}; it
can be seen as a generalization of such a finite diameter condition (see
\autoref{rem:0UBC}) and plays an important role in our uniformly bounded
simplicial methods.

We now discuss to which extend simplicial bounded cohomology satisfies the Eilenberg--Steenrod axioms: We start by observing that, because the simplicial chain complex of $\Reals$ is chain homotopic to that of a point, \autoref{exm:SBCofR} also shows that the homotopy invariance axiom does not to hold. Note that in this instance, one of the homotopies participating in the equivalence is unbounded. Imposing a boundedness condition, simplicial bounded cohomology satisfies the following version of homotopy invariance, which we shall frequently use.
\begin{lemma}
\label{lem:HomotopyInvariance}
  Let
  $
    \ContinuousMap
    \colon
    \SimplicialComplex_{\bullet}
    \to
    \SimplicialComplex'_{\bullet}
  $
  and
  $
    \ContinuousMapALT
    \colon
    \SimplicialComplex_{\bullet}
    \to
    \SimplicialComplex'_{\bullet}
  $
  denote simplicial maps. Suppose further that the induced maps on the
  simplicial chain complexes are homotopic through a bounded homotopy
  $
    \Homotopy
    \colon
    \ChainComplex{*}{\SimplicialComplex_{\bullet}}
    \to
    \ChainComplex{*+1}{\SimplicialComplex_{\bullet}'}
  $, 
  then $\ContinuousMap$ and $\ContinuousMapALT$ induce the same map in
  simplicial bounded cohomology.
\end{lemma}
\autoref{lem:HomotopyInvariance} is a elementary consequence of the fact that the dual of a bounded map is a bounded map, and we leave its proof to the reader.

As an illustrating example, if $\ContinuousMap$ and $\ContinuousMapALT$ become simplicially homotopic after finitely many
barycentric subdivisions, then the assumption in \autoref{lem:HomotopyInvariance} are satisfied (compare \autoref{lem:BarycentricSubdivision}).

While the homotopy axiom is not satisfied, simplicial bounded cohomology
does satisfy the Mayer-Vietoris axiom (and hence also the excision axiom):
\begin{lemma} \label{lem:MVSimp}
  Assume $\SimplicialComplex_\bullet = \Subcomplex_\bullet \cup \Subcomplex_\bullet'$ is the union of
  semi-simplicial sets. Let $I_\bullet = \Subcomplex_\bullet \cap \Subcomplex_\bullet'$ denote their
  intersection.
  Then there exists a Mayer--Vietoris sequence in simplicial bounded cohomology.
  \[
    \ldots
    \to
    \BoundedCohomologyOfSimplicialObject{*}{\SimplicialComplex_\bullet}{\Reals}
    \to
    \BoundedCohomologyOfSimplicialObject{*}{\Subcomplex_\bullet}{\Reals}
    \oplus
    \BoundedCohomologyOfSimplicialObject{*}{\Subcomplex_\bullet'}{\Reals}
    \to
    \BoundedCohomologyOfSimplicialObject{*}{I_\bullet}{\Reals}
    \to
    \ldots
  \]
\end{lemma}
\begin{proof}
  We show that there exists an exact sequence of normed chain-complexes,
  \begin{equation}
  \label{eqn:MVChainSequence}
    0
    \to
    \BoundedCoChainComplex{*}{\SimplicialComplex_\bullet}
    \to
    \BoundedCoChainComplex{*}{\Subcomplex_\bullet}
    \oplus
    \BoundedCoChainComplex{*}{\Subcomplex_\bullet'}
    \to
    \BoundedCoChainComplex{*}{I_\bullet}
    \to \\
    0.
  \end{equation}
  The injectivity on the left follows, because every simplex in
  $\SimplicialComplex_\bullet$ occurs as a simplex in $\Subcomplex_\bullet$ or $\Subcomplex_\bullet'$.
  The surjectivity on the right follows because every function on simplices of
  $I_\bullet$ can be extended by zero to a function on the $\Subcomplex_\bullet$ and
  $\Subcomplex_\bullet'$.
  The exactness in the middle follows because two functions that agree on the
  simplices of $I_\bullet$ can be glued to obtain a function on the simplices of
  $\SimplicialComplex_\bullet$.  Now we obtain the desired long exact sequence from the snake lemma.
\end{proof}
\begin{remark}
  Observe that for singular chain complexes of spaces and the usual Mayer-Vietoris sequence of a covering $\SimplicialComplex = \Subcomplex \cup \Subcomplex'$,
  the sequence in \autoref{eqn:MVChainSequence} fails to be injective on the
  left,
  since not every singular simplex of $X$ occurs as a simplex in $\Subcomplex$ or
  $\Subcomplex'$ -- rather this is only true after a potentially unbounded
  number of subdivisions. This is the main reason why the Mayer--Vietoris
  sequence (and hence the excision axiom) is not available for bounded
  cohomology of spaces and ultimately why it does not satisfy the
  Eilenberg--Steenrod axioms.
\end{remark}

The next lemma shows that simplicial bounded cohomology comes with a long exact sequence for semi-simplicial pairs $(\SimplicialComplex_\bullet, \Subcomplex_\bullet)$.
\begin{lemma}
\label{lem:LES}
  Let $\SimplicialComplex_{\bullet}$ denote a semi-simplicial set and
  $\Subcomplex_{\bullet}$ a semi-simplicial subset. Then there exists a
  natural
  long exact sequence
  \[
    \ldots
    \to
    \BoundedCohomologyOfSimplicialObject
      {*-1}
      {\Subcomplex_{\bullet}}
      {\Reals}
    \to
    \BoundedCohomologyOfSimplicialObject
      {*}
      {\SimplicialComplex_{\bullet},\Subcomplex_{\bullet}}
      {\Reals}
    \to
    \BoundedCohomologyOfSimplicialObject
      {*}
      {\SimplicialComplex_{\bullet}}
      {\Reals}
    \to
    \BoundedCohomologyOfSimplicialObject
      {*}
      {\Subcomplex_{\bullet}}
      {\Reals}
    \to
    \ldots
  \]
\end{lemma}
\begin{proof}
  We will show that the sequence
  \[
    0
    \to
    \BoundedCoChainComplex{*}{\SimplicialComplex_{\bullet},\Subcomplex_\bullet}
    \to
    \BoundedCoChainComplex{*}{\SimplicialComplex_{\bullet}}
    \to
    \BoundedCoChainComplex{*}{\Subcomplex_\bullet}
    \to
    0
  \]
  is exact, then the claim follows from the usual snake lemma.

  To check the exactness at the left and middle term, we note that any bounded
  morphism in $\BoundedCoChainComplex{q}{\SimplicialComplex_{\bullet},\Subcomplex_\bullet}$ gets mapped injectively to a bounded
  morphism with domain $\ChainComplex{q}{\SimplicialComplex_{\bullet}}$ that
  vanishes on the subcomplex $\ChainComplex{q}{\Subcomplex_{\bullet}}$ and that
  if a bounded morphism with domain $\ChainComplex{q}{\SimplicialComplex_{\bullet}}$
  vanishes on $\ChainComplex{q}{\Subcomplex_{\bullet}}$, then it descends to a
  relative bounded cochain. The surjectivity of the right most map follows because bounded functions with domain
  $\ChainComplex{q}{\Subcomplex_{\bullet}}$ can be extend to bounded functions with domain $\ChainComplex{q}{\SimplicialComplex_{\bullet}}$ by mapping every simplex in
  $\SimplicialComplex_{q} \setminus \Subcomplex_{q}$ to zero.
\end{proof}
\begin{remark}
  Since we are only interested in the underlying simplicial bounded cohomology groups and not their norms, we can indeed apply the usual snake lemma in \autoref{lem:MVSimp} and \autoref{lem:LES}. This comes at the cost of the boundary morphisms potentially not being continuous.
\end{remark}

Finally, simplicial bounded cohomology satisfies a finite
additivity axiom, i.e.\
$
  \BoundedCohomologyOfSimplicialObject
    {*}
    {
      \SimplicialComplex_{\bullet}
      \sqcup
      \SimplicialComplex'_{\bullet}
    }
    {\Reals}
  \cong
  \BoundedCohomologyOfSimplicialObject
    {*}
    {\SimplicialComplex_{\bullet}}
    {\Reals}
  \times
  \BoundedCohomologyOfSimplicialObject
    {*}
    {\SimplicialComplex'_{\bullet}}
    {\Reals}
$. 
For infinite disjoint unions the additivity property fails as our last example in this subsection shows.

\begin{example}
	\label{expl:uniformity-condition}
  Let $\SimplicialComplex_{\bullet}^{n}$ denote the semi-simplicial set with set of vertices $\SimplicialComplex^{n}_{0} = \{0,\ldots n\}$, set of edges $\SimplicialComplex^{n}_{1} = \{1, \ldots, n\}$ and $\SimplicialComplex^{n}_{q} = \emptyset$ for all $q > 1$. The face maps are given by $\apply{\FaceMap_{0}}{k}=k-1$ and $\apply{\FaceMap_{1}}{k}=k$. Then, the geometric realization of $\SimplicialComplex_{\bullet}^{n}$ is an interval of length $n$. Since all $\SimplicialComplex_{\bullet}^{n}$ are finite complexes, \autoref{expl:finiteness-condition} implies that their simplicial bounded cohomology agrees with their ordinary simplicial cohomology and hence vanishes is positive degrees. Let $\SimplicialComplex_{\bullet}$ denote the disjoint union of all $\SimplicialComplex^{n}_{\bullet}$ and consider the function $\phi\colon \SimplicialComplex_{0} \to \Reals$ mapping the vertex $k$ in $\SimplicialComplex_{0}^{n}$ to $k \in \Reals$.
  While the restriction of $\phi$ to all
  $\SimplicialComplex_{\bullet}^{n}$ is bounded, overall it is
  unbounded. Hence repeating the arguments of \autoref{exm:SBCofR}
  shows that $\Coboundary \phi$ represents a non-zero element of the
  first simplicial bounded cohomology of $\SimplicialComplex_{\bullet}$
  and hence
  \[
  0
  =
  \prod_{n}
  \BoundedCohomologyOfSimplicialObject
  {1}
  {\SimplicialComplex_{\bullet}^{n}}
  {\Reals}
  \not \cong
  \BoundedCohomologyOfSimplicialObject
  {1}
  {\SimplicialComplex_{\bullet}}
  {\Reals}
  \]
\end{example}
Note that in this example the diameter of each connected component of $\SimplicialComplex_\bullet$ is finite and that the set of diameters $\NaturalNumbers$ does not admit a common upper bound. This illustrate that in order for simplicial bounded cohomology to behave well for infinite disjoint unions one has to impose some kind of uniformity condition (see e.g.\ \autoref{lem:DirectedUnion} for a version where all occurring semi-simplicial sets have vanishing simplicial
bounded cohomology).

All in all simplicial bounded cohomology is very close to being a cohomology
theory i.e. satisfying the
Eilenberg--Steenrod axioms: It does satisfy the Mayer--Vietoris axiom, the long
exact sequence axiom and even the dimension axiom. Furthermore, versions of the
homotopy axiom and the additivity axiom hold under some additional boundedness and uniformity assumptions. In \autoref{sec:toolbox}, we make this more precise using the uniform boundary condition of Matsumoto--Morita.
\section{A framework for stability of bounded cohomology}
\label{scn:Setup}
This section is based on work of De la Cruz Mengual--Hartnick
\cite{HartnickDeLaCruzQuillen} and introduces their stability framework
for bounded cohomology. Their work is inspired by Quillen's approach to
classical homological stability \cite{quillen1971} and can be seen as
``norm-enriched'' analogue of its fruitful abstractions (see e.g.\
\cite{randalwilliamswahl2017homologicalstabilityforautomorphismgroups,
krannich2019, hepworth2020}).

\begin{remark}
	The machinery developed in \cite{HartnickDeLaCruzQuillen} can be used
	to study the continuous bounded cohomology of topological groups.
	Because the present work primarily concerns the bounded cohomology of
	discrete groups, we decided to discuss and state the results
	contained in \cite{HartnickDeLaCruzQuillen} only for this special
	case.
\end{remark}

De la Cruz Mengual--Hartnick \cite{HartnickDeLaCruzQuillen} proved that
the bounded cohomology of a nested sequence of groups $\{\Group_n\}_{n
\in \NaturalNumbers}$
\[
\Group_0 \hookrightarrow \Group_1 \hookrightarrow \dots \hookrightarrow \Group_n \hookrightarrow \dots
\]
exhibits a stability pattern if for each $n \in \NaturalNumbers$ there exists a semi-simplicial set $\SimplicialComplex^n_\bullet$ with a simplicial $\Group_n$-action such that the family $(\Group_{n}, \SimplicialComplex_{\bullet}^{n})_{n \in \NaturalNumbers}$ satisfies three technical conditions.

The first condition imposed on the family $(\Group_{n},
\SimplicialComplex_{\bullet}^{n})_{n \in \NaturalNumbers}$ is for the
semi-simplicial sets $\SimplicialComplex_{\bullet}^{n}$ to be highly acyclic in
the following sense.

\begin{definition}
  Let $k \geq 0$. A semi-simplicial set is called \introduce{boundedly $k$-acyclic} if
  \[
    \ReducedBoundedCohomologyOfSimplicialObject
      {q}{\SimplicialComplex_{\bullet}}{\Reals}
    =
    0
    \text{ for all }
    q \leq k.
  \]
\end{definition}

For each $n \in \NaturalNumbers$, the acyclicity condition is used to construct
bounded cohomology approximations to the classifying space
$\ClassifyingSpace{\Group_n}$. This is in analogy with ideas used in the
literature on classical homological stability (see e.g.\
\cite{randalwilliamswahl2017homologicalstabilityforautomorphismgroups}).
Consider the Borel construction of the $\Group_n$-action on
$\SimplicialComplex^n_\bullet$, i.e.\ the quotient of
$\SimplicialComplex^n_{\bullet} \times \EG{\Group_n}$ by the diagonal
$\Group_n$-action,
\[
  \HomotopyQuotient{\SimplicialComplex^n_\bullet}{\Group_n}
  \coloneqq
  (\SimplicialComplex^n_{\bullet}
  \times
  \EG{\Group_n})_{\Group_n}.
\]

\cite[Proposition 2.15]{HartnickDeLaCruzQuillen} shows that an
associated
double complex leads to a spectral sequence converging to the bounded
cohomology of $\Gamma_n$ in a range of degrees depending on how boundedly
acyclic the semi-simplicial set $\SimplicialComplex^n_\bullet$ is. This is
exactly the step in the stability argument that relies on bounded acyclicity
results. The following is an augmented version of this spectral sequence, which
converges to zero in a range.

\begin{lemma}[{\cite[Proposition 2.15]{HartnickDeLaCruzQuillen}}]
	\label{lem:spectralsequencefromacyclicity}
  Let $\SimplicialComplex_{\bullet}$ denote a boundedly $k$-acyclic
  semi-simplicial set with a simplicial $\Group$-action which is extended by a
  singleton in $\SimplicialComplex_{-1}$, then there exists a first quadrant
  spectral sequence
  $\Spectralsequence{\bullet}{\bullet}{\bullet}$
  with first page terms and differentials given by
  \[
    \Spectralsequence{p}{q}{1}
    =
    \BoundedCohomologyOfSpaceObject
      {q}
      {\Group}
      {\BoundedCoChainComplex{p-1}{\SimplicialComplex_{\bullet}}}
    \text{ and }
    \SSDifferential{p}{q}{1}
    =
    \Coboundary^{p-1}
    \text{ for all $p,q\geq 0$}
  \]
  that converges to zero in all total degrees up to and including $k+1$.
\end{lemma}

The stability argument requires a detailed understanding of the spectral
sequences constructed in the previous lemma. In order to describe the
$\Spectralsequence{\bullet}{\bullet}{1}$-page of these spectral sequences two
additional, technical assumptions are assumed.

The second condition imposed on the family $(\Group_{n}, \SimplicialComplex_{\bullet}^{n})_{n \in \NaturalNumbers}$ is for the simplicial action of $\Gamma_n$ on $\SimplicialComplex^n_\bullet$ to be highly transitive in the following sense.

\begin{definition}
  Let $\Group$ denote a group and $\SimplicialComplex_{\bullet}$ denote a
  semi-simplicial set with a simplicial $\Group$-action. This action is
  called \introduce{$k$-transitive} if for all $p \leq k$ the action of $\Gamma$ on the set of $p$-simplices $\SimplicialComplex_p$ of $\SimplicialComplex_{\bullet}$ is transitive.
\end{definition}

Note that a $k$-transitivity assumption allows one to identify, for each $p
\leq k$, the set of $p$-simplices $\SimplicialComplex^n_p$ of
$\SimplicialComplex^n_{\bullet}$ with the set of cosets
$\Group_n/\Stabilizer_{n, p}$, where $\Stabilizer_{n, p}$ is the stabilizer of
some $p$-simplex $\Simplex{p} \in \SimplicialComplex^n_p$. If
$\SimplicialComplex^n_\bullet$ is highly boundedly acyclic, one can therefore
invoke a bounded cohomology version of Shapiro's lemma (see
\cite[Proposition~10.1.3]{MonodBook}) to identify
the terms of the $\Spectralsequence{\bullet}{\bullet}{1}$-page of the spectral
sequences in
\autoref{lem:spectralsequencefromacyclicity} with the bounded cohomology of the
stabilizers of certain simplices in a range of bi-degrees.

The third and last condition imposed on the family $(\Group_{n}, \SimplicialComplex_{\bullet}^{n})_{n \in \NaturalNumbers}$ is for the actions associated to two subsequent pairs, $(\Group_{n}, \SimplicialComplex_{\bullet}^{n})$ and $(\Group_{n+1}, \SimplicialComplex_{\bullet}^{n+1})$, to be compatibility in the following sense.

\begin{definition}[{Compare with \cite[Section~1.2]{HartnickDeLaCruzQuillen}}]
  \label{def:Compatibility}
  Consider a nested sequence of groups
  $
    \{\Group_{n}\}_{n \in \NaturalNumbers}
  $, 
  where each $\Group_{n}$ acts on a semi-simplicial set
  $\SimplicialComplex_{\bullet}^{n}$. We say that the sequence
  $
    \left(
      \Group_{n}
      ,
      \SimplicialComplex_{\bullet}^{n}
    \right)
    _{n\in\NaturalNumbers}
  $
  is \introduce{$\CompatibilityFunction$-compatible} for a function
  $
    \CompatibilityFunction
    \colon
    \NaturalNumbers
    \to
    \NaturalNumbers
  $
  if
  \begin{enumerate}[(i)]
  \item
  \label{itm:CompatibilityTransitivity}
    The action of $\Group_{n}$ on $\SimplicialComplex_{\bullet}^{n}$ is
    $\apply{\CompatibilityFunction}{n}$-transitive.
  \item
  \label{itm:CompatibilityFlag}
    For all $n\in \NaturalNumbers$ there exists a flag
    $
      \Simplex{n,0}
      \leq
      \Simplex{n,1}
      \leq
      \ldots
      \leq
      \Simplex{n,\apply{\CompatibilityFunction}{n}}
      \subset
      \SimplicialComplex_{\bullet}^{n}
    $, 
    where $\leq$ denotes the face relation and $\Simplex{n,k}$ denotes a
    $k$-simplex.
  \item
  \label{itm:CompatibilityStabilizer}
    For each $q\leq \apply{\CompatibilityFunction}{n} - 1$ and $i\leq q+1$,
    there exists an
    element $\GroupElement_{n,q,i} \in \Group_{n}$ such that
    $
      \FaceMap_{i} \Simplex{n,q+1} = \GroupElement_{n,q,i} \Simplex{n,q}
    $
    and $\GroupElement_{n,q,i}$ centralizes the stabilizer $\Stabilizer_{n,q+1}$
    of $\Simplex{n,q+1}$.
  \item
  \label{itm:CompatibilityEpimorphism}
    For each $q \leq \apply{\CompatibilityFunction}{n}$, there exists an
    epimorphism
    $\EpimorphismSetup_{n,q} \colon \Stabilizer_{n,q+1} \to \Group_{n-q-1}$ with
    amenable kernel and a section
    $\SectionSetup_{n,q}$ such that the following diagrams commute:
    \[
      \begin{tikzcd}
        \Stabilizer_{n,q+1}
          \arrow[r, "{c_{\GroupElement_{n,q,i}}}"]
        &
        \Stabilizer_{n,q}
          \arrow[d, two heads, "\EpimorphismSetup_{n,q}"]
        \\
        \Group_{n-q-2}
          \ar[u, "\SectionSetup_{n,q+1}"]
          \ar[r, hook]
        &
        \Group_{n-q-1}
      \end{tikzcd}
    \]
    where $c_{\GroupElement_{n,q,i}}$ is the composition of the inclusion
    $
      \Stabilizer_{n,q+1}
      \hookrightarrow
      \Stabilizer_{n,q}
    $
    and conjugation by $\GroupElement_{n,q,i}$.
  \end{enumerate}
\end{definition}

\begin{remark}
	The conditions we stated in \autoref{def:Compatibility} are based on \cite[arXiv v1: Definition 5.3]{HartnickDeLaCruzStability}. There is one difference: In \autoref{itm:CompatibilityStabilizer} we ask that $g_{n,q, i}$ centralizes, and not just normalizes, $\Stabilizer_{n,q+1}$. Apart from \autoref{itm:CompatibilityEpimorphism}, which is slightly relaxed, the conditions are then exactly as the ones used in classical setting of e.g.\ \cite{randalwilliamswahl2017homologicalstabilityforautomorphismgroups}. Recently, the new preprint \cite{HartnickDeLaCruzQuillen} appeared and it contains several new variants of the compatibility axiom that are weaker than \autoref{def:Compatibility}, see \cite[Definition 4.2, (MQ3)]{HartnickDeLaCruzQuillen}. We highlight that \autoref{def:Compatibility} implies \cite[Definition 4.2, (MQ2) and (MQ3)]{HartnickDeLaCruzQuillen}: Indeed, \cite[Definition 4.2, (MQ2)]{HartnickDeLaCruzQuillen} follows from \autoref{itm:CompatibilityTransitivity}. And for \cite[Definition 4.2, (MQ3)]{HartnickDeLaCruzQuillen}, we note that \autoref{itm:CompatibilityStabilizer} implies that the map $c_{\GroupElement_{n,q,i}}: \Stabilizer_{n,q+1} \to \Stabilizer_{n,q}$ is equal to the inclusion $\Stabilizer_{n,q+1} \hookrightarrow \Stabilizer_{n,q}$. Therefore, \cite[Section 1.2, (MQ3c)]{HartnickDeLaCruzQuillen} holds as well.
\end{remark}

This compatibility condition enables one to compute the differentials of the
spectral sequences in \autoref{lem:spectralsequencefromacyclicity} in a range
of bi-degrees. Assuming a high boundedly acyclicity and high transitivity
condition, the compatibility assumption allows to identify, in a range of
degrees, the bounded cohomology of the stabilizer of a $p$-simplex
$\Stabilizer_{n,p}$, i.e.\ certain terms on the
$\Spectralsequence{\bullet}{\bullet}{1}$-page
of the spectral sequences in \autoref{lem:spectralsequencefromacyclicity}, with
the bounded cohomology of the group $\Group_{n - p - 1}$. Under these
identification the $\Spectralsequence{\bullet}{\bullet}{1}$-differential of the
spectral
sequence then correspond to either the map induced by the inclusion $\Group_{n
- p - 2} \hookrightarrow \Group_{n - p - 1}$ or the zero map.

The three conditions are encapsulated by the following definition.

\begin{definition}[{\cite[Definition 4.2]{HartnickDeLaCruzQuillen}}]
  \label{def:QuillenFamily}
  Let $\{\Group_{n}\}_{n \in \NaturalNumbers}$ denote a nested
  sequence of groups, and for every $n\in\NaturalNumbers$, let
  $\SimplicialComplex_{\bullet}^{n}$ denote a semi-simplicial set on which
  $\Group_{n}$ acts simplicially. Then the sequence
  $
    \left(
      \Group_{n}
      ,
      \SimplicialComplex_{\bullet}^{n}
    \right)
    _{n \in \NaturalNumbers}
  $
  is called a \introduce{Quillen family with parameters $(\AcyclicityFunction,
  \CompatibilityFunction)$} for functions
  $
    \AcyclicityFunction
    \colon
    \NaturalNumbers
    \to
    \NaturalNumbers
    \cup
    \{-\infty,\infty\}
  $
  and
  $
    \CompatibilityFunction
    \colon
    \NaturalNumbers
    \to
    \NaturalNumbers
  $, 
  if:
  \begin{enumerate}[(i)]
  \item
    $\SimplicialComplex_{\bullet}^{n}$ is boundedly
    $\apply{\AcyclicityFunction}{n}$-acyclic;
  \item
    $\Group_{n}$ acts $\apply{\CompatibilityFunction}{n}$-transitively on
    $\SimplicialComplex_{\bullet}^{n}$;
  \item
    The sequence
    $
      \left(
        \Group_{n}
        ,
        \SimplicialComplex_{\bullet}^{n}
      \right)
      _{n \in \NaturalNumbers}
    $
    is $\apply{\CompatibilityFunction}{n}$-compatible.
  \end{enumerate}
\end{definition}

An induction argument using the spectral sequences constructed in \autoref{lem:spectralsequencefromacyclicity} and the identifications outlined above yields the following general stability result for the bounded cohomology of sequences of groups.

\begin{theoremnum}[{\cite[Theorem 4.6]{HartnickDeLaCruzStability}}]
  \label{thm:stability-for-quillen-families}
  Assume that
  $(\Group_{n},\SimplicialComplex_{\bullet}^{n})_{n\in\NaturalNumbers}$ is a
  Quillen family with parameters
  $(\AcyclicityFunction,\CompatibilityFunction)$.
  Let also $q_{0} \in \NaturalNumbers_{>0}$ be such that for every
  $n\in\NaturalNumbers$ the map
  $
    \BoundedCohomologyOfSpaceObject{q}{\Group_{n+1}}{\Reals}
    \to
    \BoundedCohomologyOfSpaceObject{q}{\Group_{n}}{\Reals}
  $
  is an isomorphism for all $q\leq q_0$. Finally let us define for $q\geq 0$
  and $n\in \NaturalNumbers$ the quantities
  \[
    \apply{\widetilde{\AcyclicityFunction}}{q,n}
    \coloneqq
    \min_{j=q_{0}}^{q}
    \left\{
      \apply{\AcyclicityFunction}{n+1-2(q-j)} - j
    \right\}
    \text{ and }
    \apply{\widetilde{\CompatibilityFunction}}{q,n}
    \coloneqq
    \min_{j=q_{0}}^{q}
    \left\{
    \apply{\CompatibilityFunction}{n+1-2(q-j)} - j
    \right\}
  \]
  Then, for all such $n$ and $q$ the inclusions induce isomorphisms and
  injections respectively
  \[
    \BoundedCohomologyOfSpaceObject{q}{\Group_{n+1}}{\Reals}
    \xrightarrow{\cong}
    \BoundedCohomologyOfSpaceObject{q}{\Group_{n}}{\Reals}
    \text{ and }
    \BoundedCohomologyOfSpaceObject{q+1}{\Group_{n+1}}{\Reals}
    \xhookrightarrow{}
    \BoundedCohomologyOfSpaceObject{q+1}{\Group_{n}}{\Reals}
  \]
  whenever
  $
    \min
    \left\{
      \apply{\widetilde{\AcyclicityFunction}}{q,n}
      ,
      \apply{\widetilde{\CompatibilityFunction}}{q,n} - 1
    \right\}
    \geq 0
  $.
\end{theoremnum}
\begin{remark}
\label{rem:CountabilityQuillen}
  As stated \cite[Theorem 4.6]{HartnickDeLaCruzQuillen} only applies to countable discrete groups (see e.g.\ \cite[Theorem C]{HartnickDeLaCruzQuillen});
  not all discrete groups as we claim in \autoref{thm:stability-for-quillen-families}.
  The reason is that \cite{HartnickDeLaCruzQuillen} primarily studies the continuous bounded cohomology of
  locally compact second countable \emph{topological} groups.
  We now explain how the countability assumption can be removed if one works with discrete groups (i.e.\ why \autoref{thm:stability-for-quillen-families} holds):
  
  The proof of \cite[Theorem B]{HartnickDeLaCruzQuillen} relies on three results
  that are only stated for locally compact second countable topological groups in \cite{HartnickDeLaCruzQuillen}
  and also hold for arbitrary discrete groups.
  
  Firstly, the proof of \cite[Proposition 2.15]{HartnickDeLaCruzQuillen} uses \cite[Lemma 2.4]{HartnickDeLaCruzQuillen} which asserts that a certain functor is exact.
  The analogue of \cite[Lemma 2.4]{HartnickDeLaCruzQuillen} holds for all discrete groups (without further assumptions) as e.g.\ explained in the proof of \cite[Theorem~3.3]{NarimanMonod},
  which constructs and uses the same spectral sequence as in \cite[Proposition 2.15]{HartnickDeLaCruzQuillen} in the discrete setting.

  Secondly, the proof of \cite[Proposition~3.6]{HartnickDeLaCruzQuillen} uses a bounded version of Shapiro's Lemma due to Monod \cite[Proposition 10.1.3]{MonodBook}. 
  Again, the analogue of this lemma holds for all discrete groups (without further assumptions) as e.g.\ explained in the claim contained in the proof of \cite[Lemma~3.5]{NarimanMonod}
  or as used in the proof of \cite[Proposition~6.1]{fournierfacioloehmoraschini2022}.

  Thirdly, the proof of \cite[Lemma~3.4]{HartnickDeLaCruzQuillen} uses \cite[Proposition~2.9]{HartnickDeLaCruzQuillen}.
  The analogue of this result for all discrete groups (without further assumptions) is \cite[Theorem~4.23]{frigerio2017boundedcohomologyofdiscretegroups}.
\end{remark}

We close this section by remarking that spectral sequences that are similar to the one in \autoref{lem:spectralsequencefromacyclicity} can also be employed to calculate bounded cohomology groups: Recently, this approach has been successfully used in work of Monod--Nariman \cite{NarimanMonod} to described the entire bounded cohomology of certain homeomorphism and diffeomorphism groups. The uniformly bounded simplicial tools that we present in \autoref{sec:toolbox} are also applicable in this context and hence might be useful for similar computations in the future.
\section{Bounded cohomological stability for general linear groups}
\label{scn:GLnSetup}
In this section we construct a Quillen family in the sense of
\autoref{def:QuillenFamily} for the sequence of general linear groups
$\{\GLn{n}{\Ring}\}_{n \in \NaturalNumbers}$ of any unital ring $\Ring$ and
apply the framework for stability of bounded cohomology developed by
De la Cruz Mengual--Hartnick \cite{HartnickDeLaCruzQuillen} to prove a
generic stability result for the bounded cohomology of $\{\GLn{n}{\Ring}\}_{n
\in \NaturalNumbers}$ with a range that depends on the Bass stable rank of
$\Ring$. This theorem is the bounded cohomology analogue of a classical
homological stability result due to van der Kallen
\cite{vanderkallen1980homologystabilityforlineargroups}. The proof of the high
bounded acyclicity condition is the most difficult part of the argument and is
presented in \autoref{scn:BoundedAcyclicityProofs}. The following discussion is
parallel to and based on
\cite{vanderkallen1980homologystabilityforlineargroups}, \cite[Section
5.3]{randalwilliamswahl2017homologicalstabilityforautomorphismgroups} and
\cite[Section
2]{friedrich2016homologicalstabilityofautomorphismgroupsofquadraticmodulesandmanifolds}.
 Throughout our investigation of general linear groups, we use the following
convention.

\begin{convention}
	\label{con:general-linear-groups}
	$\Ring$ denotes a unital ring. We fix the standard basis $\{e_k : k \in \NaturalNumbers\}$ of $\Ring^\infty$ and view $\Ring^n \subset \Ring^\infty$ as the submodule spanned by the first $n$ basis vectors. Using its basis we may identify elements in $\Ring^n$ with $n \times 1$-matrices and $\GLn{n}{\Ring} = \apply{\on{Aut}_{\Ring}}{\Ring^n}$ with the group of invertible $n \times n$-matrices with entries in $\Ring$.
\end{convention}

\subsection{Bass' stable rank and Warfield's cancellation theorem}

We start by introducing the Bass' stable rank following \cite[Chapter V, Definition 3.1]{bass1968} and \cite{vasershtein1971}.

\begin{definition}
	\label{def:stable-rank}
	Let $\Ring$ be a unital ring.
	\begin{enumerate}
		\item A sequence of vectors $(v_1, \dots, v_n)$ in a $\Ring$-module $\Module$ is called \introduce{unimodular} if $(v_1, \dots, v_n)$ is an ordered basis of a free direct summand of $\Module$. In particular, a vector $v \in \Ring^l$ is unimodular if $\{v\}$ is the basis of a free summand of rank 1 in $\Ring^l$.
		\item We say that $\Ring$ satisfies Bass' \introduce{stable range condition} $(SR_{l})$ if for every $l' \geq l$ and every unimodular vector $\begin{bmatrix} r_1 \\ \vdots \\ r_{l'} \end{bmatrix} \in \Ring^{l'}$ there exist $t_1 ,\dots , t_{l'-1} \in \Ring$ such that $\begin{bmatrix} r_1 + t_1 r_{l'} \\ \vdots \\ r_{l'-1} + t_{l'-1} r_{l'} \end{bmatrix} \in \Ring^{l'-1}$ is unimodular as well.
		\item We say that the ring $\Ring$ has \introduce{stable rank} $l$, $\StableRank(\Ring) = l$, if $l \geq 1$ is the smallest number such that $\Ring$ satisfies the stable range condition $(SR_{l+1})$. We define $\StableRank(R) \coloneqq \infty$ if no such $l \in \NaturalNumbers$ exists.
	\end{enumerate}
\end{definition}

The next general cancellation theorem due to Warfield \cite{warfield1980cancellationofmodulesandgroupsandstablerangeofendomorphismrings} depends on the Bass' stable rank and is used in several subsequent arguments.

\begin{theoremnum}[Warfield's Cancellation Theorem]
	\label{lem:StableRankProperties}
	Let $\Ring$ be a unital ring of finite stable rank $\StableRank(\Ring)$. Consider two $\Ring$-modules $P$ and $Q$. If $P$ contains a direct summand isomorphic to	 $\Ring^{\StableRank(R)}$, then
	$$\Ring \oplus P \cong \Ring \oplus Q \Longrightarrow P \cong Q.$$
\end{theoremnum}
\begin{proof}
	\cite[Theorem 1.6]{warfield1980cancellationofmodulesandgroupsandstablerangeofendomorphismrings} implies that $\Ring$ satisfies the $\StableRank(\Ring)$-substitution property \cite[Definition 1.1]{warfield1980cancellationofmodulesandgroupsandstablerangeofendomorphismrings}, hence \cite[Theorem 1.2]{warfield1980cancellationofmodulesandgroupsandstablerangeofendomorphismrings} yields the claim.
\end{proof}

\begin{remark}
	The cancellation theorem stated here is the special case $M = \Ring$ of \cite[(8.12) Warfield’s Cancellation Theorem]{lam1999basssworkinringtheoryandprojectivemodules}.
\end{remark}

\subsection{A Quillen family for general linear groups}

We now start the construction of the Quillen family $(\GLn{n}{\Ring},
\SimplicialComplex^n_\bullet)_{n \in \NaturalNumbers}$ for the sequence of
groups $\{\GLn{n}{\Ring}\}_{n \in \NaturalNumbers}$ where $\Ring$ is an unital
ring. We start by introducing the semi-simplicial set
$\SimplicialComplex^n_\bullet$ on which $\GLn{n}{\Ring}$ acts simplicially.
\begin{definition}
	\label{def:ComplexOfSplitInjections}
	Let $\Ring$ be a unital ring and $n \in \NaturalNumbers$. The complex of
	$\Ring$-split injections $\SimplicialComplex^n_\bullet$ is the
	semi-simplicial set whose set of $p$-simplices is the set of the split injective $\Ring$-module homomorphisms $$\Ring^{p+1} \to \Ring^n$$ and whose $i$-th face map is given by precomposition with the inclusion $$\Ring^{p} = \Ring^{i} \oplus 0 \oplus \Ring^{(p+1)-(i+1)} \to	\Ring^{p+1}.$$
\end{definition}

\begin{remark}
	The complexes of $\Ring$-split injections $\SimplicialComplex^n_\bullet$
	appear in work of Randal-Williams--Wahl \cite[Proof of Lemma
	5.10]{randalwilliamswahl2017homologicalstabilityforautomorphismgroups} and in
	work by
	Friedrich \cite[Proof of Lemma
	2.8]{friedrich2016homologicalstabilityofautomorphismgroupsofquadraticmodulesandmanifolds} on classical homological stability for general linear groups. The complexes $\SimplicialComplex^n_\bullet$ are closely related but not equal to the complex $W^n_\bullet$ used in their homological stability arguments.
\end{remark}

\begin{remark}
	\label{rem:MaximalLengthOfUnimodularSequences}
	Let $\Ring$ be a unital ring of finite stable rank $\StableRank(\Ring)$ and let $n \geq 1$.
	Warfield's Cancellation theorem implies that the complex of $\Ring$-split
	injections $\SimplicialComplex_\bullet^n$ has dimension strictly less than
	$n+\StableRank(\Ring)$: Indeed assume there is a split injection
	$\Ring^{n+\StableRank(\Ring)+k} \to \Ring^n$ for $k > 0$. Then
	$\Ring^{n+\StableRank(\Ring)+k} \oplus P \cong \Ring^n$ for some complement
	$P$. Equivalently $\Ring \oplus (\Ring^{n-1} \oplus
	\Ring^{\StableRank(\Ring)+k} \oplus P) \cong \Ring \oplus \Ring^{n-1}$, hence
	the cancellation theorem implies that  $\Ring^{n-1} \oplus
	\Ring^{\StableRank(\Ring)+k} \oplus P \cong \Ring^{n-1}$. Applying
	cancellation $(n-1)$ times in a similar manner, implies that
	$\Ring^{\StableRank(\Ring)+k} \oplus P \cong 0$ which is impossible since $k
	> 0$.
\end{remark}

The first goal of this section is to prove that $(\GLn{n}{\Ring}, \SimplicialComplex^n_\bullet)$ is indeed a Quillen family.

\begin{proposition}
	\label{prop:AQuillenFamilyForLinearGroups}
	Let $\Ring$ be an unital ring and let $\SimplicialComplex^n_\bullet$ denote the complex of $\Ring$-split injections introduced in \autoref{def:ComplexOfSplitInjections}. Then $(\GLn{n}{\Ring}, \SimplicialComplex^n_\bullet)_{n \in \NaturalNumbers}$ is a Quillen family with parameters $$\gamma(n) = \tau(n) = n - \StableRank(R) - 1.$$
\end{proposition}

This proposition is a consequence of the next bounded acyclicity theorem and several lemmas.

\begin{theoremnum}
	\label{thm:BoundedAcyclicityResultForComplexOfSplitInjections}
	Let $\Ring$ be an unital ring and $n \in \NaturalNumbers$. The complex of
	$\Ring$-split injections $\SimplicialComplex^n_\bullet$ is boundedly
	$\gamma(n)$-acyclic for $\gamma(n) = n - \StableRank(R) - 1$.
\end{theoremnum}

The proof of this result relies on the uniformly bounded simplicial toolbox
developed in \autoref{sec:toolbox} and can be found in \autoref{scn:BoundedAcyclicityProofs} (see
\autoref{cor:complex-of-split-injections}).

\begin{lemma}
\label{lem:Transitivy}
	Let $\Ring$ be an unital ring and $n \in \NaturalNumbers$. The action of
	$\GLn{n}{\Ring}$ on the complex of $\Ring$-split injections
	$\SimplicialComplex^n_\bullet$ is
	 $\tau(n)$-transitive for $\tau(n) = n - \StableRank(R) - 1$.
\end{lemma}

\begin{proof}
	We may assume that $\StableRank(R) < \infty$, since otherwise the claim is empty. Let $p \geq 0$. Consider the p-simplex $\sigma_{n,p}$ given by the inclusion $$\Ring^{p+1} \hookrightarrow \Ring^{p+1} \oplus \Ring^{n-p-1} = \Ring^n$$ and let $f: \Ring^{p+1} \to \Ring^n$ be some other p-simplex of $\SimplicialComplex^n_\bullet$. We will use the cancellation theorem (see \autoref{lem:StableRankProperties}) to check that there is an element $M \in \GLn{n}{\Ring}$ satisfying $\sigma_{n,p} = M \circ f$ whenever $p \leq n - \StableRank(\Ring) - 1 = \tau(n)$:	Since the simplex $f$ is a split injection, it holds that
	$f: \Ring^{p+1} \hookrightarrow \im(f) \oplus Q \cong \Ring^n$ where $\im(f) \cong \Ring^{p+1}$ via a $\Ring$-linear isomorphism $\phi$ satisfying $\phi \circ f = id_{\Ring^{p+1}}$. Notice that $p \leq \tau(n)$ is equivalent to $\StableRank(R) \leq n - p - 1$. Since $\Ring^{p+1} \oplus \Ring^{n - p - 1} = \Ring^n \cong \im(f) \oplus Q$, Warfield's cancellation theorem therefore implies that $\Ring^{n-p-1} \cong Q$ via some $\Ring$-linear isomorphism $\psi$. Hence the block matrix $M = \phi \oplus \psi$ has the desired property.
\end{proof}

\begin{lemma}
  Let $\Ring$ be an unital ring. The sequence of pairs
  $
    \left(
      \GLn{n}{\Ring}
      ,
      \SimplicialComplex^n_\bullet
    \right)_{n\in \NaturalNumbers}
  $
  is $\apply{\tau}{n}$-compatible for $\apply{\tau}{n} = n - \StableRank(R) - 1$.
\end{lemma}
\begin{proof}
  By \autoref{lem:Transitivy}
  $
    \left(
      \GLn{n}{\Ring}
      ,
      \SimplicialComplex^n_\bullet
    \right)_{n\in \NaturalNumbers}
  $
  is $\apply{\tau}{n}$-transitive and hence
  \autoref{itm:CompatibilityTransitivity} of \autoref{def:Compatibility}
  is satisfied. For the flag in \autoref{itm:CompatibilityFlag} of \autoref{def:Compatibility} we choose, as
  in the previous proof, $\Simplex{n,p}$ to be the inclusion of the span of the first $p+1$ standard basis vectors of $\Ring^{n}$,
  $$\Simplex{n,p}: \Ring^{p+1} \hookrightarrow \Ring^{p+1} \oplus \Ring^{n-(p+1)} = \Ring^n.$$

  Now for \autoref{itm:CompatibilityStabilizer} of \autoref{def:Compatibility}, we note that the $i$-th face
  $\FaceMap_{i} \Simplex{n,p+1}$ of $\Simplex{n,p+1}$ corresponds to the inclusion
  \[
  \FaceMap_{i} \Simplex{n,p+1}: \Ring^{p+1} = \Ring^{i} \oplus 0 \oplus \Ring^{(p+2) - (i+1)} \hookrightarrow \Ring^{p+2} \hookrightarrow \Ring^{p+2} \oplus \Ring^{n-(p+2)} = \Ring^n.
  \]
  Let $\GroupElement_{n,p,i} \in \GLn{n}{\Ring}$ be the ``swap'' automorphism
  \[
  \GroupElement_{n,p,i}: \Ring^n = \Ring^{i} \oplus \Ring^{(p+1) - i} \oplus \Ring \oplus \Ring^{n-(p+2)} \to \Ring^{i} \oplus \Ring \oplus \Ring^{(p+1) - i} \oplus \Ring^{n-(p+2)} = \Ring^n
  \]
  that maps $(i+1)$-st standard basis vector to the $(i+2)$-nd, the $(i+3)$-rd to the $(i+4)$-th, $\dots$, the $(p+1)$-st to the $(p+2)$-nd, the $(p+2)$-nd to the $(i+1)$-st and keeps all other standard basis vectors fixed. This automorphism has the property that
  \[
    \apply{\GroupElement_{n,p,i}}{\Simplex{n,p}} = \FaceMap_{i} \Simplex{n,p+1}.
  \]
  The stabilizer $\Stabilizer_{n,p+1}$ of $\Simplex{n,p+1}$ consists of those
  matrices whose first $p+2$ columns agree with the first $p+2$ columns of the
  identity matrix. Evidently these are centralized by $\GroupElement_{n,p,i}$, which is supported on the first $p+2$ basis vectors of $\Ring^n$. This shows that \autoref{itm:CompatibilityStabilizer} of \autoref{def:Compatibility} holds.

  Since the stabilizer of $\Simplex{n,p+1}$ is given by block matrices where the upper
  left $(p + 2)\times(p + 2)$-block is given by the identity and the lower left
  block all zeros, there is a surjective projection map to
  the lower right block. This projection assigns to an element in
  $\Stabilizer_{n,p+1}$ a unique element of $\GLn{n-(p+2)}{\Ring}$
  and admits a section $\SectionSetup_{n,p+1}$ by choosing the
  upper right block to be all zeros.
  The diagram
  \[
    \begin{tikzcd}
      \Stabilizer_{n,p+1}
        \arrow[r, "{c_{\GroupElement_{n,p,i}}}"]
      &
      \Stabilizer_{n,p}
        \arrow[d, two heads, "\EpimorphismSetup_{n,p}"]
      \\
      \GLn{n-(p+2)}{\Ring}
        \ar[u, "\SectionSetup_{n,p+1}"]
        \ar[r, hook]
      &
      \GLn{n-(p+1)}{\Ring}
    \end{tikzcd}
  \]
  commutes: the elements in
  $\apply{\SectionSetup_{n,p+1}}{\GLn{n-(p+2)}{\Ring}}$ in the image of $\SectionSetup_{n,p+1}$ commute with the element
  $\GroupElement_{n,p,i}$, because these elements fix the first $p+2$ standard basis vectors and $\GroupElement_{n,p,i}$ is supported on the first $p+2$ standard basis vectors.
  Lastly the kernel of $\EpimorphismSetup_{n,p}$
  is given by a subgroup of the group of upper triangular matrices with $1$ on
  the diagonal. An induction argument using \cite[Proposition 3.4
  (2)]{frigerio2017boundedcohomologyofdiscretegroups} shows that the group of
  unitriangular
  matrices is amenable. Therefore the kernel of $\EpimorphismSetup_{n,p}$ is amenable as a subgroup of
  an amenable group \cite[Proposition 3.4
  (1)]{frigerio2017boundedcohomologyofdiscretegroups}. This means that
  \autoref{itm:CompatibilityEpimorphism} of \autoref{def:Compatibility} holds
  as well.
\end{proof}

\subsection{Proof of \autoref{thm:BoundedCohomologicalStabilityForLinearGroups}}

We will now apply \autoref{thm:stability-for-quillen-families} to prove \autoref{thm:BoundedCohomologicalStabilityForLinearGroups}.

\begin{proof}[Proof of \autoref{thm:BoundedCohomologicalStabilityForLinearGroups}]
	According to \autoref{prop:AQuillenFamilyForLinearGroups}, the
	sequence of pairs
	$(\GLn{n}{\Ring}, \SimplicialComplex_{\bullet}^n)_{n\in\NaturalNumbers}$ is a
	Quillen family with parameters $\gamma(n) = \tau(n) = n - \StableRank(\Ring)
	- 1$. Therefore, we can invoke \autoref{thm:stability-for-quillen-families}
	using $q_0 = 1$. Indeed, $q_0 =  1$ is a valid choice: If $q = 0$ or $q = 1$,
	then
	$\BoundedCohomologyOfSpaceObject{q}{\Group_{n+1}}{\Reals}
	\xrightarrow{\cong}
	\BoundedCohomologyOfSpaceObject{q}{\Group_{n}}{\Reals}$ is an isomorphism for \emph{any} sequence of groups $\{\Group_n\}_{n \in \NaturalNumbers}$. For $q = 0$ this is trivial, and for $q = 1$ it follows from the fact that $\BoundedCohomologyOfSpaceObject{1}{\Group}{\Reals} = 0$ for any group $\Group$ \cite[Chapter 2.1]{frigerio2017boundedcohomologyofdiscretegroups}. We are hence left with checking that
	\begin{equation*}
	\min
	\left\{
	\apply{\widetilde{\AcyclicityFunction}}{q,n}
	,
	\apply{\widetilde{\CompatibilityFunction}}{q,n} - 1
	\right\}
	\geq 0
	\text{ is equivalent to }
	\frac{n-\StableRank(\Ring)}{2} \geq q.
	\end{equation*}
	For this we note that $\apply{\widetilde{\AcyclicityFunction}}{q,n} = \apply{\widetilde{\CompatibilityFunction}}{q,n}$ and compute:
	\begin{align*}
	\apply{\widetilde{\AcyclicityFunction}}{q,n} &=
	\min_{j=q_{0}}^{q} \left\{ \apply{\AcyclicityFunction}{n+1-2(q-j)} - j \right\} =
	\min_{j=1}^{q} \left\{ (n+1-2(q-j) - \StableRank(\Ring) - 1) - j \right\}\\
	&= \min_{j=1}^{q} \left\{ n-2q-\StableRank(\Ring) + j \right\} =
	n-2q-\StableRank(\Ring)+1
	\end{align*}
	Hence,
	\begin{align*}
	\min \left\{ \apply{\widetilde{\AcyclicityFunction}}{q,n},
	\apply{\widetilde{\CompatibilityFunction}}{q,n} - 1	\right\} \geq 0
	&\Longleftrightarrow \apply{\widetilde{\CompatibilityFunction}}{q,n} - 1 = (n-2q-\StableRank(\Ring)+1)-1 \geq 0\\
	&\Longleftrightarrow \frac{n-\StableRank(\Ring)}{2} \geq q \qedhere
	\end{align*}
\end{proof}
\section{Bounded cohomological stability for automorphism groups of quadratic
modules}
\label{sec:automorphism-groups-of-quadratic-modules}
In this section we construct a Quillen family for automorphism groups of
quadratic modules
$\{ \Aut{n}{\Ring}{\Lambda}{\epsilon} \}_{n \in \NaturalNumbers}$ of the
quadratic modules $\HyperbolicModule^{\oplus n}$, where $\HyperbolicModule$ is
the hyperbolic module with form parameter $\FormParameter$ over $\Ring =
\Integers$ or $\Ring = \Field$ any field of characteristic zero. We apply the framework for stability
of bounded cohomology developed by De la Cruz Mengual--Hartnick
\cite{HartnickDeLaCruzQuillen} to prove a slope-$1/2$ stability result for
the bounded cohomology of $\{ \Aut{n}{\Ring}{\Lambda}{\epsilon} \}_{n \in
\NaturalNumbers}$. This theorem is the bounded cohomology analogue of classical
homological stability results due to Charney \cite{CharneyStability} and
Mirzaii--van der Kallen
\cite{mirzaiivanderkallen2002}. The proof of
the highly boundedly acyclicity condition is the most difficult part of the
argument and is presented in
\autoref{scn:BoundedAcyclicityProofsForUnitaryGroups}.
The following discussion is parallel to and based on \cite[§3.\ Algebra]{GalatiusRandalWilliamsStability}, \cite[Section 5.4]{randalwilliamswahl2017homologicalstabilityforautomorphismgroups} and \cite[Section 3]{friedrich2016homologicalstabilityofautomorphismgroupsofquadraticmodulesandmanifolds}. Throughout our investigation of automorphism groups of $\FormParameter$-quadratic modules, we use the following convention.

\begin{convention}
	\label{con:quadratic-module-setting}
	The ring $\Ring$ is equal to the integers $\Integers$ or a field of characteristic zero.
	The setup in e.g.\ \cite[Section 5.4]{randalwilliamswahl2017homologicalstabilityforautomorphismgroups} and  \cite[Section 3]{friedrich2016homologicalstabilityofautomorphismgroupsofquadraticmodulesandmanifolds} is much more general than the one in this article and e.g.\ considers rings
	$\Ring$ together with the choice of an anti-involution $\bar{\cdot}: \Ring
	\to \Ring$, exactly as in \cite{bak1969, bak1981} and \cite[Chapter 5]{hahnomeara1989}. In the present work, we implicitly equipped all rings $\Ring$ with the trivial anti-involution $\bar{\cdot} = \Identity_{\Ring}$. The following discussion is specific to this choice.
\end{convention}

\subsection{Quadratic modules and the Witt index}

Quadratic modules over $\Ring$ were first defined by Bak \cite{bak1969, bak1981} and are e.g.\ discussed in \cite[Chapter 5]{hahnomeara1989}.

\begin{definition}
A pair $\FormParameter$ consisting of the choice of a sign $\epsilon \in \{+1, -1\}$ and an additive subgroup $\Lambda \subseteq \Ring$ such that
\[
	\{a - \epsilon a : a \in \Ring\} \subseteq \Lambda \subseteq \{a \in
	\Ring : a + \epsilon a = 0\}.
\]
is called a \introduce{form parameter} for $\Ring$.
\end{definition}

\begin{example}
	For $\Ring = \Integers$, the three possible form parameters are
	$(+1,\{0\}), (-1, 2\Integers)$ and $(-1, \Integers)$. For a field $\Ring =
	\Field$ of characteristic zero,  the two possible form parameters
	are $(+1,\{0\})$ and $(-1, \Field)$.
\end{example}

A \introduce{$\FormParameter$-quadratic module} is a triple $\QuadraticModule = (\Module, \BilinearForm, \QuadraticFrom)$ consisting of a $\Ring$-module $\Module$, a bilinear form $\BilinearForm : \Module \otimes \Module \to \Ring$ satisfying $\BilinearForm(x,y) = \epsilon \BilinearForm(y,x)$, and function $\QuadraticFrom: \Module \to \Ring/\Lambda$ such that
\begin{itemize}
	\item $\QuadraticFrom(a \cdot x) = a^2 \QuadraticFrom(x)$ for all $x \in \Ring$;
	\item $\QuadraticFrom(x+y) - \QuadraticFrom(x) - \QuadraticFrom(y) \equiv \BilinearForm(x,y) \mod \Lambda$.
\end{itemize}
We call $\QuadraticModule$ \introduce{non-degenerate} if the map $\BilinearForm^\vee: \Module \to \Module^\vee: x \mapsto \BilinearForm(-,x)$ is an isomorphism.

\begin{example}
	\label{expl:hyperbolicmodule}
	The \introduce{$\FormParameter$-hyperbolic module} $\HyperbolicModule = (H, \BilinearForm, \QuadraticFrom)$ is
	the $\FormParameter$-quadratic module with
	$\Ring$-module $H = \Ring^{\oplus 2} = \Ring\{e,f\}$, $\BilinearForm$ given by
	$\BilinearForm(e,f) = 1$, $\BilinearForm(f,e) = \epsilon$ and $\BilinearForm(e,e) =
	\BilinearForm(f,f) = 0$, and function $\QuadraticFrom$ defined by $\QuadraticFrom(e) = \QuadraticFrom(f) = 0$.
\end{example}

Given two $\FormParameter$-quadratic modules $\QuadraticModule$ and
$\QuadraticModule'$ their orthogonal direct sum is
$\QuadraticModule \oplus \QuadraticModule' = (\Module \oplus \Module',
\BilinearForm_\oplus, \QuadraticFrom_\oplus)$ where $\BilinearForm_\oplus(x +
x', y+y') = \BilinearForm(x, y) + \BilinearForm'(x',y')$ and
$\QuadraticFrom_\oplus(x+x') = \QuadraticFrom(x) + \QuadraticFrom'(x')$. A
morphism $f: \QuadraticModule \to \QuadraticModule'$ of
$\FormParameter$-quadratic modules is a $\Ring$-homomorphism such that
$\BilinearForm(x,y) = \BilinearForm'(f(x), f(y))$ and $\QuadraticFrom(x) =
\QuadraticFrom'(f(x))$. Given a $\FormParameter$-quadratic module $\QuadraticModule$ and a $\FormParameter$-submodule $\QuadraticModule' \subseteq \QuadraticModule$, then we denote by
$(\QuadraticModule')^\perp = \{x \in \QuadraticModule: \BilinearForm(x,y) = 0 \text{ for all } y \in \QuadraticModule'\}$ the orthogonal complement of $\QuadraticModule'$ in $\QuadraticModule$.

\begin{lemma}
	\label{lem:split-injection-non-degenerate-quadratic-module}
	If $\QuadraticModule$ is a non-degenerate $\FormParameter$-quadratic
	module, then any map of $\FormParameter$-modules $f: \QuadraticModule \to
	\QuadraticModule'$ is split injective and $\QuadraticModule' \cong
	\QuadraticModule \oplus \im(f)^\perp$. Here, $\im(f)^\perp$ is the unique complement of $\im(f)$ in $\QuadraticModule'$.
\end{lemma}

\begin{proof}
	See e.g.\ end of first paragraph of  \cite[§3.\ Algebra]{GalatiusRandalWilliamsStability} or \cite[5.1.2. and 5.1.3. et seq.]{hahnomeara1989}.
\end{proof}

\begin{lemma}
	\label{lem:hyperbolic-module-non-degenerate}
	The $\FormParameter$-hyperbolic module $\HyperbolicModule$ in \autoref{expl:hyperbolicmodule} is non-degenerate.
\end{lemma}

\begin{proof}
	See e.g.\ first paragraph of proof of \cite[Lemma 2.7]{sprehnwahl2020homologicalstabilityforclassicalgroups}.
\end{proof}

Hence, the following notion of Witt index can be used to assign to
$\FormParameter$-quadratic modules, which are not the form
$\HyperbolicModule^{\oplus n}$, a ``rank''.

\begin{definition}
	The \introduce{unstable Witt index} of a $\FormParameter$-quadratic module
	$\QuadraticModule$ is
	\[
	\WittIndex{\QuadraticModule} = \sup \{g \in \NaturalNumbers : \text{ there
	exists a morphism } \HyperbolicModule^{\oplus g} \to \QuadraticModule\}.
	\]
	The \introduce{stable Witt index} of a $\FormParameter$-quadratic module
	$\QuadraticModule$ is
	\[
	\StableWittIndex{\QuadraticModule} = \sup \{\WittIndex{\QuadraticModule
	\oplus \HyperbolicModule^{\oplus k}} - k : k \geq 0\}.
	\]
\end{definition}

\begin{lemma}
		Let $\Module$ be a $\FormParameter$-quadratic module and $\HyperbolicModule$ the $\FormParameter$-hyperbolic module in \autoref{expl:hyperbolicmodule}, then the
		unstable and stable Witt index satisfy
		\[
		\WittIndex{\QuadraticModule \oplus \HyperbolicModule} \geq
		\WittIndex{\QuadraticModule} + 1
		\text{ and }
		\StableWittIndex{\QuadraticModule \oplus \HyperbolicModule} =
		\StableWittIndex{\QuadraticModule} + 1
		\text{ respectively.}
		\]
\end{lemma}
\begin{proof}
	See e.g. \cite[§3.\ Algebra]{GalatiusRandalWilliamsStability} and \cite[Lemma 2.7]{sprehnwahl2020homologicalstabilityforclassicalgroups}.
\end{proof}

\subsection{A Quillen family for unitary groups}

We now start the construction of the Quillen family
$(\Aut{n}{\Ring}{\Lambda}{\epsilon}, \SimplicialComplex^n_\bullet)_{n \in
\NaturalNumbers}$ for the sequence of groups
$\{\Aut{n}{\Ring}{\Lambda}{\epsilon}\}_{n \in \NaturalNumbers}$ where $\Ring$
is $\Integers$ or a field of characteristic zero. As before we start by
introducing the
semi-simplicial set $\SimplicialComplex^n_\bullet$ on which
$\Aut{n}{\Ring}{\Lambda}{\epsilon}$ acts simplicially.

\begin{definition}
	\label{def:ComplexOfHyperbolicSplitInjections}
	Let $\FormParameter$ be a form parameter for $\Ring$ and $\QuadraticModule$ a $\FormParameter$-quadratic module. The complex of hyperbolic $\FormParameter$-split injections
	$\SimplicialComplex^\QuadraticModule_\bullet$ of $\QuadraticModule$ is the semi-simplicial set, whose set of $p$-simplices is the set of $\FormParameter$-quadratic module morphisms
	\[
		\HyperbolicModule^{\oplus p+1} \to \QuadraticModule.
	\]
	and whose $i$-th face map is given by precomposition with the inclusion
	\[
	\HyperbolicModule^{p} = \HyperbolicModule^{i} \oplus 0 \oplus
	\HyperbolicModule^{(p+1)-(i+1)} \to	\HyperbolicModule^{p+1}.
	\]
	If $n \in \NaturalNumbers$ and $\QuadraticModule =
	\HyperbolicModule^{\oplus n}$, we write $\SimplicialComplex^n_\bullet
	\coloneqq \SimplicialComplex^\QuadraticModule_\bullet$.
\end{definition}

\begin{remark}
	The complexes of hyperbolic $\FormParameter$-split injections
	$\SimplicialComplex^\QuadraticModule_\bullet$ of $\QuadraticModule$
	or closely related posets and simplicial complexes have been studied in
	\cite{CharneyStability, VogtmannStability,
	mirzaiivanderkallen2002} and, more recently,
	\cites[Section 5.4]{randalwilliamswahl2017homologicalstabilityforautomorphismgroups}
	[Section 3]{friedrich2016homologicalstabilityofautomorphismgroupsofquadraticmodulesandmanifolds}
	[Section 3 \& Section 4]{GalatiusRandalWilliamsStability}.
\end{remark}

\begin{remark}
	\label{rem:DimensionOfComplexOfHyperbolicSplitInjections}
	Notice that the unstable Witt index satisfies
	$\dim(\SimplicialComplex^\QuadraticModule_\bullet) =
	\WittIndex{\QuadraticModule} - 1$. In particular, it holds that
	$\dim(\SimplicialComplex^{n}_\bullet) = n - 1$ since
	$\WittIndex{\HyperbolicModule^{\oplus n}} = n$.
\end{remark}

The first goal of this section is to prove that
$(\Aut{n}{\Ring}{\Lambda}{\epsilon}, \SimplicialComplex^n_\bullet)_{n \in
\NaturalNumbers}$ is indeed a Quillen family.

\begin{proposition}
	\label{prop:AQuillenFamilyForUnitaryGroups}
	Let $\Ring$ denote the integers $\Integers$ or any field of characteristic zero. Let $\FormParameter$ be a 
	form parameter for $\Ring$. Let $\SimplicialComplex^n_\bullet$ denote 
	the complex of hyperbolic $\FormParameter$-split injections of 
	$\HyperbolicModule^{\oplus n}$ introduced in 
	\autoref{def:ComplexOfHyperbolicSplitInjections}. Then 
	$(\Aut{n}{\Ring}{\Lambda}{\epsilon}, \SimplicialComplex^n_\bullet)_{n \in 
		\NaturalNumbers}$ is a Quillen family 
	with parameters $$\gamma(n) = \lfloor \frac{n-4}{2} \rfloor \text{ and } \tau(n) = n - 4$$ 
\end{proposition}

This proposition is a consequence of the next bounded acyclicity theorem and several lemmas.

\begin{theoremnum}
	\label{thm:BoundedAcyclicityResultForComplexOfHyperbolicSplitInjections}
	The complex of hyperbolic $\FormParameter$-split injections $\SimplicialComplex^n_\bullet$ 
	is boundedly $\gamma(n)$-acyclic for $$\gamma(n) = \lfloor \frac{n-4}{2} \rfloor.$$
\end{theoremnum}

The proof of \autoref{thm:BoundedAcyclicityResultForComplexOfHyperbolicSplitInjections} is given in \autoref{scn:BoundedAcyclicityProofsForUnitaryGroups} and relies on the uniformly bounded simplicial methods developed in \autoref{sec:toolbox}.

\begin{lemma}
	\label{lem:TransitivityUnitaryGroups}
	The action of $\Aut{n}{\Ring}{\Lambda}{\epsilon}$ on the 
	 complex of hyperbolic $\FormParameter$-split injections  
	 $\SimplicialComplex^n_\bullet$ $\tau(n)$-transitive for $\tau(n) 
	= n - 4$.
\end{lemma}
\begin{proof}
	Using that $\Integers$ and fields have Krull dimension less than or equal to $1$, it follows that the unitary stable rank $\on{usr}(\Ring)$ of $\Ring$ is less than or equal to $3$ (see e.g.\ \cite[Examples 3.2 (1)]{friedrich2016homologicalstabilityofautomorphismgroupsofquadraticmodulesandmanifolds}). Hence \cite[Proposition 5.13]{randalwilliamswahl2017homologicalstabilityforautomorphismgroups} implies that if $\HyperbolicModule \oplus \HyperbolicModule^{\oplus l-1} \cong \HyperbolicModule \oplus Q$ and $l-1 \geq 3$, then $\HyperbolicModule^{\oplus l-1} \cong Q$. 
	
	Therefore, we can argue similarly to \autoref{lem:Transitivy}: Let $p \geq 0$. Consider the p-simplex $\sigma_{n,p}$ given by the inclusion $$\HyperbolicModule^{\oplus p+1} \hookrightarrow \HyperbolicModule^{\oplus p+1} \oplus \HyperbolicModule^{\oplus n-p-1} = \HyperbolicModule^{\oplus n}$$ and let $f: \HyperbolicModule^{\oplus p+1} \to \HyperbolicModule^{\oplus n}$ be some other p-simplex of $\SimplicialComplex^n_\bullet$. Since the simplex $f$ is a split injection (by \autoref{lem:split-injection-non-degenerate-quadratic-module} and \autoref{lem:hyperbolic-module-non-degenerate}), it holds that
	$f: \HyperbolicModule^{\oplus p+1} \hookrightarrow \im(f) \oplus Q \cong \HyperbolicModule^{\oplus n}$ where $\im(f) \cong \HyperbolicModule^{\oplus p+1}$ via a isomorphism of $\FormParameter$-quadratic modules $\phi$ satisfying $\phi \circ f = id_{\HyperbolicModule^{\oplus p+1}}$. Notice that $p \leq \tau(n)$ is equivalent to $3 \leq n - p - 1$. Since $\HyperbolicModule^{\oplus p+1} \oplus \HyperbolicModule^{\oplus n - p - 1} = \HyperbolicModule^{\oplus n} \cong \im(f) \oplus Q$, the cancellation result in the first paragraph therefore implies that $\HyperbolicModule^{\oplus n-p-1} \cong Q$ via some isomorphism of $\FormParameter$-quadratic modules $\psi$. Hence the $\FormParameter$-quadratic module morphism $M = \phi \oplus \psi$ has the desired property.
\end{proof}

\begin{lemma}
	Let $\Ring$ denote the integers or any field of characteristic zero and let
	$\FormParameter$ be a
	form parameter. The sequence $(\Aut{n}{\Ring}{\Lambda}{\epsilon},
	\SimplicialComplex^n_\bullet)_{n \in \NaturalNumbers}$ is
	$\apply{\tau}{n}$-compatible for $\apply{\tau}{n} = n-4$.
\end{lemma}
\begin{proof}
	By \autoref{lem:TransitivityUnitaryGroups}
	$(\Aut{n}{\Ring}{\Lambda}{\epsilon},
	\SimplicialComplex^n_\bullet)_{n \in \NaturalNumbers}$
	is $\apply{\tau}{n}$-transitive and hence
	\autoref{itm:CompatibilityTransitivity} of \autoref{def:Compatibility}
	is satisfied. For the flag in \autoref{itm:CompatibilityFlag} of
	\autoref{def:Compatibility} we chose, as
	in the previous proof, $\Simplex{n,p}$ to be inclusion
	$$\Simplex{n,p}: \HyperbolicModule^{\oplus p+1} \hookrightarrow
	\HyperbolicModule^{\oplus p+1} \oplus \HyperbolicModule^{\oplus n-(p+1)} =
	\HyperbolicModule^{\oplus n}.$$

	Now for \autoref{itm:CompatibilityStabilizer} of
	\autoref{def:Compatibility}, we let $\GroupElement_{n,p,i} \in
	\Aut{n}{\Ring}{\Lambda}{\epsilon}$ be the ``swap'' automorphism
	\[
	\GroupElement_{n,p,i}: \HyperbolicModule^{\oplus i}
	\oplus
	\HyperbolicModule^{\oplus (p+1) - i} \oplus
	\HyperbolicModule \oplus \HyperbolicModule^{\oplus n-(p+2)} \to
	\HyperbolicModule^{\oplus i} \oplus \HyperbolicModule \oplus
	\HyperbolicModule^{\oplus (p+1)
	- i} \oplus \HyperbolicModule^{\oplus n-(p+2)}
	\]
	that maps $(i+1)$-st copy of $\HyperbolicModule$ to the $(i+2)$-nd, the
	$(i+3)$-rd to the $(i+4)$-th, $\dots$, the $(p+1)$-st to the $(p+2)$-nd,
	the $(p+2)$-nd to the $(i+1)$-st and keeps all other copies of
	$\HyperbolicModule$ fixed. This automorphism has the property that
	\[
	\apply{\GroupElement_{n,p,i}}{\Simplex{n,p}} = \BoundarySimplex_{i}
	\Simplex{n,p+1}.
	\]
	By \autoref{lem:split-injection-non-degenerate-quadratic-module}, the stabilizer $\Stabilizer_{n,p+1}$ of $\Simplex{n,p+1}$ is
	\[
	\Stabilizer_{n,p+1} \cong \{id_{\HyperbolicModule^{\oplus p+2}}\} \times
	\Aut{n-p-2}{\Ring}{\Lambda}{\epsilon}.
	\]
	Evidently, $\GroupElement_{n,p,i}$ centralizes
	any element in $\Stabilizer_{n,p+1}$, because $\GroupElement_{n,p,i}$ is
	supported the first $(p+2)$ copies of $\HyperbolicModule$. This shows that
	\autoref{itm:CompatibilityStabilizer} of \autoref{def:Compatibility} holds.

	There is an obvious projection $\EpimorphismSetup_{n,p}:
	\Stabilizer_{n,p+1} \to \Aut{n-p-2}{\Ring}{\Lambda}{\epsilon}$ that is an
	isomorphism with inverse $\SectionSetup_{n,p+1}$ making the diagram
	\[
	\begin{tikzcd}
		\Stabilizer_{n,p+1}
		\arrow[r, "{c_{\GroupElement_{n,p,i}}}"]
		&
		\Stabilizer_{n,p}
		\arrow[d, "\EpimorphismSetup_{n,p}", "\cong"']
		\\
		\Aut{n-p-2}{\Ring}{\Lambda}{\epsilon}
		\ar[u, "\SectionSetup_{n,p+1}", "\cong"']
		\ar[r, hook]
		&
		\Aut{n-p-1}{\Ring}{\Lambda}{\epsilon}
	\end{tikzcd}
	\]
	commute. This means that \autoref{itm:CompatibilityEpimorphism}
	of \autoref{def:Compatibility} holds as well.
\end{proof}

\subsection{Proof of
\autoref{thm:BoundedCohomologicalStabilityForUnitaryGroups}}

We will now apply \autoref{thm:stability-for-quillen-families} to prove
\autoref{thm:BoundedCohomologicalStabilityForUnitaryGroups}.

\begin{proof}[Proof of 
\autoref{thm:BoundedCohomologicalStabilityForUnitaryGroups}]
	By \autoref{prop:AQuillenFamilyForUnitaryGroups}, we know that 
	$(\Aut{n}{\Ring}{\Lambda}{\epsilon}, \SimplicialComplex^n_\bullet)_{n \in 
	\NaturalNumbers}$ is a Quillen family with parameters $\gamma(n) = \lfloor \frac{n-4}{2} \rfloor$ and $\tau(n) = n-4$.
	Therefore, we can invoke 
	\autoref{thm:stability-for-quillen-families} using $q_0 = 1$. Indeed, $q_0 
	=  1$ is a valid choice: If $q = 0$ or $q = 1$, then
	$\BoundedCohomologyOfSimplicialObject{q}{\Group_{n+1}}{\Reals}
	\xrightarrow{\cong} 
	\BoundedCohomologyOfSimplicialObject{q}{\Group_{n}}{\Reals}$ is an 
	isomorphism for \emph{any} sequence of groups $\{\Group_n\}_{n \in 
	\NaturalNumbers}$. For $q = 0$ this is trivial, and for $q = 1$ it follows 
	from the fact that $\BoundedCohomologyOfSimplicialObject{1}{\Group}{\Reals} 
	= 0$ for any group $\Group$ \cite[Chapter 
	2.1]{frigerio2017boundedcohomologyofdiscretegroups}. We are hence left with 
	checking that
	\[
	\min
	\left\{
	\apply{\widetilde{\AcyclicityFunction}}{q,n}
	,
	\apply{\widetilde{\CompatibilityFunction}}{q,n} - 1
	\right\}
	\geq 0
	\]
	is equivalent to
	\[
	\frac{n-3}{2} \geq q
	\]
	For this we compute 
	\begin{align*}
		\apply{\widetilde{\CompatibilityFunction}}{q,n} &= \min_{j=q_{0}}^{q} \left\{ \apply{\CompatibilityFunction}{n+1-2(q-j)} - j \right\} = \min_{j=1}^{q} \left\{ (n+1-2(q-j)-4) - j \right\}\\
		&= \min_{j=1}^{q} \left\{ n-2q-3 + j \right\} = n-2q-3 + 1 = n - 2q - 2
	\end{align*}
	\begin{align*}
		\apply{\widetilde{\AcyclicityFunction}}{q,n} 
		&= \min_{j=q_{0}}^{q} \left\{ 
		\apply{\AcyclicityFunction}{n+1-2(q-j)} - j \right\} 
		= \min_{j=1}^{q} \left\{ \left\lfloor \frac{n+1-2(q-j)-4}{2} \right\rfloor - j \right\}\\
		&= \min_{j=1}^{q} \left\{ \left\lfloor \frac{n-2q-3+2j}{2} \right\rfloor - j \right\}
		= \left\lfloor \frac{n - 2q - 3}{2} \right\rfloor
	\end{align*}
	Hence,
	\begin{align*}
	\min \left\{ \apply{\widetilde{\AcyclicityFunction}}{q,n}, 
	\apply{\widetilde{\CompatibilityFunction}}{q,n} - 1	\right\} \geq 0
	&\Longleftrightarrow \apply{\widetilde{\CompatibilityFunction}}{q,n} - 1 = n - 2 q - 3 \geq 0\\
	&\Longleftrightarrow \frac{n-3}{2} \geq q \qedhere
	\end{align*}
\end{proof}
\section{Posets, simplicial complexes, and regular semi-simplicial sets}
\label{scn:PosetsBackground}
In this section, we collect various definitions and discuss basic concepts
related to posets, simplicial complexes and regular semi-simplicial sets, which
are frequently used in the subsequent investigations of simplicial bounded
cohomology. The material presented here is mostly standard and discussed in
e.g.\ \cite[Section 9]{Björner_1995}, \cite[Appendix: Simplicial CW
structures]{hatcher2002}, \cite{rourkesanderson1971} and
\cite[2.10]{vanderkallen1980homologystabilityforlineargroups}.

\begin{definition}
	An \introduce{ordered simplicial complex} $\SimplicialComplex$ is a simplicial complex together with a partial ordering of its vertex set $X_0$ that restricts to a total ordering on every simplex $\sigma$ of $\SimplicialComplex$. Using the partial ordering to define the face maps, an ordered simplicial complex gives rise to a semi-simplicial set which we denote by $\SimplicialComplex_\bullet$.
\end{definition}

The remainder of this work contains a detailed study of the simplicial bounded cohomology of
semi-simplicial sets $\SimplicialComplex_\bullet$ that arise from ordered
simplicial complexes $\SimplicialComplex$. The techniques that we develop in
\autoref{sec:toolbox} for this purpose rely on the following basic notions for
simplicial complexes.

\begin{definition}
	\label{def:simplicial-complex-link-and-star}
	Let $\SimplicialComplex$ be a simplicial complex and let $\sigma$ be a simplex of $\SimplicialComplex$.	The \introduce{link} of $\sigma$ in $\SimplicialComplex$, denoted by $\Link{\SimplicialComplex}{\sigma}$, is the subcomplex of $\SimplicialComplex$ consisting of all simplices $\tau$ such that $\tau$ and $\sigma$ have disjoint vertex sets and $\sigma \sqcup \tau$ is a simplex.	The \introduce{star} of $\sigma$ in $\SimplicialComplex$, denoted by $\Star{\SimplicialComplex}{\sigma}$, is the subcomplex $\sigma \ast \Link{\SimplicialComplex}{\sigma}$ of $\SimplicialComplex$ given by the simplical join of $\sigma$ and the link of $\sigma$.
\end{definition}

Many of the ordered simplicial complexes that we are interested in are obtained
from posets via the following construction.

\begin{definition}
  Let $\Poset$ be a poset. The \introduce{order complex} of $\Poset$ is the ordered simplicial complex that is defined as follows: Its set of vertices
  $\Poset_{0}$ is equal to $\Poset$ and a $p$-simplex
  $\sigma \in \Poset_{p}$ is given by a flag
  $
    \PosetElement_{0}
    \lneq
    \ldots
    \lneq
    \PosetElement_{p}
  $ 
  consisting of $p+1$ elements. The $k$-th face of such a $p$-simplex $\sigma$ is given by omitting the element $\PosetElement_k$ in the flag $\PosetElement_{0} \lneq \ldots \lneq \PosetElement_{p}$.
\end{definition}

Whenever we discuss simplicial, topological or bounded cohomology properties of
a poset $\Poset$ in this article, we mean properties of its order complex.

In this work, we frequently pass from one poset $\Subposet$ to another $\Poset$
via gluing constructions. In order to describe these, we need to study the link
of an element $\PosetElement \in \Poset$ in the subposet $\Subposet$. We now
introduce notation and describe the basic construction that we shall study later.

\begin{definition}
	\label{def:poset-links}
	Let $\Poset$ be a poset and let $\Subposet$ be a subposet of $\Poset$. Let $\PosetElement \in \Poset$ be an element of $\Poset$.
	
	\begin{itemize}
		\item  We set
		\[
		\Subposet^+(\PosetElement) \coloneqq \{y \in \Subposet: y \geq \PosetElement\}
		\text{ and }
		\Subposet^-(\PosetElement) \coloneqq \{y \in \Subposet: y \leq \PosetElement\}.
		\]
		\item Furthermore, we let
		\[
		\Link{\Subposet}{\PosetElement}^+ \coloneqq \{y \in \Subposet : y > \PosetElement \}
		\text{ as well as }
		\Link{\Subposet}{\PosetElement}^- \coloneqq \{y \in \Subposet : y < \PosetElement \}.
		\] 
		and define the link of $\PosetElement$ in $\Subposet$ as
		\[
		\Link{\Subposet}{\PosetElement} \coloneqq \Link{\Subposet}{\PosetElement}^+ \ast \Link{\Subposet}{\PosetElement}^-.
		\]
	\end{itemize}
\end{definition}

\begin{remark}
	The notation $\Subposet^+(\PosetElement)$ and $\Subposet^-(\PosetElement)$ introduced above is used by van der Kallen in \cite[2.10]{vanderkallen1980homologystabilityforlineargroups}. Since this article relies on arguments contained in \cite{vanderkallen1980homologystabilityforlineargroups}, we decided to adopt this notation instead of using e.g.\ $\Subposet_{\geq \PosetElement}$ and $\Subposet_{\leq \PosetElement}$
\end{remark}

Note that the order complex of the poset $\Link{\Poset}{\PosetElement}$ of an element $\PosetElement$ in a poset $\Poset$ (as in \autoref{def:poset-links}) is exactly the link of
the vertex $\PosetElement$ in the order complex of $\Poset$ in the sense of \autoref{def:simplicial-complex-link-and-star}.

In particular, if a poset $\Poset$ is obtained from a subposet
$\Subposet \subseteq \Poset = \Subposet \sqcup \{\PosetElement\}$ by adding one
element $\PosetElement$, then the order complex of $\Poset$ is obtained from
that of $\Subposet$ by gluing the cone $\Star{\Poset}{\PosetElement} =
\{\PosetElement\} \ast \Link{\Poset}{\PosetElement}$ to $S$ along
$\Link{\Subposet}{\PosetElement} = \Link{\Poset}{\PosetElement}$,
$$\Poset  = \Subposet \cup_{\Link{\Subposet}{\PosetElement}}
\Star{\Poset}{\PosetElement}.$$

More general, we are interested in the following gluing operations.

\begin{definition}
	\label{def:PosetObtainedByGluing}
	Let $\Subposet \subseteq \Poset$ be a subposet such that the set-complement $\Poset \setminus \Subposet$ is a \emph{discrete} subposet of $\Poset$, i.e.\ for all $\PosetElement \in \Poset \setminus \Subposet$ it holds that $\Link{\Poset \setminus \Subposet}{\PosetElement} = \emptyset$.
	Then we say that \emph{$\Poset$ is obtained from $\Subposet$ by gluing all elements $\PosetElement \in \Poset \setminus \Subposet$ to $\Subposet$} along their respective link $\Link{\Subposet}{x}$ in $\Subposet$.
\end{definition}

We close this subsection by discussing the relation between semi-simplicial
sets and posets. For this, we need the following technical assumption on our
semi-simplicial sets.

\begin{definition}
	\label{def:regular-semi-simplicial-set}
	A semi-simplicial set $\SimplicialComplex_\bullet$ is called \emph{regular} if every $p$-simplex $\sigma \in \SimplicialComplex_p$ has $(p+1)$ distinct codimension-1 faces, i.e.\ if $|\{\apply{\FaceMap_0}{\sigma}, \dots, \apply{\FaceMap_p}{\sigma}\}| = p+1$.
\end{definition}

\begin{remark}
	From a topological view point, the regularity condition in \autoref{def:regular-semi-simplicial-set} means that the attaching map of the $p$-cell corresponding to $\sigma$ in the geometric realization $\GeometricRealization{\SimplicialComplex_\bullet}$ of the semi-simplicial set $\SimplicialComplex_\bullet$ is a homeomorphism (see e.g.\ \cite[Appendix: Simplicial CW Structures]{hatcher2002} or \cite{rourkesanderson1971}).
\end{remark}

Any regular semi-simplicial set can be encoded in a poset, its barycentric subdivision.

\begin{definition}
	\label{def:barycentric-subdivision}
  Let $\SimplicialComplex_{\bullet}$ be a regular semi-simplicial set. The barycentric
  subdivision $\BarycentricSubdivion{\SimplicialComplex_{\bullet}}$ of $\SimplicialComplex_\bullet$ is the poset of simplices of $\SimplicialComplex_{\bullet}$, i.e.\ the underlying set is
  \[
  	\BarycentricSubdivion{\SimplicialComplex_{\bullet}}
  	=
    \bigsqcup_{p \geq 0}
    \SimplicialComplex_{p}
  \]
  and the partial order is defined by setting $\Simplex{1}\leq \Simplex{2}$ if $\Simplex{1}$ is a face of $\Simplex{2}$.
\end{definition}

\begin{remark}
	\label{rem:geometric-realization-of-regular-semi-simplicial-sets}
	The geometric realization $\GeometricRealization{\SimplicialComplex_\bullet}$ of a regular semi-simplicial set is homeomorphic to the geometric realization of its barycentric subdivision $\GeometricRealization{\BarycentricSubdivion{\SimplicialComplex_\bullet}}$ (see e.g.\ \cite[Appendix: Simplicial CW Structures]{hatcher2002} or \cite{rourkesanderson1971}). This is not true if the regularity assumption is dropped. In \autoref{lem:BarycentricSubdivision}, we show that, similarly, the simplicial bounded cohomology of a regular semi-simplicial set is preserved under barycentric subdivisions.
\end{remark}
\section{Uniformly bounded simplicial methods}
\label{sec:toolbox}

As explained in \autoref{scn:SBC}, simplicial bounded cohomology
does not satisfy all Eilenberg--Steenrod axioms -- the homotopy and the
additivity axiom are only satisfied under additional assumptions. This section
investigates which of the tools that are classically used to investigate acyclicity properties of posets,
simplicial complexes, and semi-simplicial sets can be adapted to simplicial bounded cohomology.
In the first subsection, we introduce the notion of ``uniform $n$-acyclicity'' on which these uniformly bounded simplicial methods rely. In the others, we assemble our toolbox.

\subsection{Uniform $n$-acyclicity}

The key notion of this article is based on the following definition, which Matsumoto--Morita \cite{matsumotomorita1985boundedcohomologyofcertaingroupsofhomeomorphisms} coined in their seminal work on the bounded acyclicity of the group of compactly supported homeomorphisms of $\Reals^n$.

\begin{definition}
  Let $q \in \Integers$. A normed chain complex $\AbstractChainComplex{*}$ satisfies the
  \introduce{$q$-uniform boundary condition}, short $q$-UBC, \introduce{with constant
  $\UBCConstant{q}{\AbstractChainComplex{*}}$} if for
  every boundary
  $
    \Chain
    \in
    \apply
      {\BoundarySimplex_{q+1}}
      {\AbstractChainComplex{q+1}}
  $
  there exists a chain $\BoundingChain \in \AbstractChainComplex{q+1}$ with
  $\apply{\BoundarySimplex}{\BoundingChain} = \Chain$ and
  $
    \Norm{\BoundingChain}
    \leq
    \UBCConstant{q}{\AbstractChainComplex{*}}
    \Norm{\Chain}.
  $
  If $\apply{\BoundarySimplex_{q+1}}{\AbstractChainComplex{q+1}} = \{0\}$ (e.g.\ whenever $\AbstractChainComplex{q+1}$ is the trivial module), then the constant $\UBCConstant{q}{\AbstractChainComplex{*}}$ is defined to be equal to zero.
\end{definition}

The UBC has many interesting connections to the bounded cohomology of the
normed chain complex. The most important for us is the following \cite[Theorem
2.8]{matsumotomorita1985boundedcohomologyofcertaingroupsofhomeomorphisms} of
Matsumoto--Morita, see also \cite[Theorem
6.8]{frigerio2017boundedcohomologyofdiscretegroups}.

\begin{proposition}[Matsumoto--Morita]
\label{prop:MatsumotoMorita}
  A normed chain complex $C_*$ satisfies the $q$-UBC if and only if the
  comparison map
  \[
    c_{q+1}
    \colon
    \BoundedCohomologyOfSpaceObject{q+1}{C_*}{\Reals}
    \to
    \CohomologyOfSpaceObject{q+1}{C_*}{\Reals}
  \]
  is injective.
\end{proposition}

This inspired the following notion, which plays a central role in this work.

\begin{definition}
	\label{def:uniform-n-acyclicity}
	Let $n \geq -1$. We call a non-empty semi-simplicial set $\SimplicialComplex_{\bullet}$ \introduce{uniformly $n$-acyclic} if it is $n$-acyclic (i.e.\ the reduced real simplicial homology
	$\ReducedHomologyOfSpaceObject{q}{\SimplicialComplex_{\bullet}}{\Reals}$ vanishes in all degrees $0 \leq q\leq n$) and the simplicial chain complex of
	$\SimplicialComplex_{\bullet}$ satisfies the $q$-UBC for all $q\leq n$. Similarly, the inclusion
	$
	\Subcomplex_{\bullet} \hookrightarrow \SimplicialComplex_{\bullet}
	$
	of a non-empty semi-simplicial subset $\Subcomplex_\bullet$	is uniformly $n$-acyclic if the pair $(\SimplicialComplex_{\bullet},\Subcomplex_{\bullet})$ is $n$-acyclic (i.e.\
	$
	\HomologyOfSpaceObject
	{q}
	{\SimplicialComplex_{\bullet},\Subcomplex_{\bullet}}
	{\Reals}
	= 0
	$
	for $0 \leq q \leq n$) and the relative simplicial chain complex of
	$(\SimplicialComplex_{\bullet},\Subcomplex_{\bullet})$
	satisfies the $q$-UBC for all $q \leq n$.
\end{definition}

\autoref{prop:MatsumotoMorita} has the following immediate consequence.

\begin{corollary}
\label{cor:qUBCimpliesAcyclic}
  A uniformly $n$-acyclic semi-simplicial set is boundedly
  $n$-acyclic.
\end{corollary}

In particular, this yields a strategy for proving bounded acyclicity results for semi-simplicial sets. The computational toolbox developed in this section is a collection of techniques that allow one to check the acyclicity and the uniform boundary condition in \autoref{def:uniform-n-acyclicity} simultaneously.

\begin{remark}
	In the remainder of this work we derive formulas for UBC constants of various simplicial complexes and semi-simplicial sets.
	We want to emphasis that while precise expressions for these constants are provided for the sake of rigour and possible future applications, only their dependencies and their existence, \emph{not the explicit shape of the formulas}, are relevant for our arguments.
\end{remark}

We close this first subsection by giving some intuition for the UBC and discussing \autoref{def:uniform-n-acyclicity}: We note that a semi-simplicial set $\SimplicialComplex_{\bullet}$ is uniformly $-1$-acyclic if and only if $\SimplicialComplex_{\bullet}$ is non-empty, and that the $-1$-UBC constant $\UBCConstant{-1}{X}$ is equal to zero in this case. The next example gives a geometric interpretation of uniform $0$-acyclicity, and should be compared with the role that the diameter plays in \autoref{exm:SBCofR}.

\begin{example}
	\label{rem:0UBC}
	Let $\SimplicialComplex_\bullet$ be a non-empty semi-simplicial set. The
	simplicial chain complex of $\SimplicialComplex_\bullet$ satisfies the
	$0$-UBC if and only if there exists a constant $\UBCConstant{}{} \in
	\Reals_{\geq 0}$ such that each connected component of the $1$-skeleton of
	$\SimplicialComplex_\bullet$ has finite diameter at most $\UBCConstant{}{}$
	with respect to the simplicial metric that assigns length $1$ to every edge.
	In particular, \autoref{exm:SBCofR} and \autoref{expl:uniformity-condition}
	do not satisfy the $0$-UBC. It follows
	that $\SimplicialComplex_\bullet$ is uniformly $0$-acyclic if and only if its
	$1$-skeleton is non-empty, connected and of finite diameter. In this case,
	the $0$-UBC constant $\UBCConstant{0}{\SimplicialComplex}$ can be chosen to
	be the diameter of the $1$-skeleton divided by $2$.
\end{example}

The following more sophisticated example, shows that
$k$-acyclicity, even of a Kan simplicial set, is not
sufficient to ensure uniform $k$-acyclicity.
\begin{example}
  Let $K$ denote a hyperbolic knot in $S^3$, then by
  performing $\frac{1}{n}$-Dehn surgery, one obtains a
  homology $3$-sphere $M$, which, except for an at most
  finite set of values for $n$, will be hyperbolic.
  Since hyperbolic manifolds are aspherical, this implies
  that its fundamental group will be a $2$-acyclic
  word-hyperbolic group, but as was proven in
  \cite{EpsteinFujiwaraBCofHyperbolicGroups} its second
  bounded cohomology is non-zero. Hence using the main
  result of \cite{IvanovSimplicialMappingTheorem},we have
  that the singular set of $M$, is $2$-acyclic, but its
  second bounded cohomology is non-zero and hence it also
  cannot satisfy $1$-UBC. Similarly using results from
  \cite{RatcliffeTschantzHomologySpheres}, one can
  construct a $3$-acyclic example using more sophisticated
  tools.
\end{example}

The final remark concerns a possible alternative to \autoref{def:uniform-n-acyclicity}.

\begin{remark}
  Note that \autoref{prop:MatsumotoMorita} relates the $q$-UBC and the comparison map in degree $q+1$.
  To obtain \autoref{cor:qUBCimpliesAcyclic}, it therefore would be sufficient to require $n$-acyclicity and that the simplicial chain complex of $\SimplicialComplex_\bullet$ satisfies the $q$-UBC for all $q \leq n-1$, instead of $q \leq n$, in \autoref{def:uniform-n-acyclicity}. However, this would mean that e.g.\ any connected semi-simplicial set would be uniformly $0$-acyclic -- not just the ones with finite diameter. For this reason we think it is
  conceptually more reasonable to chose our \autoref{def:uniform-n-acyclicity}. This choice also seems more natural in the subsequent arguments and has no effect on the main theorems of this article.
\end{remark}

\subsection{Two-out-of-three property}

We now start developing a tools for checking that a semi-simplicial set is
uniformly $n$-acyclic. Many of these tools represent more quantitative
refinements of the corresponding axioms in \autoref{scn:SBC}.
The first item in our uniformly bounded simplicial toolbox is such a
``quantitative'' refinement for the long exact sequence in \autoref{lem:LES},
which keeps track of the uniform boundary condition. This
builds on the following observation.

\begin{observation}
	\label{obs:InclusionAcyclic}
	Let $(\SimplicialComplex_\bullet, \Subcomplex_\bullet)$ be a pair of semi-simplicial sets and let $q \geq 0$.
	Then, the short exact sequence
	$
	0
	\to
	\ChainComplex{q}{\Subcomplex_{\bullet}}
	\to
	\ChainComplex{q}{\SimplicialComplex_{\bullet}}
	\to
	\ChainComplex{q}{\SimplicialComplex_{\bullet},\Subcomplex_{\bullet}}
	\to
	0
	$
	of simplicial chain modules	splits norm-preservingly by \autoref{obs:norm-preserving-splitting}. As a consequence,
	$
	\Subcomplex_{\bullet}
	\hookrightarrow
	\SimplicialComplex_{\bullet}
	$
	being uniformly $n$-acyclic in the sense of \autoref{def:uniform-n-acyclicity} is equivalent to the following property: If $q \leq n$ and
	$\Chain \in \ChainComplex{q}{\SimplicialComplex_{\bullet}}$
	is a chain with
	$
	\partial \Chain \in \ChainComplex{q-1}{\Subcomplex_\bullet}
	$,
	then there exists a chain
	$
	\BoundingChain
	\in
	\ChainComplex{q+1}{\SimplicialComplex_\bullet}
	$
	such that
	$
	\partial \BoundingChain
	=
	\Chain
	+
	\Chain_{\Subcomplex}
	$,
	where $\Chain_{\Subcomplex}$ is contained in
	$\ChainComplex{q}{\Subcomplex_{\bullet}}$ and
	$
	\Norm{\BoundingChain}
	\leq
	\UBCConstant{q}{\SimplicialComplex,\Subcomplex}
	\Norm{\Chain}
	$.
\end{observation}

\begin{lemma}
	\label{lem:two-out-of-three-property}
  Let $\SimplicialComplex_\bullet$ denote an semi-simplicial set and
  $\Subcomplex_\bullet$ a semi-simplicial subset, then:
  \begin{enumerate}[(i)]
  \item \label{item-1-toot}
    If $\SimplicialComplex_\bullet$ is uniformly $n$-acyclic and $\Subcomplex_\bullet$ is
    uniformly $(n-1)$-acyclic,
    then $\Subcomplex_\bullet \hookrightarrow \SimplicialComplex_\bullet$ is uniformly
    $n$-acyclic with constant in degree $q\leq n$ given by
    \[
      \UBCConstantTwoOutOfThreeThree{q}{\UBCConstant{q}{\SimplicialComplex}}{\UBCConstant{q-1}{\Subcomplex}}
      \coloneqq
      \UBCConstant{q}{\SimplicialComplex}
      \left(
        1
        +
        \UBCConstant{q-1}{\Subcomplex}
        \cdot (q+1)
      \right)
    \]
    where $\UBCConstant{q}{\SimplicialComplex}$ and
    $\UBCConstant{q-1}{\Subcomplex}$ denote the UBC constants of $\SimplicialComplex_\bullet$ and $\Subcomplex_\bullet$ in
    degree $q$ and $q-1$, respectively.
  \item \label{item-2-toot}
    If $\SimplicialComplex_\bullet$ is uniformly $n$-acyclic and
    $\Subcomplex_\bullet \hookrightarrow \SimplicialComplex_\bullet$ is uniformly $n+1$-acyclic,
    then $\Subcomplex_\bullet$ is uniformly $n$-acyclic with constant in degree $q\leq n$
    given by
    \[
      \UBCConstantTwoOutOfThreeTwo{q}{\UBCConstant{q}{\SimplicialComplex}}{\UBCConstant{q+1}{\SimplicialComplex, \Subcomplex}}
      \coloneqq
      (q+2)
      \UBCConstant{q+1}{\SimplicialComplex, \Subcomplex}
      \UBCConstant{q}{\SimplicialComplex}
      +
      \UBCConstant{q}{\SimplicialComplex}
    \]
    where $\UBCConstant{q+1}{\SimplicialComplex, \Subcomplex}$ and
    $\UBCConstant{q}{\SimplicialComplex}$ denote the UBC constants of $(\SimplicialComplex_\bullet, \Subcomplex_\bullet)$ and $\SimplicialComplex_\bullet$
    in degree $q+1$ and $q$, respectively.
  \item \label{item-3-toot}
    If $\Subcomplex_\bullet$ and $\Subcomplex_\bullet \hookrightarrow \SimplicialComplex_\bullet$ are
    both uniformly $n$-acyclic, then $\SimplicialComplex_\bullet$ is uniformly
    $n$-acyclic as well with constant in degree $q \leq n$ given by
    \[
      \UBCConstantTwoOutOfThreeOne{q}{\UBCConstant{q}{\Subcomplex}}{\UBCConstant{q}{\SimplicialComplex, \Subcomplex}}
      \coloneqq
      \UBCConstant{q}{\SimplicialComplex, \Subcomplex}
      +
      \UBCConstant{q}{\Subcomplex}
      +
      (q+2)\UBCConstant{q}{\SimplicialComplex, \Subcomplex}\UBCConstant{q}{\Subcomplex}
    \]
    where $\UBCConstant{q}{\SimplicialComplex,\Subcomplex}$ and
    $\UBCConstant{q}{\SimplicialComplex}$ denote the UBC constants of $(\SimplicialComplex_\bullet,\Subcomplex_\bullet)$ and $\SimplicialComplex_\bullet$ in degree $q$, respectively.
  \end{enumerate}
\end{lemma}
\begin{proof}
  For \autoref{item-1-toot}, we use the criterion in \autoref{obs:InclusionAcyclic}. Let $q \leq n$ and
  $\Chain \in \ChainComplex{q}{\SimplicialComplex_{\bullet}}$
  be a chain with
  $
  \partial \Chain \in \ChainComplex{q-1}{\Subcomplex_\bullet}
  $.
  Then $\partial \Chain \in \ker(\partial_{q-1})$ and, since $\Subcomplex_\bullet$ is uniformly
  $(n-1)$-acyclic and $q-1 \leq n-1$, there exists a chain
  $\BoundingChain \in \ChainComplex{q}{\Subcomplex_\bullet}$ such that $\partial \BoundingChain = - \partial \Chain$ and
  \[
    \Norm{\BoundingChain}
    \leq
    \UBCConstant{q-1}{\Subcomplex}
    \Norm{\partial \Chain}
    \leq
    \UBCConstant{q-1}{\Subcomplex}
    (q+1)
    \Norm{\Chain}
  \]
  using that the norm of $\partial = \partial_q$ is at most $(q+1)$.
  Since $\partial(\Chain + \BoundingChain) = 0$, $\SimplicialComplex_\bullet$ is uniformly $n$-acyclic and $q \leq n$, it follows that
  $\Chain + \BoundingChain$ is the boundary of a chain $\alpha \in \ChainComplex{q+1}{\SimplicialComplex_\bullet}$ such that
  \[
    \Norm{\alpha}
    \leq
    \UBCConstant{q}{\SimplicialComplex}
    \Norm{\Chain + \BoundingChain}
    \leq
    \UBCConstant{q}{\SimplicialComplex}
    \left(
      1
      +
      \UBCConstant{q-1}{\Subcomplex}(q+1)
    \right)
    \Norm{\Chain}.
  \]
  Hence, \autoref{item-1-toot} follows from \autoref{obs:InclusionAcyclic}.

  For \autoref{item-2-toot}, let $q \leq n$ and $\Chain \in \ChainComplex{q}{\Subcomplex_\bullet}$ with $\partial \Chain = 0$.
  Since $\SimplicialComplex_\bullet$ is uniformly
  $n$-acyclic, there exists a chain $\BoundingChain \in
  \ChainComplex{q+1}{\SimplicialComplex}$ such that $\partial \BoundingChain = \Chain$ and
  \[
    \Norm{\BoundingChain}
    \leq
    \UBCConstant{q}{\SimplicialComplex}
    \Norm{\Chain}.
  \]
  Since $\Subcomplex_\bullet \hookrightarrow \SimplicialComplex_\bullet$ is uniformly $(n+1)$-acyclic and $\partial (\BoundingChain + \ChainComplex{q+1}{\Subcomplex_\bullet}) = 0 \in \ChainComplex{q}{\SimplicialComplex_\bullet, \Subcomplex_\bullet}$, the equivalence class $\BoundingChain + \ChainComplex{q+1}{\Subcomplex_\bullet}$ is the boundary of a class $\alpha + \ChainComplex{q+1}{\Subcomplex_\bullet}$ in $\ChainComplex{*}{\SimplicialComplex_\bullet, \Subcomplex_\bullet}$ such that
  \[
    \Norm{\alpha + \ChainComplex{q+1}{\Subcomplex_\bullet}}
    \leq
    \UBCConstant{q+1}{\SimplicialComplex, \Subcomplex}
    \Norm{\BoundingChain + \ChainComplex{q+1}{\Subcomplex_\bullet}}
    \leq
    \UBCConstant{q+1}{\SimplicialComplex, \Subcomplex}
    \Norm{\BoundingChain}
    \leq
    \UBCConstant{q+1}{\SimplicialComplex, \Subcomplex}
    \UBCConstant{q}{\SimplicialComplex}
    \Norm{\Chain}.
  \]
  Using \autoref{obs:norm-preserving-splitting} we can pick $\alpha \in \ChainComplex{q+1}{\SimplicialComplex_\bullet}$ such that $\Norm{\alpha} = \Norm{\alpha + \ChainComplex{q+1}{\Subcomplex_\bullet}}$. Then it follows that $\BoundingChain - \partial \alpha \in \ChainComplex{q+1}{\Subcomplex_\bullet}$ satisfies $\partial (\BoundingChain - \partial \alpha) = \partial \BoundingChain = \Chain$
  and
  \[
    \Norm{\BoundingChain - \partial \alpha}
    \leq
    \Norm{\BoundingChain} + \Norm{\partial \alpha}
    \leq
    \Norm{\BoundingChain} + (q+2)\Norm{\alpha}
    \leq
    \left(
      \UBCConstant{q}{\SimplicialComplex}
      +
      (q+2)
      \UBCConstant{q+1}{\SimplicialComplex, \Subcomplex}
      \UBCConstant{q}{\SimplicialComplex}
    \right)
    \Norm{\Chain}.
  \]

  For \autoref{item-3-toot}, let $q \leq n$ and $\Chain \in \ChainComplex{q}{\SimplicialComplex_\bullet}$ with $\partial \Chain = 0$.
  Since $\Subcomplex_\bullet \hookrightarrow \SimplicialComplex_\bullet$ is uniformly $n$-acyclic, the equivalence class
  $\Chain + \ChainComplex{q}{\Subcomplex_\bullet}$ is the boundary of a class
  $
    \BoundingChain_1 + \ChainComplex{q+1}{\Subcomplex_\bullet}
    \in
    \ChainComplex{q+1}{\SimplicialComplex_\bullet,\Subcomplex_\bullet}
  $
  such that
  \[
    \Norm{\BoundingChain_1 + \ChainComplex{q+1}{\Subcomplex_\bullet}}
    \leq
    \UBCConstant{q}{\SimplicialComplex, \Subcomplex}
    \Norm{\Chain + \ChainComplex{q}{\Subcomplex_\bullet}}
    \leq
    \UBCConstant{q}{\SimplicialComplex, \Subcomplex}
    \Norm{\Chain}
  \]
  Using \autoref{obs:norm-preserving-splitting} we can pick a chain
  $\BoundingChain_1 \in \ChainComplex{q+1}{\SimplicialComplex_\bullet}$ such
  that $\Norm{\BoundingChain_1} = \Norm{\BoundingChain_1 +
  \ChainComplex{q+1}{\Subcomplex_\bullet}}$. This representative satisfies that
  $\partial \BoundingChain_1 = \Chain + \Chain_{\Subcomplex}$ for some
  $\Chain_{\Subcomplex} \in \ChainComplex{q}{\Subcomplex_\bullet}$ with $\partial \Chain_{\Subcomplex} = 0$.
  Since $\Subcomplex_\bullet$ is uniformly $n$-acyclic, $\Chain_{\Subcomplex}$ is the
  boundary of another chain $\BoundingChain_2 \in \ChainComplex{q+1}{\Subcomplex_\bullet}$ such that
  \[
    \Norm{\BoundingChain_2}
    \leq
    \UBCConstant{q}{\Subcomplex}
    \Norm{\Chain_{\Subcomplex}}
    \leq
    \UBCConstant{q}{\Subcomplex}
    (\Norm{\partial \BoundingChain_1} + \Norm{\Chain})
    \leq
    \UBCConstant{q}{\Subcomplex}
    ((q+2)\Norm{\BoundingChain_1} + \Norm{\Chain})
    \leq
    \UBCConstant{q}{\Subcomplex}
    ((q+2)\UBCConstant{q}{\SimplicialComplex, \Subcomplex} + 1)
    \Norm{\Chain}
  \]
  using that $\Chain_{\Subcomplex} = \partial \BoundingChain_1 - \Chain$, that $\Norm{\partial_{q+1}} \leq q+2$ and the previous estimate.
  We finish the proof by observing that the boundary of $\BoundingChain_1 - \BoundingChain_2$ is $\partial(\BoundingChain_1 - \BoundingChain_2) = \Chain$
  and that
  \[
    \Norm{\BoundingChain_1 - \BoundingChain_2}
    \leq
    \Norm{\BoundingChain_1} + \Norm{\BoundingChain_2}
    \leq
    \left(
      \UBCConstant{q}{\SimplicialComplex, \Subcomplex}
      +
      \UBCConstant{q}{\Subcomplex}
      +
      (q+2)\UBCConstant{q}{\SimplicialComplex, \Subcomplex}\UBCConstant{q}{\Subcomplex}
    \right)
    \Norm{\Chain}.
  \]
\end{proof}

The following two lemmas are variants of \autoref{lem:two-out-of-three-property} that are
used in several inductive proofs in subsequent (sub-)sections.

\begin{lemma}
	\label{lem:factoring-through-highly-uniformly-acyclic-subcomplex}
  Let $\SimplicialComplex_\bullet$ denote a simplicial complex and $\Subcomplex_\bullet$ a
  subcomplex such that the inclusion
  $
    \Subcomplex_\bullet
    \hookrightarrow
    \SimplicialComplex_\bullet
  $
  is uniformly $n$-acyclic and additionally factors through a subcomplex
  $\Subcomplex_\bullet'$ of $\SimplicialComplex_\bullet$ which is uniformly $n$-acyclic. Then $\SimplicialComplex_\bullet$ is
  uniformly $n$-acyclic as well. The UBC constants for $\SimplicialComplex_\bullet$ in degrees $q \leq n$ can
  be chosen to be
  \[
    \UBCConstantFactorThrough{q}{\UBCConstant{q}{\SimplicialComplex,\Subcomplex}}{\UBCConstant{q}{\Subcomplex'}}
    \coloneqq
    \UBCConstant{q}{\SimplicialComplex,\Subcomplex}
    +
    \UBCConstant{q}{\Subcomplex'}
    +
    (q+2)\UBCConstant{q}{\SimplicialComplex,\Subcomplex}\UBCConstant{q}{\Subcomplex'}
  \]
  where $\UBCConstant{q}{\Subcomplex'}$ and $\UBCConstant{q}{\SimplicialComplex, \Subcomplex}$ denote
  the respective UBC constants of $\Subcomplex_\bullet'$ and $(\SimplicialComplex_\bullet, \Subcomplex_\bullet)$ in degree $q$, respectively.
\end{lemma}
\begin{proof}
  Let $q \leq n$ and let $\Chain \in \ChainComplex{q}{\SimplicialComplex_\bullet}$ be a chain with $\partial \Chain = 0$.
  The assumption that $\Subcomplex_\bullet \hookrightarrow \SimplicialComplex_\bullet$ is uniformly $n$-acyclic and \autoref{obs:InclusionAcyclic} imply that
  there exists a chain $\BoundingChain_1 \in \ChainComplex{q+1}{\SimplicialComplex_\bullet}$ such that
  $\partial \BoundingChain_1 = \Chain + \Chain_{\Subcomplex}$, where $\Chain_{\Subcomplex} \in \ChainComplex{q}{\Subcomplex_{\bullet}}$ and
  $\Norm{\BoundingChain_1} \leq \UBCConstant{q}{\SimplicialComplex,\Subcomplex} \Norm{\Chain}$.
  Using $\Chain_{\Subcomplex} = \partial \BoundingChain_1 - \Chain$, $\Norm{\partial_{q+1}} \leq q+2$ and the previous estimate,
  this implies that
  \[
    \Norm{\Chain_{\Subcomplex}}
    \leq
    \Norm{\partial \BoundingChain_1} + \Norm{\Chain}
    \leq
    (q+2)\Norm{\BoundingChain_1} + \Norm{\Chain}
    \leq
    ((q+2)\UBCConstant{q}{\SimplicialComplex,\Subcomplex} + 1) \Norm{\Chain}.
  \]
  Since
  $
  \partial \Chain_{\Subcomplex}
  =
  \partial \Chain
  +
  \partial \Chain_{\Subcomplex}
  =
  \partial \partial \BoundingChain_1
  =
  0  
  $
  and $\Subcomplex_\bullet'$ is uniformly $n$-acyclic, $\Chain_{\Subcomplex}$ is the
  boundary of a chain $\BoundingChain_2$ in $\ChainComplex{q+1}{\Subcomplex_\bullet'}$
  with
  \[
    \Norm{\BoundingChain_2}
    \leq
    \UBCConstant{q}{\Subcomplex'}
    \Norm{\Chain_{\Subcomplex}}
    \leq
    \UBCConstant{q}{\Subcomplex'}
    ((q+2)\UBCConstant{q}{\SimplicialComplex,\Subcomplex} + 1)
    \Norm{\Chain}
  \]
  We finish by observing that the boundary of $\BoundingChain = \BoundingChain_1 - \BoundingChain_2$ is $\Chain$ and
  we have
  \[
    \Norm{\BoundingChain_1 - \BoundingChain_2}
    \leq
    \Norm{\BoundingChain_1} + \Norm{\BoundingChain_2}
    \leq
    \left(
      \UBCConstant{q}{\SimplicialComplex,\Subcomplex}
      +
      \UBCConstant{q}{\Subcomplex'}
      +
      (q+2)\UBCConstant{q}{\SimplicialComplex,\Subcomplex}\UBCConstant{q}{\Subcomplex'}
    \right)
    \Norm{\Chain} \qedhere
  \]
\end{proof}
\begin{lemma}
	\label{lem:compositions-of-highly-acyclic-maps}
	Let $Z_\bullet \subseteq \Subcomplex_\bullet \subseteq \SimplicialComplex_\bullet$ be a nested sequence of
	simplicial complexes. Assume that the inclusions
	$
	Z_\bullet
	\hookrightarrow
	\Subcomplex_\bullet
	\text{ and }
	\Subcomplex_\bullet
	\hookrightarrow
	\SimplicialComplex_\bullet
	$
	are uniformly $n$-acyclic with UBC constant in degree $q \leq n$ given by $\UBCConstant{q}{\Subcomplex, Z}$ and $\UBCConstant{q}{\SimplicialComplex, \Subcomplex}$, respectively.
	Then the inclusion
	$
	Z_\bullet
	\hookrightarrow
	\SimplicialComplex_\bullet
	$
	is uniformly $n$-acyclic as well and its UBC constants in degree $q \leq n$ can
	be chosen as
	\[
	\UBCConstantFactorThrough{q}{\UBCConstant{q}{\SimplicialComplex,\Subcomplex}}{\UBCConstant{q}{\Subcomplex, Z}}
	=
	\UBCConstant{q}{\SimplicialComplex,\Subcomplex}
	+
	\UBCConstant{q}{\Subcomplex, Z}
	+
	(q+2)\UBCConstant{q}{\SimplicialComplex,\Subcomplex}\UBCConstant{q}{\Subcomplex, Z}.
	\]
\end{lemma}
\begin{proof}
	Let $q \leq n$ and let $\Chain \in \ChainComplex{q}{\SimplicialComplex_\bullet}$ be a chain with $\partial \Chain \in \ChainComplex{q-1}{Z_\bullet}$.
	Then $\partial \Chain \in \ChainComplex{q-1}{\Subcomplex_\bullet}$.
	The assumption that $\Subcomplex_\bullet \hookrightarrow \SimplicialComplex_\bullet$ is uniformly $n$-acyclic and \autoref{obs:InclusionAcyclic} imply that
	there exists a chain $\BoundingChain_1 \in \ChainComplex{q+1}{\SimplicialComplex_\bullet}$ such that
	$\partial \BoundingChain_1 = \Chain + \Chain_{\Subcomplex}$, where $\Chain_{\Subcomplex} \in \ChainComplex{q}{\Subcomplex_{\bullet}}$ and
	$\Norm{\BoundingChain_1} \leq \UBCConstant{q}{\SimplicialComplex,\Subcomplex} \Norm{\Chain}$.
	Exactly as in the proof of \autoref{lem:factoring-through-highly-uniformly-acyclic-subcomplex}, one checks that
	$
	\Norm{\Chain_{\Subcomplex}}
	\leq
	((q+2)\UBCConstant{q}{\SimplicialComplex,\Subcomplex} + 1) \Norm{\Chain}.
	$
	
	Next, we note that $0 = \partial \partial \BoundingChain_1 = \partial \Chain + \partial \Chain_{\Subcomplex}$ implies that $\partial \Chain_{\Subcomplex} \in \ChainComplex{q-1}{Z_\bullet}$.
	The assumption that $Z_\bullet \hookrightarrow \Subcomplex_\bullet$ is uniformly $n$-acyclic and \autoref{obs:InclusionAcyclic} therefore yield 
	a chain $\BoundingChain_2 \in \ChainComplex{q+1}{\SimplicialComplex_\bullet}$ such that
	$\partial \BoundingChain_2 = \Chain_{\Subcomplex} + \Chain_{Z}$, where $\Chain_{Z} \in \ChainComplex{q}{Z_{\bullet}}$ and
	$\Norm{\BoundingChain_2} \leq \UBCConstant{q}{\Subcomplex, Z} \Norm{\Chain_{\Subcomplex}}$.
	
	To finish, observe that $\BoundingChain = \BoundingChain_1 - \BoundingChain_2$ satisfies $\partial \BoundingChain = \Chain - \Chain_Z$
	and
	\[
	\Norm{\BoundingChain_1 - \BoundingChain_2}
	\leq
	\Norm{\BoundingChain_1} + \Norm{\BoundingChain_2}
	\leq
	\left(
	\UBCConstant{q}{\SimplicialComplex,\Subcomplex}
	+
	\UBCConstant{q}{\Subcomplex, Z}
	+
	(q+2)\UBCConstant{q}{\SimplicialComplex,\Subcomplex}\UBCConstant{q}{\Subcomplex, Z}
	\right)
	\Norm{\Chain}. \qedhere
	\]
\end{proof}

\subsection{Mayer--Vietoris principle}
As we have seen in \autoref{lem:MVSimp}, in contrast to ordinary bounded
cohomology, simplicial bounded cohomology
satisfies the Mayer-Vietoris axiom. This leads to a gluing principle for bounded acyclicity.
Our investigations of uniform acyclicity properties in subsequent sections rely on the following
``quantitative'' version of it.
\begin{lemma}
  \label{lem:MV-qubc}
  Let $\SimplicialComplex_\bullet$ denote a semi-simplicial set that is the union of
  the semi-simplicial subsets $\Subcomplex_\bullet$ and $\Subcomplex_\bullet'$.
  If $\Subcomplex_\bullet$, $\Subcomplex_\bullet'$ are uniformly $n$-acyclic with UBC constants in
  degree $q$ denoted by $\UBCConstant{q}{\Subcomplex}$ and
  $\UBCConstant{q}{\Subcomplex'}$, respectively, and the intersection
  $\Subcomplex_\bullet \cap \Subcomplex_\bullet'$ is
  uniformly $(n-1)$-acyclic with UBC constant
  $
    \UBCConstant
      {q}
      {\Subcomplex \cap \Subcomplex'}
  $
  in degree $q$,
  then $\SimplicialComplex_\bullet$
  is uniformly $n$-acyclic and the UBC constants in degrees $q \leq n$ can be chosen to be
  \[
  	\UBCConstantMayerVietoris{q}{\UBCConstant{q}{\Subcomplex}}{\UBCConstant{q}{\Subcomplex'}}{\UBCConstant{q-1}{\Subcomplex \cap \Subcomplex'}}
  	\coloneqq
    \left(
      1+(q+1)\UBCConstant{q-1}{\Subcomplex \cap \Subcomplex'}
    \right)
    \left(
      \UBCConstant{q}{\Subcomplex}
      +
      \UBCConstant{q}{\Subcomplex'}
    \right).
  \]
  If $n=0$, then the diameter of $\SimplicialComplex_\bullet$ is bounded by
  $\UBCConstant{0}{\Subcomplex}+\UBCConstant{0}{\Subcomplex'}$.
\end{lemma}
\begin{proof}
  Let $q\leq n$ and let $\Chain \in \ChainComplex{q}{\SimplicialComplex_\bullet}$ be a chain with $\partial \Chain = 0$.
  Write
  $
     \Chain
     =
    \Chain_{\Subcomplex}
    +
    \Chain_{\Subcomplex'}
  $, 
  where $\Chain_{\Subcomplex} \in \ChainComplex{q}{\Subcomplex_\bullet}$,
  $\Chain_{\Subcomplex'} \in \ChainComplex{q}{\Subcomplex_\bullet'}$ and
  $
    \Norm{\Chain}
    =
    \Norm{\Chain_{\Subcomplex}}
    +
    \Norm{\Chain_{\Subcomplex'}}.
  $
  (This decomposition is not unique, in general.)
  We have
  $
    0
    =
    \BoundarySimplex \Chain
    =
    \BoundarySimplex \Chain_{\Subcomplex}
    +
    \BoundarySimplex \Chain_{\Subcomplex'}
  $, 
  hence $\partial \Chain_{\Subcomplex} = - \partial \Chain_{\Subcomplex'}$.
  Therefore $\BoundarySimplex \Chain_{\Subcomplex}$ and $\partial \Chain_{\Subcomplex'}$
  are both contained in $\ChainComplex{q-1}{\Subcomplex_\bullet \cap \Subcomplex_\bullet'}$ and, since $\partial^2 = 0$, even in $\ker(\partial_{q-1})$.
  Since $\Subcomplex_\bullet \cap \Subcomplex_\bullet'$ is uniformly $(n-1)$-acyclicity and $q-1 \leq n-1$, it follows that
  $\partial \Chain_{\Subcomplex}$ is the boundary of a chain $\BoundingChain \in \ChainComplex{q}{\Subcomplex_\bullet \cap \Subcomplex_\bullet'}$ such that
  $$\Norm{\BoundingChain} \leq \UBCConstant{q-1}{\Subcomplex \cap \Subcomplex'} \Norm{\partial \Chain_{\Subcomplex}} \leq (q+1) \UBCConstant{q-1}{\Subcomplex \cap \Subcomplex'} \Norm{\Chain_{\Subcomplex}} \leq (q+1) \UBCConstant{q-1}{\Subcomplex \cap \Subcomplex'} \Norm{\Chain}$$
  using that $\Norm{\partial_q} \leq q+1$. Now $\Chain_{\Subcomplex} + \BoundingChain \in \ChainComplex{q}{\Subcomplex_\bullet}$ and
  $\Chain_{\Subcomplex'} - \BoundingChain \in \ChainComplex{q}{\Subcomplex_\bullet'}$ are by construction contained in $\ker(\partial_q)$ and both have norm at most
  $
    (1 + (q+1)
    \UBCConstant{q-1}{\Subcomplex \cap \Subcomplex'})
    \Norm{\Chain}
  $. 
  Since $\Subcomplex_\bullet$ and $\Subcomplex_\bullet'$ are uniformly $n$-acyclic and $q \leq n$, it follows that there exist $\BoundingChain_1 \in \ChainComplex{q+1}{\Subcomplex_\bullet}$ and $\BoundingChain_2 \in \ChainComplex{q+1}{\Subcomplex_\bullet'}$ with $\partial \BoundingChain_1 = \Chain_{\Subcomplex} + \BoundingChain$, $\partial \BoundingChain_2 = \Chain_{\Subcomplex'} - \BoundingChain$ and such that
  $$\Norm{\BoundingChain_1} \leq \UBCConstant{q}{\Subcomplex}\Norm{\Chain_{\Subcomplex} + \BoundingChain} \text { as well as } \Norm{\BoundingChain_2} \leq \UBCConstant{q}{\Subcomplex'}\Norm{\Chain_{\Subcomplex'} - \BoundingChain}$$
  Since
  $
    \Chain
    =
    (\Chain_{\Subcomplex} + \BoundingChain)
    +
    (\Chain_{\Subcomplex'} - \BoundingChain),
  $
  it follows that $\BoundingChain_1 + \BoundingChain_2 \in \ChainComplex{q+1}{\SimplicialComplex_\bullet}$ is a chain with boundary $\Chain$ such that
  \begin{align*}
    \Norm{\BoundingChain_1 + \BoundingChain_2} \leq \Norm{\BoundingChain_1} +
    \Norm{\BoundingChain_2}
    &
    \leq
    \UBCConstant{q}{\Subcomplex}\Norm{\Chain_{\Subcomplex} + \BoundingChain} +
    \UBCConstant{q}{\Subcomplex'}\Norm{\Chain_{\Subcomplex'} - \BoundingChain}
    \\
    &
    \leq
    (1 + (q+1)
    \UBCConstant{q-1}{\Subcomplex \cap
    \Subcomplex'})(\UBCConstant{q}{\Subcomplex} +
    \UBCConstant{q}{\Subcomplex'})
    \Norm{\Chain}
  \end{align*}
  as claimed.
\end{proof}

\subsection{Uniform homotopy equivalences}

In \autoref{scn:SBC}, we discussed why simplicial bounded cohomology is in general not homotopy invariant. We also observed that this can be rectified if one imposes additional boundedness conditions (see \autoref{lem:HomotopyInvariance}). In this subsection, we introduce a ``uniform'' version of homotopy equivalence for normed chain complexes that preserves the UBC. In the next three subsections, we then use this notion to investigate how uniform acyclicity properties behave under poset deformation retracts, cones and suspensions as well as barycentric subdivisions.

\begin{definition}
	\label{def:uniform-chain-homotopy-equivalence}
	Two normed chain complexes $\AbstractChainComplex{*}$ and $\AbstractChainComplex{*}'$ are called \introduce{uniformly homotopy equivalent} if the following holds:
	There exists a chain homotopy equivalence
	$
	f
	\colon
	\AbstractChainComplex{*}
	\to
	\AbstractChainComplex{*}'
	$
	with homotopy inverse
	$
	g
	\colon
	\AbstractChainComplex{*}'
	\to
	\AbstractChainComplex{*}
	$
	and chain homotopies
	$
	\Homotopy_{C}
	\colon
	C_{*}
	\to
	C_{*+1}
	$
	and
	$
	\Homotopy_{C'}
	\colon
	C'_{*}
	\to
	C'_{*+1}
	$
	, i.e. $g \circ f \sim_{\Homotopy_C} \Identity_{C}$
	and $f \circ g \sim_{\Homotopy_{C'}} \Identity_{C'}$, such that $f$, $g$,
	$\Homotopy_{C}$ and $\Homotopy_{C'}$ are bounded operators in each degree
	(Note that it is allowed that they are unbounded on the whole complex i.e.
	the operator norm in each degree tends to infinity as the degree increases).
	In this case, $(f,g,\Homotopy_{C}, \Homotopy_{C'})$ is called a
	\introduce{uniform homotopy equivalence} between $C_{*}$ and $C_{*}'$.
\end{definition}

The following lemma is elementary, but plays an important role.
\begin{lemma}
  \label{lem:UBCChainhomotopy}
  If $C_{*}$ and $C'_{*}$ are two uniformly homotopy equivalent normed chain complexes, then 
  $$C_{*} \text{ satisfies $q$-UBC if and only if } C'_{*} \text{ satisfies $q$-UBC.}$$
  Furthermore, the $q$-UBC constants $\UBCConstant{q}{C}$ and $\UBCConstant{q}{C'}$ satisfy
  \[
    \UBCConstant{q}{C}
    \leq
      \Norm{\Homotopy_C}
      +
      \Norm{g}
      \UBCConstant{q}{C'}
      \Norm{f}
    \text{ and }
    \UBCConstant{q}{C'}
    \leq
    \Norm{\Homotopy_{C'}}
    +
    \Norm{g}
    \UBCConstant{q}{C}
    \Norm{f},
  \]
  for any uniform chain homotopy equivalence $(f,g,\Homotopy_{C}, \Homotopy_{C'})$ between $C_{*}$ and $C_{*}'$.
\end{lemma}
\begin{proof}
  Assume that $C'_{*}$ satisfies the $q$-UBC. Let $\Chain_{C} \in \partial_{q+1}(C_{q+1})$ be a boundary in $C_q$.
  Since $f$ is a chain map, it follows that $\apply{f}{\Chain_{C}} \in C_{q}'$ is a boundary as well.
  It follows that there exists a chain $\BoundingChain_{C'} \in C_{q+1}'$ with 
  $\partial \BoundingChain_{C'} = \apply{f}{\Chain_{C}}$ such that 
  $
    \Norm{\BoundingChain_{C'}}
    \leq
    \UBCConstant{q}{C'}
    \Norm{\apply{f}{\Chain_{C}}}
    \leq 
    \UBCConstant{q}{C'}
    \Norm{f}
    \Norm{\Chain_{C}}
  $. 
  Now consider the element 
  $
    \BoundingChain_{C}
    \coloneqq
    \apply{g}{\BoundingChain_{C'}}
    -
    \apply{\Homotopy_{C}}{\Chain_{C}}
    \in C_{q+1}.
  $
  Its boundary is given by 
  \[
  \partial \BoundingChain_{C} 
  =
  \partial \apply{g}{\BoundingChain_{C'}} - \partial \apply{\Homotopy_{C}}{\Chain_{C}}
  =
  \apply{g}{\partial\BoundingChain_{C'}} - (\apply{g}{\apply{f}{\Chain_{C}}} - \Chain_{C} - \apply{\Homotopy_{C}}{\partial \Chain_{C}})
  =
  \Chain_{C},
  \] 
  using that $\partial \BoundingChain_{C'} = \apply{f}{\Chain_{C}}$ and that $\partial \Chain_{C} = 0$,
  and its norm satisfies
  \[
    \Norm{\BoundingChain_{C}}
    \leq
    \Norm{\apply{g}{\BoundingChain_{C'}}}
    +
    \Norm{\apply{\Homotopy_{C}}{\Chain_{C}}}
    \leq
    \Norm{g} \Norm{\BoundingChain_{C'}}
    +
    \Norm{\Homotopy_{C}}\Norm{\Chain_{C}}
    \leq
    \left(
      \Norm{\Homotopy_{C}}
      +
      \Norm{g}
      \UBCConstant{q}{C'}
      \Norm{f}
    \right)
    \Norm{\Chain_{C}}
  \]
  as required.
\end{proof}

As a consequence to the previous lemma and the definition, we obtain the
following corollary.
\begin{corollary}
	\label{cor:uniform-homotopy-equivalence}
	If $C_{*}$ and $C'_{*}$ are two uniformly homotopy equivalent normed chain
	complexes, then their bounded cohomology is isomorphic. Furthermore, $C_{*}$
	is uniformly $n$-acyclic if and only if $C'_{*}$ is uniformly $n$-acyclic,
	and the UBC constants are related as in \autoref{lem:UBCChainhomotopy}.
\end{corollary}

\subsection{Poset deformation retracts}

In this subsection, we introduce a key tool which we will use to investigate the bounded cohomology of posets. There exists many interesting tools to study homotopy properties of posets in the literature, see e.g.\ \cite{quillen1978} and \cite[Section 10]{Björner_1995}. The following lemma is a uniformly acyclic refinement of \cite[2.11. Lemma]{vanderkallen1980homologystabilityforlineargroups}, see also e.g.\ \cite[Corollary 10.12]{Björner_1995} and \cite[1.5]{quillen1978}.

\begin{lemma}
  \label{lem:PosetDeformation}
  Let $\Poset$ be a poset. Assume that $\Subposet \subseteq \Poset$ is a subposet such that for each
  $\PosetElement \in \Poset$ the poset
  $\Subposet^-(\PosetElement)$ has a supremum (in itself!).
  Then $\Subposet$ is a topological deformation retract of $\Poset$ and
  \begin{enumerate}
  	\item \label{item-i-pdr} the inclusion $\Inclusion	\colon \Subposet \hookrightarrow \Poset$ induces a uniform homotopy equivalence between $\ChainComplex{*}{\Subposet}$ and $\ChainComplex{*}{\Poset}$;
  	\item \label{item-ii-pdr} the simplicial bounded cohomology of $\Subposet$ and $\Poset$ is isomorphic. Furthermore, $\Subposet$ is uniformly $n$-acyclic if and only if $\Poset$ is uniformly $n$-acyclic, where the UBC constants in degree $q$ are related by
  	\[
  	\UBCConstant{q}{\Subposet}
  	\leq
  	\UBCConstant{q}{\Poset}
  	\text{ and }
  	\UBCConstant{q}{\Poset}
  	\leq
  	(q+1)
  	+
  	\UBCConstant{q}{\Subposet}.
  	\]
  \end{enumerate}
\end{lemma}
\begin{proof}
  We define a map of posets
  $
    \Retraction
    \colon
    \Poset
    \to
    \Subposet
  $
  via
  $
    \apply{\Retraction}{\PosetElement}
    =
    \sup \Subposet^-(\PosetElement)
  $. 
  The claim that $\Subposet$ is a topological deformation retract of $\Poset$ via $\Retraction$ is verified in \cite[2.11. Lemma]{vanderkallen1980homologystabilityforlineargroups}.

  Let $q \in \NaturalNumbers$. Since $\Inclusion_*: \ChainComplex{*}{\Subposet} \to \ChainComplex{*}{\Poset}$ and $\Retraction_*: \ChainComplex{*}{\Poset} \to \ChainComplex{*}{\Subposet}$ are induced by poset maps, it immediately follows that $\Norm{\Inclusion} \leq 1$ and $\Norm{\Retraction} \leq 1$ in degree $q$. Furthermore, $\Retraction \circ \Inclusion = \Identity_{\Subposet}$. Therefore, the chain endomorphism $\Retraction_* \circ \Inclusion_*$ of $\ChainComplex{*}{\Subposet}$ is homotopic to the identity by the trivial chain homotopy $\Homotopy_{\Subposet} = 0$, which is of norm $\Norm{\Homotopy_{\Subposet}} = 0$ in degree $q$.

   It remains to construct the chain homotopy $\Homotopy_{\Poset}: \ChainComplex{*}{\Poset} \to \ChainComplex{*+1}{\Poset}$ between the chain endomorphism
   $
    \Inclusion_*
    \circ
    \Retraction_*
  $
  and the identity map of $\ChainComplex{*}{\Poset}$. We define
  $\Homotopy_{\Poset}$ as follows:
  \begin{equation}
  	\label{eq:chain-homotopy}
    \apply{\Homotopy}{\Chain}
    =
    \sum_{i=0}^{q}
    (-1)^{i+1}
    \sup \Subposet^-(\PosetElement_{0})
    \lneq
    \ldots
    \lneq
    \sup \Subposet^-(\PosetElement_{i})
    \lneq
    \PosetElement_{i}
    \lneq
    \ldots
    \lneq
    \PosetElement_{q},
\end{equation}
  where
  $
  \Chain
  =
  \PosetElement_{0}
  \lneq
  \ldots
  \lneq
  \PosetElement_{q}
  \in \Poset_{q}
  $
  is a basis element in $\ChainComplex{q}{\Poset} = \Reals[\Poset_{q}]$ and all
  degenerate $(q+1)$-simplices that appear in the right hand term are defined
  to be equal to $0 \in \ChainComplex{q+1}{\Poset}$. One readily verifies that
  $\Homotopy_{\Poset}$ is indeed a chain homotopy between $\Inclusion_* \circ
  \Retraction_*$ and $\Identity_{\Poset}$. Furthermore, it follows from
  \autoref{eq:chain-homotopy} that $\Norm{\apply{\Homotopy_{\Poset}}{\Chain}}
  \leq (q+1) \Norm{\Chain}  =  q+1$ that for all $\Chain \in \Poset_q$.
  Therefore, $\Norm{\Homotopy_{\Poset}} \leq q+1$. This finishes the proof of
  \autoref{item-i-pdr}.

  Finally \autoref{item-ii-pdr} is a consequence of
  \autoref{cor:uniform-homotopy-equivalence}.
\end{proof}

\subsection{Cones and suspensions}

In many classical acyclicity arguments, one glues cones to other
semi-simplicial sets along semi-simplicial subsets or one needs to take
suspensions.
For this reason, we have to understand the uniform acyclicity properties of
cones and be able to describe how
suspension operations affect uniform acyclicity. Achieving this is the goal of this subsection.
\begin{lemma}
\label{lem:ConeUBC}
  Let $\SimplicialComplex_\bullet$ be a semi-simplicial set. Then the
  simplicial cone $\Cone{\SimplicialComplex_\bullet} =
  \SimplicialComplex_\bullet \ast \{c\}$ on $\SimplicialComplex_\bullet$ is
  uniformly acyclic, i.e.\ uniformly $n$-acyclic for all $n \in
  \NaturalNumbers$, and the UBC constant in every degree $q \geq 0$ can be
  chosen to be $1$. In particular, this applies to the star
  $
    \Star{\SimplicialComplex}{\Simplex{}}
  $
  of any simplex $\Simplex{}$ in an ordered simplicial complex $\SimplicialComplex$.
\end{lemma}
\begin{proof}
  We show that $\Cone{\SimplicialComplex}_\bullet$ is uniformly homotopy equivalent to the one-point semi-simplicial set $c_\bullet$ with $c_0 = \{c\}$ and $c_q = \emptyset$ for $q > 0$, i.e.\ that $\Cone{\SimplicialComplex}_\bullet$ is uniformly contractible. Let $f: \ChainComplex{*}{\Cone{\SimplicialComplex}_\bullet} \to \ChainComplex{*}{c_\bullet}$ be the chain map that is zero if $* >0$ and that, if $*  = 0$, maps every vertex $\Simplex{0} \in \Cone{\SimplicialComplex}_0$ to $c \in c_0$. Let $g: \ChainComplex{*}{c_\bullet} \to \ChainComplex{*}{\Cone{\SimplicialComplex}_\bullet}$ be the chain map that sends $c \in c_0$ to $c \in \Cone{\SimplicialComplex}_0$, then $f \circ g = \Identity_{c_\bullet}$ and we let $\Homotopy_{c_\bullet}: \ChainComplex{*}{c_\bullet} \to \ChainComplex{*+1}{c_\bullet}$ be the zero map. To define the homotopy $\Cone{\SimplicialComplex}_\bullet$ between $g \circ f$ and $\Identity_{\Cone{\SimplicialComplex}_\bullet}$, we note that the set of $q$-simplices of $\Cone{\SimplicialComplex}_\bullet$ can be written as
  $
    \Cone{\SimplicialComplex}_{q}
    =
    \SimplicialComplex_{q} \cup \{\Simplex{q-1} \ast c : \Simplex{q-1} \in \SimplicialComplex_{q-1}\}
  $
  . The chain contraction $\Homotopy_{\Cone{\SimplicialComplex}_\bullet}$ is then defined by sending a simplex $\Simplex{q}$ in
  $
    \SimplicialComplex_{q}
    \subset
    \Cone{\SimplicialComplex_{q}}
  $
  to $(-1)^q \cdot \Simplex{q} \ast c$ and the
  simplices $\Simplex{q-1} \ast c$ in $\Cone{\SimplicialComplex_{q}}\setminus
  \SimplicialComplex_{q}$ to zero.
  We observe that $\Norm{f} \leq 1$, $\Norm{g} \leq 1$, $\Norm{\Homotopy_{\Cone{\SimplicialComplex}_\bullet}} \leq 1$ and $\Norm{\Homotopy_{c_\bullet}} = 0$. Since $c_\bullet$ is uniformly acyclic with UBC constant equal to $0$ in all degrees $q \geq 0$, the claim then follows from \autoref{cor:uniform-homotopy-equivalence}.
\end{proof}
\begin{lemma} \label{lem:join-qubc}
  Let $\BoundarySimplex\StandardSimplex{n}_\bullet$ denote the simplicial boundary
  $(n-1)$-sphere of the standard simplex $\StandardSimplex{n}_\bullet$, i.e.\ the semi-simplicial
  set obtained from $\StandardSimplex{n}_\bullet$ by removing its unique $n$-simplex. If
  $\SimplicialComplex_\bullet$ is a uniformly $m$-acyclic semi-simplicial set with UBC constant in degree $q$
  denoted by $\UBCConstant{q}{\SimplicialComplex}$, then the suspension
  $\BoundarySimplex\StandardSimplex{n} \ast \SimplicialComplex$
  is uniformly $(n+m)$-acyclic and the UBC constant in degree $q$ can be chosen as
  \[
  	\UBCConstantSuspension{n-1}{q}{\UBCConstant{q-n}{\SimplicialComplex}} \coloneqq
    2^n
    \frac{(q+1)!}{(q+1-n)!}
    \UBCConstant{q-n}{\SimplicialComplex}
    +
    \sum_1^{n}
    2^{i}
    \frac{(q+1)!}{(q+2 - i)!}.
  \]
\end{lemma}
\begin{proof}
  We prove this lemma by induction on $n$. If $n = 0$, then $\BoundarySimplex
  \StandardSimplex{n}_\bullet \ast \SimplicialComplex_\bullet = \SimplicialComplex_\bullet$ and the
  bound holds trivially. For the induction step, assume that $n \geq 1$. Let
  $(\Horn{n}{n})_\bullet \subset \BoundarySimplex \StandardSimplex{n}_\bullet$ denote the
  $n$-horn of $\StandardSimplex{n}$ and $\FaceMap_n{\StandardSimplex{n}}_\bullet$ the $n$-th face of $\StandardSimplex{n}$ that is missing in $(\Horn{n}{n})_\bullet$ (considered as a semi-simplicial set).
  Then, we may cover the join $\BoundarySimplex \StandardSimplex{n}_\bullet \ast
  \SimplicialComplex_\bullet$ by
  \[
    \FaceMap_{n}\StandardSimplex{n}_\bullet \ast \SimplicialComplex_\bullet
    \text{ and }
    (\Horn{n}{n})_\bullet \ast \SimplicialComplex_\bullet.
  \]
  The intersection $\FaceMap_{n}\StandardSimplex{n}_\bullet \ast \SimplicialComplex_\bullet \cap (\Horn{n}{n})_\bullet \ast \SimplicialComplex_\bullet$ is given by $\BoundarySimplex
  \StandardSimplex{n-1}_\bullet \ast \SimplicialComplex_\bullet$, which is
  assumed to be $(n - 1 + m)$-uniformly acyclic and the UBC constant in degree $q$
  can be chosen as
  \[
    \UBCConstant{q}{\BoundarySimplex
    	\StandardSimplex{n-1}_\bullet \ast \SimplicialComplex_\bullet} \leq 2^{n-1}
    \frac{(q+1)!}{(q+1-(n-1))!} \UBCConstant{q-(n-1)}{\SimplicialComplex} +
    \sum_1^{n-1} 2^{i} \frac{(q+1)!}{(q+2 - i)!}
  \]
  by the induction hypothesis. Note that both $\FaceMap_{n}\StandardSimplex{n}_\bullet \ast
  \SimplicialComplex_\bullet$ and $(\Horn{n}{n})_\bullet \ast X_\bullet$ are
  simplicial cones. Therefore, \autoref{lem:ConeUBC} implies that both
  are uniformly acyclic and that the UBC constants in degree $q$ can be chosen as
  $\UBCConstant{q}{\FaceMap_{n}\StandardSimplex{n} \ast  \SimplicialComplex} = \UBCConstant{q}{\Horn{n}{n} \ast X}  = 1$.
  The claim then follows by applying the Mayer--Vietoris principle (see \autoref{lem:MV-qubc}) to the covering of $\BoundarySimplex \StandardSimplex{n}_\bullet \ast
  \SimplicialComplex_\bullet$ by $\FaceMap_{n}\StandardSimplex{n}_\bullet \ast \SimplicialComplex_\bullet$ and $(\Horn{n}{n})_\bullet \ast \SimplicialComplex_\bullet$:
\begin{align*}
  &(1 + (q+1)\UBCConstant{q-1}{\BoundarySimplex
  	\StandardSimplex{n-1} \ast \SimplicialComplex})
  (\UBCConstant{q}{\FaceMap_{n}\StandardSimplex{n} \ast  \SimplicialComplex} + \UBCConstant{q}{\Horn{n}{n} \ast X})\\
  &\leq (1 + (q+1)\UBCConstant{q-1}{\BoundarySimplex \StandardSimplex{n-1} \ast \SimplicialComplex}) (1+1) = 2 +
  2(q+1)\UBCConstant{q-1}{\BoundarySimplex \StandardSimplex{n-1} \ast \SimplicialComplex}\\
  &\leq 2 + 2(q+1)(2^{n-1} \frac{q!}{(q-(n-1))!}
  \UBCConstant{(q-1)-(n-1)}{\SimplicialComplex} + \sum_1^{n-1} 2^{i}
  \frac{q!}{(q+1 - i)!})\\
  &= 2 + (2^{n} \frac{(q+1)!}{(q-(n-1))!}
  \UBCConstant{(q-1)-(n-1)}{\SimplicialComplex} + \sum_1^{n-1} 2^{i+1}
  \frac{(q+1)!}{(q+1 - i)!})\\
  &= 2^n \frac{(q+1)!}{(q+1-n)!} \UBCConstant{q-n}{\SimplicialComplex} +
  \sum_1^{n} 2^{i} \frac{(q+1)!}{(q+2 - i)!}. \qedhere
\end{align*}
\end{proof}

\subsection{Barycentric subdivisions}

It is well-known that the classical simplicial (co-)homology of a regular
semi-simplicial
sets is preserved under subdivisions (see
\autoref{rem:geometric-realization-of-regular-semi-simplicial-sets}). In this
subsection, we show that this also holds for simplicial bounded cohomology.
\begin{lemma}
\label{lem:BarycentricSubdivision}
  Let $\SimplicialComplex_{\bullet}$ be a regular semi-simplicial set. Then
  the simplicial chain complexes of $\SimplicialComplex_{\bullet}$ and its barycentric subdivision $\BarycentricSubdivion{\SimplicialComplex_{\bullet}}$ 
  are uniformly homotopy equivalent. In particular, the simplicial bounded cohomology of 
  $\SimplicialComplex_{\bullet}$ and $\BarycentricSubdivion{\SimplicialComplex_{\bullet}}$ are isomorphic, 
  and $\SimplicialComplex_{\bullet}$ is uniformly $n$-acyclic if and only if $\BarycentricSubdivion{\SimplicialComplex_{\bullet}}$ is uniformly $n$-acyclic with UBC constants related by
  \[
  \UBCConstant{q}{\SimplicialComplex}
  \leq
  \Factorial{(q+1)}
  \cdot
  \UBCConstant{q}{\BarycentricSubdivion{\SimplicialComplex}}
  \text{ and }
  \UBCConstant{q}{\BarycentricSubdivion{\SimplicialComplex}}
  \leq
  \Factorial{(q+2)}
  +
  \Factorial{(q+1)}
  \cdot
  \UBCConstant{q}{\SimplicialComplex}.
  \]
\end{lemma}
\begin{proof}
  We first note that by \autoref{def:barycentric-subdivision} the set of $q$-simplices $\BarycentricSubdivion{\SimplicialComplex_q}$ of $\BarycentricSubdivion{\SimplicialComplex_\bullet}$ consists of all proper
  flags
  $
  \Simplex{i_0}
  \lneq
  \ldots
  \lneq
  \Simplex{i_q}
  $
  of simplices of $\SimplicialComplex_\bullet$, where $\leq$ denotes the face relation in $\SimplicialComplex_\bullet$.

  We now start working towards the definition of a chain map $\ContinuousMap$ from $\ChainComplex{*}{\SimplicialComplex_\bullet}$ to $\ChainComplex{*}{\BarycentricSubdivion{\SimplicialComplex_\bullet}}$:
  Since $\SimplicialComplex_\bullet$ is regular, every $q$-simplex $\Simplex{q}$ of $\SimplicialComplex_\bullet$ has $q+1$ distinct
  and ordered vertices $(v_0, \dots, v_q)$. However, the set of vertices might
  not determine $\Simplex{q}$ uniquely,
  since $\SimplicialComplex_\bullet$ is not assumed to be an ordered simplicial complex.
  The $q$-simplex $\Simplex{q}$ gives rise to
  $\Factorial{(q+1)}$-many $q$-simplices in the barycentric subdivision $\BarycentricSubdivion{\SimplicialComplex_\bullet}$:
  Each reordering $(v_{\apply{\phi}{0}}, \ldots, v_{\apply{\phi}{q}})$ of its $q+1$ distinct vertices corresponds to
  a unique flag of $q+1$ faces of $\Simplex{q}$,
  $$[\Simplex{q} : v_{\apply{\phi}{0}}, \ldots, v_{\apply{\phi}{q}}] \coloneqq \FaceMap_{{\apply{\phi}{1}}} \dots \FaceMap_{{\apply{\phi}{q}}} \Simplex{q} \lneq \FaceMap_{{\apply{\phi}{2}}} \dots \FaceMap_{{\apply{\phi}{q}}} \Simplex{q} \lneq \dots \lneq \Simplex{q} \in \BarycentricSubdivion{\SimplicialComplex_q}.$$
  In degree $q$, the chain map $\ContinuousMap$ is then defined by
  \begin{equation}
  	\label{eq:chain-map-barycentric}
  	\Simplex{q} \mapsto \sum_{\phi \in \apply{\on{Sym}}{\{0, \dots, q+1\}}} \apply{\on{sign}}{\phi} \cdot [\Simplex{q} : v_{\apply{\phi}{0}}, \ldots, v_{\apply{\phi}{q}}]
  \end{equation}
  Note that $\Norm{\ContinuousMap} \leq \Factorial{(q+1)}$ in degree $q$.

  We now construct a homotopy inverse $\ContinuousMapALT$ for the chain map
  $\ContinuousMap$:
  Given a $k$-simplex $\Simplex{k}$ in $\SimplicialComplex_k$, we denote by
  $\max \Simplex{k} = v_k \in \SimplicialComplex_{0}$ the last element in its
  ordered vertex set $(v_0, \dots, v_k)$.
  A $q$-simplex of the barycentric subdivision  $\Simplex{i_0} \lneq \dots
  \lneq \Simplex{i_q}$, then gives rise to an ordered set
  $V = (\max \Simplex{i_0}, \ldots, \max \Simplex{i_q})$ of elements in $\SimplicialComplex_{0}$ (possibly with repetitions).
  In degree $q$, the chain map $\ContinuousMapALT$ sends $\Simplex{i_0} \lneq
  \dots \lneq \Simplex{i_q}$
  to the unique $q$-dimensional face of $\Simplex{i_q}$ that has $V$ as its ordered set of vertices if it exists or to the zero element in $\ChainComplex{q}{\SimplicialComplex_\bullet}$ if $\Simplex{i_q}$ does not have such a face (i.e.\ if $V$ has repetitions or if the ordering of $V$ is not compatible with the ordering of the vertex set of $\Simplex{i_q}$).
  Note that $\Norm{\ContinuousMapALT} \leq 1$ in all degrees $q$.

  To see that $\ContinuousMap$ and $\ContinuousMapALT$ are homotopy inverses of each other, we first observe
  that $\ContinuousMapALT \circ \ContinuousMap = \Identity_{\SimplicialComplex_\bullet}$: The term indexed by $\Identity \in \apply{\on{Sym}}{\{0, \dots, q\}}$ in the sum on the right of \autoref{eq:chain-map-barycentric} is mapped to $\Simplex{q}$ by $\ContinuousMapALT$, and all other terms map to zero. In particular, we can pick $\Homotopy_{\SimplicialComplex_\bullet}$ to be the zero map. We are left with checking that $\ContinuousMap \circ \ContinuousMapALT$ is chain homotopic to
  the $\Identity_{\BarycentricSubdivion{\SimplicialComplex_\bullet}}$ via a bounded chain homotopy $\Homotopy_{\BarycentricSubdivion{\SimplicialComplex_\bullet}}$. To do this, we proceed inductively: For a vertex $\Simplex{}$ in the barycentric subdivision, we define
  $\apply{\Homotopy_{\BarycentricSubdivion{\SimplicialComplex_\bullet}}}{\Simplex{}}$ to be the $1$-simplex $\max \Simplex{}
  \lneq \Simplex{}$ if $\Simplex{}$ is not a vertex of $\SimplicialComplex_\bullet$, and zero otherwise.
  Let $q \geq 1$ and suppose that $\Homotopy_{\BarycentricSubdivion{\SimplicialComplex_\bullet}}$ is defined in degrees $k < q$ such that
  $\Homotopy_{\BarycentricSubdivion{\SimplicialComplex_\bullet}}(\overline{\Simplex{}})$ is supported in $\BarycentricSubdivion{\Simplex{i_{k}}}$ for all
  $k$-simplices $\overline{\Simplex{}} = \Simplex{i_0} \lneq \ldots \lneq \Simplex{i_k}$ and such that $\Norm{\Homotopy_{\BarycentricSubdivion{\SimplicialComplex_\bullet}}} \leq \Factorial{(k+2)}-1$ in degree $k$. Consider a
  $q$-simplex
  $
  \overline{\Simplex{}}
  =
  \Simplex{i_0}
  \lneq
  \ldots
  \lneq
  \Simplex{i_q}.
  $
  Then it holds that
  $
  \overline{\Simplex{}}
  -
  \apply
  {(\ContinuousMap \circ \ContinuousMapALT)}
  {\overline{\Simplex{}}}
  -
  \apply{\Homotopy_{\BarycentricSubdivion{\SimplicialComplex_\bullet}}}{\BoundarySimplex \overline{\Simplex{}}}
  $
  is a closed $q$-chain supported in $\BarycentricSubdivion{\Simplex{i_q}}$. Since
  $\BarycentricSubdivion{\Simplex{i_q}}$ is a cone, \autoref{lem:ConeUBC}
  implies that it is boundedly acyclic in all degrees with UBC constant $1$. Hence,
  $
  \overline{\Simplex{}}
  -
  \apply
  {(\ContinuousMap \circ \ContinuousMapALT)}
  {\overline{\Simplex{}}}
  -
  \apply{\Homotopy_{\BarycentricSubdivion{\SimplicialComplex_\bullet}}}{\BoundarySimplex \overline{\Simplex{}}}
  $
  is the boundary a chain $\BoundingChain$ supported in $\BarycentricSubdivion{\Simplex{i_q}}$ of norm $\Norm{\BoundingChain}$ at most
  \begin{align*}
  	\Norm{\Identity_{\BarycentricSubdivion{\SimplicialComplex_\bullet}}} + \Norm{\ContinuousMap \circ \ContinuousMapALT} + \Norm{\Homotopy_{\BarycentricSubdivion{\SimplicialComplex_\bullet}} \circ \BoundarySimplex_q}
  	&
  	\leq
    1 + \Factorial{(q+1)} + (\Factorial{(q+1)}-1)(q+1)
    \\
    &
    =
    \Factorial{(q+2)} - q
    \leq \Factorial{(q+2)} - 1.
  \end{align*}

  This completes the proof the the first part of the lemma. The second part is a consequence of \autoref{cor:uniform-homotopy-equivalence}, using that the constructed uniform homotopy equivalence $(\ContinuousMap,
  \ContinuousMapALT, \Homotopy_{\SimplicialComplex_\bullet},
  \Homotopy_{\BarycentricSubdivion{\SimplicialComplex_\bullet}})$ satisfies
  $\Norm{\ContinuousMap} \leq \Factorial{(q+1)}$, $\Norm{\ContinuousMapALT}
  \leq 1$, $\Norm{\Homotopy_{\SimplicialComplex_\bullet}} = 0$ and
  $\Norm{\Homotopy_{\BarycentricSubdivion{\SimplicialComplex_\bullet}}} \leq
  \Factorial{(q+2)} - 1 \leq \Factorial{(q+2)}$ in degree $q$.
\end{proof}

\subsection{Infinite unions}
In contrast to classical (co-)homology with field coefficients, simplicial
bounded cohomology does not behave well with respect to colimits and infinite
unions. This has been mentioned e.g.\
in the context of \autoref{expl:uniformity-condition} at the end of
\autoref{scn:SBC}. Our next example illustrates that an infinite nested union of bounded
$\infty$-acyclic simplicial complexes might not be bounded
$\infty$-acyclic (even though in the setting of classical (co-)homology this is true).

\begin{example}
	\label{expl:directed-unions}
	Let $\SimplicialComplex_\bullet$ be the real line equipped with the standard semi-simplicial structure,
	i.e.\ as in \autoref{exm:SBCofR}. 
	Let $\Subcomplex^i_\bullet = [-i, i]_\bullet \subset \SimplicialComplex_\bullet$ 
	denote the interval from $-i$ to $i$ equipped with the usual simplicial structure.
	Then $\Subcomplex^i_\bullet$ is a finite simplicial complex that is contractible. 
	Finiteness implies that the reduced simplicial bounded cohomology agrees with the usual reduced 
	simplicial cohomology of $\Subcomplex^i_\bullet$, which is trivial in all degrees. 
	Hence, $\Subcomplex^i_\bullet$ is bounded $\infty$-acyclic. However, 
	$$\SimplicialComplex_\bullet = \bigcup_{i \in \NaturalNumbers} \Subcomplex^i_\bullet,$$
	is not bounded $\infty$-acyclic as we have seen in \autoref{exm:SBCofR}.
\end{example}

However, the following lemma shows that uniform acyclicity is preserved under
unions if one imposes a ``global'' uniformity requirements.
\begin{lemma} \label{lem:DirectedUnion}
  Let $(\Subcomplex^i_\bullet)_{i \in \NaturalNumbers}$ be a nested sequence of uniformly
  $n$-acyclic semi-simplicial sets such that their respective UBC constants satisfy
  $\UBCConstant{q}{\Subcomplex^i} \leq K$ for some $K\in \Reals$ and all $q \leq n$.
  Then
  \[
    \SimplicialComplex_\bullet
    =
    \bigcup_{i \in \NaturalNumbers}
    \Subcomplex^i_\bullet
  \]
 is uniformly $n$-acyclic and the UBC constant can be chosen as $\UBCConstant{q}{\SimplicialComplex} = K$ for $q \leq n$.
\end{lemma}
\begin{proof}
  Let $q \leq n$. Every $q$-chain $\Chain$ in $\ReducedChainComplex{q}{\SimplicialComplex_\bullet}$ is supported in $\ReducedChainComplex{q}{\Subcomplex^i_\bullet}$ for some $i \in \NaturalNumbers$. If $\BoundaryChainComplex{} \Chain = 0$, 
  then it follows from the assumptions that it is the boundary of a $(q+1)$-chain in $\ReducedChainComplex{q+1}{\Subcomplex^i_\bullet} \subseteq \ReducedChainComplex{q+1}{\SimplicialComplex_\bullet}$ of norm at most $K \Norm{\Chain}$.
\end{proof}

In light of \autoref{rem:0UBC}, \autoref{expl:directed-unions} violates the
global uniform boundary condition on the
filtration in \autoref{lem:DirectedUnion} in degree 0: The diameter of
$\Subcomplex^i_\bullet = [-i, i]$ is $2i$ and diverges to
infinity as $i$ goes to infinity. Hence, the $0$-UBC constants
$\UBCConstant{0}{\Subcomplex^i}$ do not admit a global bound
$\UBCConstant{}{} \in \Reals$.

\subsection{Uniform discrete Morse theory}

A widely used strategy for proving high-acyclicity results for the classical
simplicial (co-)homology of a
simplicial complex or poset $\SimplicialComplex_\bullet$ can be summarized under the title ``discrete Morse theory'' (see e.g.\ \cite{BestvinaPLMorseTheory}):
The idea is to find a suitable way of constructing $\SimplicialComplex_\bullet$
from a subcomplex $\Subcomplex_\bullet$ (which is well-understood e.g.\ by induction)
 by attaching (possibly infinitely many) cells of a fixed dimension (compare \autoref{def:PosetObtainedByGluing}).
 The key point is that if one knows that e.g.\ $\Subcomplex_\bullet$ and the link of the attached cells are highly acyclic,
  then it follows that $\SimplicialComplex_\bullet$ is highly acyclic as well. As seen in the previous subsection,
 simplicial bounded cohomology does not behave well with respect to colimits (and hence cell attachments). To get uniform analogues of discrete Morse theory techniques
we therefore need to impose a global uniformity condition. This is the content of the next lemma and its corollary.

\begin{lemma} \label{lem:LinkAttachments}
  Let $\SimplicialComplex_\bullet$ be an ordered simplicial complex that is obtained from
  $\Subcomplex_\bullet \subseteq \SimplicialComplex_\bullet$ by attaching
  (possibly infinitely many) $p$-simplices $\{\Simplex{p}^{i}\}_{i \in I}$ along
  their links
  $
    \Link{\SimplicialComplex_\bullet}{\Simplex{p}^{i}}.
  $
  Assume there exist a constant $\UBCConstant{}{\Linkfunction} \geq 0$
  such that for each $i \in I$ it holds that
  $\Link{\SimplicialComplex_\bullet}{\Simplex{p}^{i}}$ is uniformly $(n-p-1)$-acyclic
  and the UBC constant in degree $q \leq n - p - 1$ can be chosen as
  $\UBCConstant{}{\Linkfunction}$.
  Then the
  inclusion
  $\Subcomplex_\bullet \hookrightarrow \SimplicialComplex_\bullet$
  is uniformly $n$-acyclic and the UBC constant in degree $q$ can be chosen as
  \[
    \UBCConstantCellAttachment{p}{q}{\UBCConstant{}{\Linkfunction}} \coloneqq 1 + (q+1) \cdot \UBCConstantSuspension{p-1}{q}{\UBCConstant{}{\Linkfunction}},
  \]
  where $\UBCConstantSuspension{-}{-}{-}$ is the function introduced in \autoref{lem:join-qubc}.
\end{lemma}
\begin{proof}
	Let $q \leq n$. By \autoref{obs:InclusionAcyclic}, we need to check that if
	$\Chain \in \ChainComplex{q}{\SimplicialComplex_{\bullet}}$
	is a chain with
	$
	\partial \Chain \in \ChainComplex{q-1}{\Subcomplex_\bullet}
	$,
	then there exists a chain
	$
	\BoundingChain
	\in
	\ChainComplex{q+1}{\SimplicialComplex_\bullet}
	$
	such that
	$
	\BoundaryChainComplex{} \BoundingChain
	=
	\Chain
	+
	\alpha
	$, 
	where $\alpha$ is contained in
	$\ChainComplex{q}{\Subcomplex_{\bullet}}$ and
	$
	\Norm{\BoundingChain}
	\leq
	\UBCConstant{q}{\SimplicialComplex,\Subcomplex}
	\Norm{\Chain}
	$
	for $\UBCConstant{q}{\SimplicialComplex,\Subcomplex} = 1 + (q+1) \cdot \UBCConstantSuspension{p-1}{q}{\UBCConstant{}{\Linkfunction}}$.
	Since $\SimplicialComplex_\bullet = \Subcomplex_\bullet \cup (\bigcup_i \Star{\SimplicialComplex_\bullet}{\Simplex{p}^i})$
	the chain $\Chain$ can be written as
	$
	\Chain
	=
	\Chain_{\Subcomplex}
	+
	\sum_{i}
	\Chain_{i}
	$, 
	where $\Chain_{\Subcomplex} \in \ChainComplex{q}{\Subcomplex_\bullet}$,
	$\Chain_{i} \in \ChainComplex{q}{\Star{\SimplicialComplex_\bullet}{\Simplex{p}^i}}$,
	$
	\Norm{\Chain}
	=
	\Norm{\Chain_{\Subcomplex}}
	+
	\sum_{i}
	\Norm{\Chain_{i}}
	$
	and all but finitely many sum terms are zero.
	(This decomposition is not unique, in general.)
	Since $\BoundaryChainComplex{} \Chain \in \ChainComplex{q-1}{\Subcomplex_\bullet}$, it follows that
	$\BoundaryChainComplex{} \Chain_{\Subcomplex} \in  \ChainComplex{q-1}{\Subcomplex_\bullet}$
	and $\BoundaryChainComplex{} \Chain_{i} \in  \ChainComplex{q-1}{\Subcomplex_\bullet}$ for all $i$.
	Therefore, $\BoundaryChainComplex{} \Chain_{i} \in  \ChainComplex{q-1}{\Subcomplex_\bullet \cap \Star{\SimplicialComplex_\bullet}{\Simplex{p}^i}}$.
	Noting that $\Subcomplex_\bullet \cap \Star{\SimplicialComplex_\bullet}{\Simplex{p}^i} = \partial \Simplex{p}^i \ast \Link{\SimplicialComplex_\bullet}{\Simplex{p}^i}$
	is a $(p-1)$-suspension, the assumptions and \autoref{lem:join-qubc} imply that this complex is uniformly $(n-1)$-acyclic with UBC constant in degree $q \leq n-1$ given by
	$\UBCConstantSuspension{p-1}{q}{\UBCConstant{}{\Linkfunction}}$.
	Hence, for each $i \in I$ there exists a chain
	$\alpha_{i} \in \ChainComplex{q}{\Subcomplex_\bullet \cap \Star{\SimplicialComplex_\bullet}{\Simplex{p}^i}}$
	satisfying $\BoundaryChainComplex{} \alpha_i = \BoundaryChainComplex{} \Chain_i$
	and
	\[
	\Norm{\alpha_{i}}
	\leq
	\UBCConstantSuspension{p-1}{q}{\UBCConstant{}{\Linkfunction}} \Norm{\BoundaryChainComplex{} \Chain_{i}}
	\leq
	(q+1) \cdot \UBCConstantSuspension{p-1}{q}{\UBCConstant{}{\Linkfunction}} \Norm{\Chain_i}.
	\]
	Set $\alpha = \sum_i \alpha_i$ and note that $\alpha \in \ChainComplex{q}{\Subcomplex_\bullet}$.
	Now, consider the chains $\Chain_{i} - \alpha_{i} \in
	\ChainComplex{q}{\Star{\SimplicialComplex_\bullet}{\Simplex{p}^i}}$. These
	satisfy $\BoundaryChainComplex{}(\Chain_{i} - \alpha_{i})  = 0$ and,
	therefore, \autoref{lem:ConeUBC} implies that there exists
	$\BoundingChain_{i} \in
	\ChainComplex{q+1}{\Star{\SimplicialComplex_\bullet}{\Simplex{p}^i}}$ such
	that $\BoundaryChainComplex{} \BoundingChain_i = \Chain_{i} - \alpha_{i}$ and
	\[
	\Norm{\BoundingChain_{i}}
	\leq
	\Norm{\Chain_{i} - \alpha_{i}}
	\leq
	\Norm{\Chain_{i}} + \Norm{\alpha_{i}}
	\leq
	(1 + (q+1) \cdot \UBCConstantSuspension{p-1}{q}{\UBCConstant{}{\Linkfunction}}) \Norm{\Chain_{i}}.
	\]
	As a consequence, the chain $\BoundingChain = \sum_{i} \BoundingChain_{i} \in \ChainComplex{q+1}{\SimplicialComplex_\bullet}$ satisfies $\BoundaryChainComplex{} \BoundingChain = \sum_{i} \Chain_{i} - \sum_i \alpha_i = \Chain - \alpha$
	and, using $\sum_{i} \Norm{\Chain_{i}} \leq \Norm{\Chain}$ in the second step,
	\[
	\Norm{\BoundingChain}
	\leq
	(1 + (q+1) \cdot \UBCConstantSuspension{p-1}{q}{\UBCConstant{}{\Linkfunction}}) \cdot \sum_i \Norm{\Chain_{i}}
	\leq
	(1 + (q+1) \cdot \UBCConstantSuspension{p-1}{q}{\UBCConstant{}{\Linkfunction}}) \cdot \Norm{\Chain}
	\]
	as claimed.
\end{proof}

As a consequence we obtain the following uniform discrete Morse theory
principle.

\begin{corollary}
	\label{cor:uniform-discrete-morse-theory}
	In the setting of \autoref{lem:LinkAttachments}, assume additionally that
	$\Subcomplex_\bullet$ is uniformly $n$-acyclic with UBC constant in degree $q
	\leq n$ given by $\UBCConstant{q}{\Subcomplex}$. Then
	$\SimplicialComplex_\bullet$ is uniformly $n$-acyclic and its UBC constant in
	degree $q$ can be chosen as
	\[
	\UBCConstantMorseTheory{p}{q}{\UBCConstant{q}{\Subcomplex}}{\UBCConstant{}{\Linkfunction}}
	\coloneqq
	\UBCConstantTwoOutOfThreeOne{q}{\UBCConstant{q}{\Subcomplex}}{\UBCConstantCellAttachment{p}{q}{\UBCConstant{}{\Linkfunction}}}.
	\]
\end{corollary}
\begin{proof}
	This follows from \autoref{item-3-toot} of \autoref{lem:two-out-of-three-property} and \autoref{lem:LinkAttachments}.
\end{proof}

The next example illustrates what goes wrong if the global uniformity condition in
\autoref{lem:LinkAttachments} on the links is violated for $q = 0$.

\begin{example}
	\label{expl:attaching-infinitely-many-cones}
	Let $\SimplicialComplex_\bullet$ denote real line equipped with the usual simplicial structure 
	as in \autoref{exm:SBCofR}. Let $\Cone{\SimplicialComplex_\bullet} = \SimplicialComplex_\bullet \ast c$ denote the simplicial cone 
	on $\SimplicialComplex_\bullet$ obtained by adding the cone point $c$ and ordering it last in each new simplex. \autoref{lem:ConeUBC} 
	implies that $\Cone{\SimplicialComplex_\bullet}$ is bounded $\infty$-acyclic. Let $\SimplicialComplex_\bullet'$ denote the ordered simplicial complex obtained from 
	$\Cone{\SimplicialComplex_\bullet}$ by attaching, for each $i \in \NaturalNumbers$, a new point 
	$c_i$ along the simplicial interval $\Subcomplex^i_\bullet = [-i, i]_\bullet \subset \SimplicialComplex_\bullet \subset 
	\Cone{\SimplicialComplex_\bullet}$ and ordering $c_i$ last in each new simplex.
	By definition, it holds that 
	$$\Link{\SimplicialComplex_\bullet'}{c_i} = [-i,i]_\bullet.$$
	Therefore, $\Link{\SimplicialComplex_\bullet'}{c_i}$ is bounded 
	$\infty$-acyclic for each $i \in \NaturalNumbers$.
	Classical discrete Morse theory implies that $\SimplicialComplex_\bullet'$ is $\infty$-acyclic,
	in fact it is easy to see that $\SimplicialComplex_\bullet'$ is contractible.	
	
	Nevertheless, we claim that $\BoundedCohomologyOfSimplicialObject{2}{\SimplicialComplex_\bullet'}{\Reals} \neq 
	0$, i.e.\ that $\SimplicialComplex_\bullet'$ is not bounded $\infty$-acyclic as one might initially expect.
	To see this, we first show that $\SimplicialComplex_\bullet'$ does not satisfy the 1-UBC 
	condition: Consider the simplicial ``loop''
	$$\Chain^i = [i, c] - [i, c_i] + [-i, c_i] - [-i, c]$$
	in $\ChainComplex{1}{\SimplicialComplex_\bullet'}$ for each $i \in \NaturalNumbers$.
	This $1$-cycle is visibly a boundary, i.e.\ $\Chain^i \in 
	\apply{\BoundarySimplex}{\ChainComplex{2}{\SimplicialComplex_\bullet'}}$,  
	and satisfies $\Norm{\Chain^i} =  4$. However, the $\ellone$-norm 
	of any 2-chain $\BoundingChain^i \in \ChainComplex{2}{\SimplicialComplex}$ 
	with $\apply{\BoundarySimplex}{\BoundingChain^i} = \Chain^i$ satisfies 
	$\Norm{\BoundingChain^i} \geq 4i$ which diverges to 
	infinity as $i$ goes to infinity.  Hence, $\SimplicialComplex_\bullet'$ does not 
	satisfy the 1-UBC.	
	This implies that bounded cohomology of 
	$\SimplicialComplex_\bullet'$ is nontrivial in degree $2$, because an application of
	\autoref{prop:MatsumotoMorita} shows that the comparison map 
	$\BoundedCohomologyOfSimplicialObject{2}{\SimplicialComplex_\bullet'}{\Reals} \to \CohomologyOfSpaceObject{2}{\SimplicialComplex_\bullet'}{\Reals} = 0$ is not 
	injective.
\end{example}

We note that in \autoref{expl:attaching-infinitely-many-cones} the diameter of
$\Link{\SimplicialComplex_\bullet'}{c_i} = [-i, i]_\bullet$ diverges to 
infinity as $i$ approaches to infinity. Hence, \autoref{rem:0UBC} implies that the $0$-UBC constants 
$\UBCConstant{0}{\Link{\SimplicialComplex_\bullet'}{c_i}}$ do not admit a global bound 
$\UBCConstant{0}{\Linkfunction} \in \Reals$ and that the uniformity condition in \autoref{lem:LinkAttachments} is violated.

\subsection{Cohen--Macaulay complexes and injective words}

Cohen--Macaulay complexes are a class of simplicial complexes with particularly good local and global connectivity properties.
They naturally arise in the context of e.g.\ homological stability arguments (see e.g.\ \cites{randalwilliamswahl2017homologicalstabilityforautomorphismgroups, hatchervogtmann2017, GalatiusRandalWilliamsStability}). In this subsection,
we discuss the uniformly bounded analogue of this class of complexes as well as related semi-simplicial sets.

\begin{definition}
	Let $n \geq 0$. A simplicial complex $\SimplicialComplex$ is called 
	\introduce{uniformly bounded Cohen--Macaulay} of level $n$ if there exists a constant $\UBCConstant{}{} \in \Reals$ such that 
	$\SimplicialComplex$ is uniformly $(n-1)$-acyclic and satisfies the q-UBC 
	for all $q \leq n-1$ with constant $\UBCConstant{}{}$ and the link 
	$\Link{\SimplicialComplex}{\Simplex{k}}$ of each $k$-simplex $\Simplex{k}$ 
	is uniformly $(n-k-2)$-acyclic and satisfies the q-UBC condition for all $q 
	\leq n - k - 2$ with constant $\UBCConstant{}{}$. We write $\CMub(\SimplicialComplex) \geq n$. Similarly, we say that $\SimplicialComplex$ is \introduce{locally uniformly bounded Cohen--Macaulay} of level $n$ if there exists a constant $\UBCConstant{}{} \in \Reals$ such that the link 
	$\Link{\SimplicialComplex}{\Simplex{k}}$ of each $k$-simplex $\Simplex{k}$ 
	is uniformly $(n-k-2)$-acyclic and satisfies the q-UBC condition for all $q 
	\leq n - k - 2$ with constant $\UBCConstant{}{}$. We write $\lCMub(\SimplicialComplex) \geq n$ in this case.
\end{definition}

We record an observation, which can be checked exactly as in \cite[Paragraph 5, p.1887]{hatchervogtmann2017}.

\begin{observation}
	\label{obs:CM-localness}
	If $\SimplicialComplex$ is uniformly bounded Cohen--Macaulay of level $n$ with UBC constant $\UBCConstant{}{}$ and $\Simplex{p}$ is a $p$-simplex of $\SimplicialComplex$, then $\Link{\SimplicialComplex}{\Simplex{p}}$ is uniformly bounded Cohen--Macaulay of level $n - p - 1$ with UBC constant $\UBCConstant{}{}$.
\end{observation}

The following is a generalization of the construction of the complex of
injective words considered for example by Farmer \cite{farmer1978} and used by
Kerz \cite{kerz2005} to give a new
proof of the fact that symmetric groups satisfy (classical) homological stability.

\begin{definition}
	\label{def:ord-complex}
	Let $\SimplicialComplex$ be a simplicial complex. Then we can associate a 
	semi-simplicial set $\Ord{\SimplicialComplex}{\bullet}$ whose 
	$k$-simplices are ordered $(k+1)$-tuples $(x_0, \dots, x_k)$ such that the 
	unordered set $\{x_0, \dots, x_k\}$ is a $k$-simplex of 
	$\SimplicialComplex$. The $i$-th face of a $k$-simplex $(x_0, \dots, x_k) 
	\in \Ord{\SimplicialComplex}{k}$ is given by omitting $x_i$.
\end{definition}

It is a well-known fact that if a simplicial complex $\SimplicialComplex$ is homotopy Cohen--Macaulay of level $n$, then $\Ord{\SimplicialComplex}{\bullet}$ is $(n-1)$-connected (see e.g.\ \cite[Proposition 2.14]{randalwilliamswahl2017homologicalstabilityforautomorphismgroups} or \cite[Proposition 2.10]{hatchervogtmann2017}). The next lemma shows that a similar statement holds true in the uniformly bounded setting.

\begin{lemma}
	\label{lem:ord-complex-connectivity-lemma}
	There exists a function $\on{K^{ord}} : \NaturalNumbers \times \Reals \to \Reals$ such that
	if $\SimplicialComplex$ is a simplicial complex which is uniformly bounded Cohen--Macaulay of level $n$ with UBC constant $\UBCConstant{}{}$,
	then the associated semi-simplicial set $\Ord{\SimplicialComplex}{\bullet}$ is uniformly
	$(n-1)$-acyclic and its UBC constant in degree $q \leq n-1$ can be chosen as
	$\UBCConstantOrdConstruction{n}{\UBCConstant{}{}}.$
\end{lemma}

The proof of this lemma is a uniformly bounded refinement of an argument due to
Hatcher--Vogtmann \cite[Proposition 2.10]{hatchervogtmann2017}. Since this argument
is an adaptation, the structure and presentation are similar to
\cite[Proposition 2.10]{hatchervogtmann2017}.

\begin{proof}
	We first note that we may assume that
	$\SimplicialComplex$ has dimension $n$ because simplices
	of dimension $d > n$ do not influence uniform $(n-1)$-acyclicity.

	For $n = 0$ the claim holds trivially and we define
	$\UBCConstantOrdConstruction{0}{\UBCConstant{}{}} \coloneqq 0$. Hence, let $n
	> 0$ and assume for
	an induction on $n$ that we already constructed a function $\on{K^{ord}}: \{0, \dots, n-1\} \times \Reals \to \Reals$ with the claimed property for
	$\CMub(\SimplicialComplex) < n$. For the induction step we first verify the following claim.

\begin{claim}
	There exists a function $\on{K^{ord}_{St}}: \NaturalNumbers \times \Reals \to \Reals$ such that in the setting of \autoref{lem:ord-complex-connectivity-lemma} the following holds: Let $\Simplex{}$ be any simplex in
	$\SimplicialComplex$, then
	$
    \Ord
      {
          \Star
            {\SimplicialComplex}
            {\Simplex{}}
      }
      {\bullet}
  $  is
	uniformly
	$(n-1)$-acyclic with constant $\UBCConstantOrdConstructionStar{n}{\UBCConstant{}{}}$.
\end{claim}

\begin{proof}[Proof of Claim]
	For $n = 0$, the claim holds trivially and we define
	$\UBCConstantOrdConstructionStar{0}{\UBCConstant{}{}} \coloneqq 0$. Hence,
	let us assume that $n > 0$.

	Let $a \in \Simplex{}$ be a vertex, and write $\Simplex{} = \{a\} \sqcup
	\tau$. Then $\Star{\SimplicialComplex}{\Simplex{}} = \{a\} \ast (\tau \ast
	\Link{\SimplicialComplex}{\Simplex{}}) = \{a\} \ast \Subcomplex$ is a simplicial cone
	over $\Subcomplex \coloneqq \tau \ast
	\Link{\SimplicialComplex}{\Simplex{}}$. \autoref{obs:CM-localness} implies that $\Link{\SimplicialComplex}{\Simplex{}}$ is a uniformly
	bounded Cohen--Macaulay of level $n - \dim \Simplex{} - 1$ with
	constant $\UBCConstant{}{}$ and it is easy to check that $\tau$ is uniformly bounded Cohen--Macaulay of level $\dim \Simplex{} - 2$ with UBC constant $1$.
	It follows that the join $\Subcomplex$ is uniformly bounded Cohen--Macaulay of level $n-1$ with UBC constant $\max \{1, K\}$.

	We now filter the complex
	$
    \Ord
    {
      \Star
      {\SimplicialComplex}
      {\Simplex{}}
    }
    {\bullet}
    = \Ord{(\{a\} \ast
	\Subcomplex)}{\bullet}$ by the subcomplexes $\{F_i\Ord{(\{a\} \ast \Subcomplex)}{\bullet}\}_{i \in \NaturalNumbers}$
	where $F_i\Ord{(\{a\} \ast \Subcomplex)}{\bullet}$ contains all simplices $(x_0, \dots, x_k)$ with
	the property that the vertex $a$ can only be equal appear in the first
	$i+1$ positions, i.e.\ either $a \in \{x_0, \dots, x_i\}$ or $a \not\in
	\{x_0, \dots, x_k\}$. Notice that this filtration is finite: The assumption
	that $\SimplicialComplex$ has dimension $n$, implies that $F_n\Ord{(\{a\} \ast \Subcomplex)}{\bullet} =
	\Ord{\Star{\SimplicialComplex}{\Simplex{}}}{\bullet}$.

	We now use induction on $i \in \NaturalNumbers$ to show that $F_i\Ord{(\{a\} \ast \Subcomplex)}{\bullet}$ is uniformly $(n-1)$-acyclic.
	For the base case, $i = 0$, we observe that $F_0\Ord{(\{a\} \ast \Subcomplex)}{\bullet} = \{a\} \ast \Ord{Y}{\bullet}$ is a
	simplicial cone and hence uniformly $\infty$-acyclic with UBC constant $K_0 =
	1$ by \autoref{lem:ConeUBC}. For the induction step, let $i > 0$ and assume
	that $F_{i-1}\Ord{(\{a\} \ast Y)}{\bullet}$ is uniformly
	$(n-1)$-acyclic with UBC constant in degree $q \leq n-1$ given by $\UBCConstant{}{i-1}(n, \UBCConstant{}{})$, a function only depending on $n$, $\UBCConstant{}{}$ and $i$.

	Passing from $F_{i-1}(\{a\} \ast Y)_\bullet^{\text{ord}}$ to $F_{i}(\{a\}
	\ast Y)_\bullet^{\text{ord}}$, we need to glue on all $k$-simplices of the
	form $(x_0, \dots, x_{i-1}, a, y_{i+1}, \dots, y_k)$ for $k \leq n$. Fixing
	the prefix $\vec x a = (x_0, \dots, x_{i-1}, a)$ and letting the elements $y_{i+1}, \dots, y_k$
	vary, we obtain a semi-simplicial subset $F_{i}(\vec x a)_\bullet$ of
	$F_{i}(\{a\} \ast Y)_\bullet^{\text{ord}}$.
	Denote by $\eta = \{x_0, \dots, x_{i-1}\}$ the simplex of $\Subcomplex$ associated to $\vec x$.
	It holds that $F_{i}(\vec x a)_\bullet$ is the simplicial cone $F_{i}(\vec x a)_\bullet = (\vec x a)
	\ast \Ord{(\Link{Y}{\eta})}{\bullet}$ with cone point $a$, and
	therefore $F_{i}(\vec xa)_\bullet$ is uniformly $\infty$-acyclic with
	UBC constant $1$. Notice that $F_{i}(\vec xa)_\bullet \cap F_{i-1}\Ord{(\{a\} \ast
	Y)}{\bullet} = \partial(\vec x a) \ast
	\Ord{(\Link{Y}{\eta})}{\bullet}$ is a $(i-1)$-suspension, since $\partial (\vec x a)$ is the
	boundary of an $i$-simplex. The complex $\Link{Y}{\eta}$ is uniformly bounded Cohen--Macaulay
	of level $(n - 1) - (i - 1) - 1 = n - i - 1$ with UBC constant $\max \{1, K\}$
	as a link in $\Subcomplex$.
	Hence, the first induction hypothesis implies that
	$\Ord{(\Link{Y}{\eta})}{\bullet}$ is uniformly $(n-i-2)$-acyclic
	with constant $\UBCConstantOrdConstruction{n - i - 1}{\max \{1, K\}}$. \autoref{lem:join-qubc} therefore shows that $\partial(\vec x
	a) \ast \Ord{(\Link{Y}{\eta})}{\bullet}$ is uniformly $(n-2)$-acyclic
	with UBC constant in degree $q \leq n-2$ given by $\UBCConstantSuspension{i-1}{q}{\UBCConstantOrdConstruction{n - i - 1}{\max \{1, K\}}}$.
	As a consequence, \autoref{lem:MV-qubc} implies that
	$F_{i}(\vec xa)_\bullet \cup F_{i-1}\Ord{(\{a\} \ast
	Y)}{\bullet}$ is uniformly $(n-1)$-acyclic where the UBC constant in degree $q \leq n-1$ can be chosen as
	\begin{align*}
		\UBCConstant{}{i}(n, \UBCConstant{}{}) &\coloneqq \\
		&\max_{q \leq n-1} \UBCConstantMayerVietoris{q}{1}{\UBCConstant{}{i-1}(n, \UBCConstant{}{})}{\UBCConstantSuspension{i-1}{q-1}{\UBCConstantOrdConstruction{n - i - 1}{\max \{1, K\}}}}.
	\end{align*}

	Note that for any other $\vec z a = (z_0, \dots, z_{i-1}, a) \neq (\vec x
	a)$, the intersection of $F_{i}(\vec xa)_\bullet$ and $F_{i}(\vec
	za)_\bullet$ is contained in $F_{i-1}\Ord{(\{a\} \ast Y)}{\bullet}$.
	Furthermore, applying the same argument as in the previous paragraph to
	$F_{i}(\vec za)_\bullet$ shows that $F_{i}(\vec za)_\bullet$ also is
	uniformly $(n-1)$-acyclic with the \emph{same} UBC constant $\UBCConstant{}{i}(n, \UBCConstant{}{})$.
	We conclude that for any \emph{finite} choice of $(\vec z_0 a), \dots,
	(\vec z_l a)$ of pairwise different prefixes, it holds that
	$$F_{i-1}\Ord{(\{a\} \ast Y)}{\bullet} \cup (\bigcup_{m = 0}^l F_{i}(\vec
	z_m a)_\bullet)$$
	is uniformly $(n-1)$-acyclic with constant $\UBCConstant{}{i}(n, \UBCConstant{}{})$. Therefore
	\autoref{lem:DirectedUnion} implies that $F_{i}\Ord{(\{a\}
	\ast Y)}{\bullet}$ is uniformly $(n-1)$-acyclic with constant
	$\UBCConstant{}{i}(n, \UBCConstant{}{})$. This completes the induction step,
	and we may define $\UBCConstantOrdConstructionStar{n}{\UBCConstant{}{}} \coloneqq \UBCConstant{}{n}(n, \UBCConstant{}{})$.
\end{proof}

	We now proceed with the proof of \autoref{lem:ord-complex-connectivity-lemma}, assuming that $n \geq 1$.
	Note that, on the one hand, forgetting the ordering
	yields a map $\Retraction: \Ord{\SimplicialComplex}{\bullet}
	\to \SimplicialComplex$. On the other hand, picking a total order on the
	set of vertices of $\SimplicialComplex$, yields an embedding
	$s: \SimplicialComplex \hookrightarrow
	\Ord{\SimplicialComplex}{\bullet}$ with the property $\Retraction
	\circ s = \Identity_{\SimplicialComplex}$. This shows that the image of
	$s$, which we denote by $\SimplicialComplex_\bullet$, is a retract of
	$\Ord{\SimplicialComplex}{\bullet}$. This remains true when passing
	to the barycentric subdivisions
	$\BarycentricSubdivion{\SimplicialComplex_\bullet}$ and
	$\BarycentricSubdivion{\SimplicialComplex_\bullet^{\text{ord}}}$. We
	claim that, in degrees $k < n$, the map $s \circ \Retraction$ is bounded
	chain-homotopic
	to
	$\Identity_{\BarycentricSubdivion{\Ord{\SimplicialComplex}{\bullet}}}$,
	$$\Homotopy_* :
	\ChainComplex{*}{\BarycentricSubdivion{\Ord{\SimplicialComplex}{\bullet}}}
	\to
	\ChainComplex{*+1}{\BarycentricSubdivion{\Ord{\SimplicialComplex}{\bullet}}}
	 \text{ for } * < n.
	$$
	We construct this chain-homotopy inductively. Our first task is to
	describe a map
	$$\Homotopy_0 :
	\ChainComplex{0}{\BarycentricSubdivion{\Ord{\SimplicialComplex}{\bullet}}}
	\to
	\ChainComplex{1}{\BarycentricSubdivion{\Ord{\SimplicialComplex}{\bullet}}}.
	$$
	A vertex in
	$\BarycentricSubdivion{\Ord{\SimplicialComplex}{\bullet}}$ is a
	simplex $\sigma \in \Ord{\SimplicialComplex}{\bullet}$,
	which upon forgetting the order yields unique simplices
	$\Retraction(\sigma)
	\in \SimplicialComplex$ and $(s \circ \Retraction)(\sigma) \in \SimplicialComplex_\bullet$.
	By the claim, it holds that the subcomplex
	$\Ord{\Star{\SimplicialComplex}{\Retraction(\sigma)}}{\bullet}$ of
	$\Ord{\SimplicialComplex}{\bullet}$ is uniformly $(n-1)$-acyclic with UBC constant $\UBCConstantOrdConstructionStar{n}{\UBCConstant{}{}}$.
	Notice that both $\sigma$ and $(s \circ \Retraction)(\sigma)$ are vertices in $\Ord{\Star{\SimplicialComplex}{\Retraction(\sigma)}}{\bullet}$.
	It follows that there is a simplicial path $\gamma^1 = (\sigma
	= \sigma_0, \sigma_1, \dots, \sigma_l = (s \circ \Retraction)(\sigma))$ in
	$\Ord{\Star{\SimplicialComplex}{\Retraction(\sigma)}}{\bullet}$
	of length $l \leq 2 \cdot \UBCConstantOrdConstructionStar{n}{\UBCConstant{}{}}$.
	We define
	$$\Homotopy_0(\sigma) \coloneqq \gamma^1 = \sum_1^{l-1} \on{sign}_i \cdot [\sigma_i, \sigma_{i+1}]$$
	where $\on{sign}_i \in \{+1,-1\}$ is chosen such that
	$$
	\BoundarySimplex_1 \Homotopy_0(\sigma) = \sigma - (s \circ \Retraction)(\sigma) =
	\Identity_{\BarycentricSubdivion{\Ord{\SimplicialComplex}{\bullet}}}(\sigma)
	- \BarycentricSubdivion{s \circ \Retraction}(\sigma).
	$$
	This finishes the induction beginning $k = 0$. We note that $\Norm{\Homotopy_0} \leq 2 \cdot \UBCConstantOrdConstructionStar{n}{\UBCConstant{}{}}$. Now assume for an induction on $k$,
	that for any $k < n-1$ we have defined a homotopy between
	$\Identity_{\BarycentricSubdivion{\Ord{\SimplicialComplex}{\bullet}}}$
	 and $\BarycentricSubdivion{s \circ \Retraction}$,
	$$\Homotopy_k :
	\ChainComplex{k}{\BarycentricSubdivion{\Ord{\SimplicialComplex}{\bullet}}}
	\to
	\ChainComplex{k+1}{\BarycentricSubdivion{\Ord{\SimplicialComplex}{\bullet}}},
	$$
	with the property that for all $k$-simplices $\alpha_k = \sigma_0 \leq
	\dots \leq \sigma_k$ in
	$\BarycentricSubdivion{\Ord{\SimplicialComplex}{\bullet}}$ it holds
	that
	$$
	\Homotopy_k(\alpha_k) \in
	\ChainComplex{k+1}{\BarycentricSubdivion{\Ord{\Star{\SimplicialComplex}{\Retraction(\sigma_0)}}{\bullet}}}
	\subseteq
	\ChainComplex{k+1}{\BarycentricSubdivion{\Ord{\SimplicialComplex}{\bullet}}}
	$$
	and $\Norm{\Homotopy_k} \leq \UBCConstant{\Homotopy_k}{}(n, \UBCConstant{}{})$ for a function only depending on $k, n$ and $\UBCConstant{}{}$.

	Now, let $\alpha_{n-1} = \sigma_0 \leq \dots \leq \sigma_{n-1}$ be a
	$(n-1)$-simplex in
	$\BarycentricSubdivion{\Ord{\SimplicialComplex}{\bullet}}$.
	Since
	$$\Star{\SimplicialComplex}{\Retraction(\sigma_{n-1})}
	 \subseteq \dots \subseteq
	\Star{\SimplicialComplex}{\Retraction(\sigma_0)},$$
	it holds that
	$$
	\BarycentricSubdivion{\Ord{\Star{\SimplicialComplex}{\Retraction(\sigma_{n-1})}}{\bullet}}
	\subseteq \dots \subseteq
	\BarycentricSubdivion{\Ord{\Star{\SimplicialComplex}{\Retraction(\sigma_0)}}{\bullet}}.
	$$
	The induction hypothesis implies that
	$\Homotopy_{n-2}(\BoundaryChainComplex{} 	\alpha_{n-1})$ is contained in
	$$\ChainComplex{n-1}{\BarycentricSubdivion{\Ord{\Star{\SimplicialComplex}{\Retraction(\sigma_0)}}{\bullet}}}.$$
	It follows that
	$$\alpha_{n-1} + \BarycentricSubdivion{s \circ \Retraction}(\alpha_{n-1}) - \Homotopy_{n-2}(\BoundaryChainComplex{} \alpha_{n-1})$$
	is a $(n-1)$-cycle in
	$\ChainComplex{n-1}{\BarycentricSubdivion{\Ord{\Star{\SimplicialComplex}{\Retraction(\sigma_0)}}{\bullet}}}$.
	By the claim, the cycle is a boundary and there exists some
	$$\gamma^n \in
	\ChainComplex{n}{\BarycentricSubdivion{\Ord{\Star{\SimplicialComplex}{\Retraction(\sigma_0)}}{\bullet}}}$$
	with the property that $\BoundaryChainComplex{} \gamma^n = \alpha_{n-1} +
	\BarycentricSubdivion{s \circ \Retraction}(\alpha_{n-1}) -
	\Homotopy_{n-2}(\BoundaryChainComplex{} \alpha_{n-1})$ and that
	$\Norm{\gamma^n}$ is at most
	\begin{align*}
		&\UBCConstantOrdConstructionStar{n}{\UBCConstant{}{}}
		\cdot
		(\Norm{\Identity} + \Norm{\BarycentricSubdivion{s \circ \Retraction}} + \Norm{\Homotopy_{n-2} \circ \BoundaryChainComplex{}})
		&\leq
		\UBCConstantOrdConstructionStar{n}{\UBCConstant{}{}}
		\cdot
		(2 + n \cdot \UBCConstant{\Homotopy_{n-2}}{}(n, \UBCConstant{}{})),
	\end{align*}
	using that $\Norm{\BarycentricSubdivion{s \circ \Retraction}} \leq 1$.
	We set $\UBCConstant{\Homotopy_{n-1}}{}(n, \UBCConstant{}{}) \coloneqq \UBCConstantOrdConstructionStar{n}{\UBCConstant{}{}}
	\cdot (2 + n \cdot \UBCConstant{\Homotopy_{n-2}}{}(n, \UBCConstant{}{}))$ and conclude that
	$$\Homotopy_{n-1}(\alpha_{n-1}) \coloneqq \gamma^n$$
	is a chain homotopy with the desired properties.

	Let $\UBCConstant{\Homotopy}{}(n, \UBCConstant{}{}) \coloneqq \max_{k \leq n-1}  \UBCConstant{\Homotopy_{k}}{}(n, \UBCConstant{}{})$.
	Then it follows that for any $k \leq n-1$, $\Norm{\Homotopy_k} \leq \UBCConstant{\Homotopy}{}(n, \UBCConstant{}{})$ and \autoref{cor:uniform-homotopy-equivalence}
	implies that $\BarycentricSubdivion{\Ord{\SimplicialComplex}{\bullet}}$ is uniformly $(n-1)$-acyclic such that the UBC constant in degree $q \leq n-1$ can be chosen
	as $\UBCConstant{\Homotopy}{}(n, \UBCConstant{}{})$. An application of \autoref{lem:BarycentricSubdivision} therefore shows that $\Ord{\SimplicialComplex}{\bullet}$
	is also uniformly $(n-1)$-acyclic and that its UBC constant in degree $q \leq n-1$ can be chosen as
	\[
	\UBCConstantOrdConstruction{n}{\UBCConstant{}{}} \coloneqq n! \cdot \UBCConstant{\Homotopy}{}(n, \UBCConstant{}{}) \qedhere
	\]
\end{proof}

\subsection{Technical lemma I}
The following lemma will play a crucial
role in the proof of stability patterns in bounded cohomology of automorphism groups of quadratic modules.
It is a uniformly bounded analogue of \cite[Proposition 2.5]{GalatiusRandalWilliamsStability}.
\begin{lemma}
  \label{lem:FullSubcomplex}
  Let $\SimplicialComplex$ be a simplicial complex,
  $\Subcomplex \subseteq \SimplicialComplex$ be a full subcomplex
  and $n \in \NaturalNumbers$.
  Assume that for each $q \leq n$ there exists a constant $\UBCConstant{q}{\Linkfunction} \geq 0$ such that the following holds:
  For every $p$-simplex $\Simplex{p}$ of $\SimplicialComplex$
  having no vertex in $\Subcomplex$, the complex
  $\Link{\Subcomplex}{\Simplex{p}} = Y \cap \Link{X}{\Simplex{p}}$ is uniformly $(n-p-1)$-acyclic
  and its UBC constant in degree $q \leq n-p-1$ can be chosen as $\UBCConstant{q}{\Linkfunction}$.
  Then the inclusion
  $\Subcomplex \hookrightarrow \SimplicialComplex$
  is uniformly $n$-connected and its UBC constant in degree $q \leq n$ can be
  chosen as
  \begin{align*}
  \UBCConstantTechnicalLemmaI{q}{n}{\UBCConstant{q}{\Linkfunction}}
  &\coloneqq\\
  &\UBCConstantFactorThrough{q}{\UBCConstantCellAttachment{n}{q}{\UBCConstant{q}{\Linkfunction}}}{\UBCConstantFactorThrough{q}{\UBCConstantCellAttachment{n-1}{q}{\UBCConstant{q}{\Linkfunction}}}{K^{\on{Fact}} \bigl ( \dots \bigr)}}.
  \end{align*}
\end{lemma}
\begin{proof}
  We construct a factorization of the inclusion $\Subcomplex \hookrightarrow
  \SimplicialComplex$
  through a sequence of subcomplex $\Subcomplex_0, \dots, \Subcomplex_n$ of $\SimplicialComplex$
  and then invoke \autoref{lem:LinkAttachments}.

  Let $S_{p}$ denote the set of $p$-simplices of $\SimplicialComplex$ with no
  vertices in $\Subcomplex$.
  We define
  $
    \Subcomplex_{-1}
    =
    \Subcomplex
  $
  and inductively for $p \leq n$,
  $
    \Subcomplex_{p}
    =
    \Subcomplex_{p-1}
    \cup
    (\bigcup_{\Simplex{p}\in S_{p}}
    \Simplex{p}
    \ast
    \Link{\Subcomplex}{\Simplex{p}}).
  $

  Applying \autoref{lem:LinkAttachments}, it follows that the inclusion
  $
    \Subcomplex_{p-1}
    \hookrightarrow
    \Subcomplex_{p},
  $
  is uniformly $n$-acyclic and that its UBC constant
  in degree $q \leq n$ can be chosen as
  $\UBCConstantCellAttachment{p}{q}{\UBCConstant{q}{\Linkfunction}}$.

  For the last inclusion $\Subcomplex_n \hookrightarrow \SimplicialComplex$,
  we claim that the $n$-skeleton of $\SimplicialComplex$ is
  contained in $\Subcomplex_{n}$: Let $\Simplex{p}$ be a $p$-simplex in
  $\SimplicialComplex$ for $p\leq n$. If $\Simplex{p}$ does not contain any
  vertices in $\Subcomplex$, then it is contained in $\Subcomplex_p \subseteq \Subcomplex_{n}$
  by construction. If $\Simplex{p}$ contains vertices in $\Subcomplex$, then it
  can
  be written as
  $
  \Simplex{p}^{\Subcomplex}
  \sqcup
  \Simplex{p}^{\SimplicialComplex \setminus \Subcomplex}
  $
  for $\Simplex{p}^{\Subcomplex}$ a nonempty simplex of $\Subcomplex$
  and $\Simplex{p}^{\SimplicialComplex \setminus \Subcomplex}$ a (possibly empty) simplex of $\SimplicialComplex$
  that has no vertex in $\Subcomplex$. Note that the dimension $d$ of
  $
  \Simplex{p}^{\SimplicialComplex \setminus \Subcomplex}
  $
  is strictly less than $p \leq n$ and that
  $
  \Simplex{p}^{\Subcomplex}
  \in
  \Link{\Subcomplex}{\Simplex{p}^{\SimplicialComplex \setminus \Subcomplex}}.
  $
  Therefore $\Simplex{p}$ is contained in $\Subcomplex_d \subseteq \Subcomplex_{n}$.

  It follows that the inclusion
  $\Subcomplex_{n}\hookrightarrow\SimplicialComplex$ is uniformly $n$-acyclic
  as well and that its UBC constant in degree $q \leq n$ can be chosen as $0$.

  Iteratively applying \autoref{lem:compositions-of-highly-acyclic-maps} hence implies that the inclusions
  $\Subcomplex \hookrightarrow \SimplicialComplex$ is uniformly $n$-acyclic as
  well and that its UBC constant in degree $q \leq n$ can be chosen as stated in the claim.
\end{proof}

\subsection{Technical lemma II}
The final lemma of this section is the lynchpin of the bounded acyclicity argument which implies
that general linear groups satisfy bounded cohomological stability.
It is a uniformly bounded refinement of an argument contained in \cite[Section 2]{vanderkallen1980homologystabilityforlineargroups}.
\begin{lemma}
\label{lem:UBCindegreenplusone}
  Let $\Poset$ denote a poset and suppose that $\Poset\cup \{v\}$ is obtained
  from $\Poset$ by attaching $v$ along $\Link{\Poset}{v}$. Furthermore assume:
  \begin{enumerate}[(i)]
  \item \label{item-1-UBCindegreenplusone}
    $\Link{\Poset}{v}$ is uniformly $n$-acyclic with constant
    $\UBCConstant{q}{\Linkfunction}$.
  \item \label{item-2-UBCindegreenplusone}
    There exists a subposet $\Subposet \subseteq \Poset$ as in \autoref{lem:PosetDeformation}, i.e.\ 
    $\iota: \Subposet \hookrightarrow \Poset$ is a poset deformation retract 
    via the map $\Retraction \colon \Poset \to \Subposet$, $\apply{\Retraction}{\PosetElement} \coloneqq \sup \Subposet^-(\PosetElement)$.
  \item \label{item-3-UBCindegreenplusone}
    There exists a poset embedding $\Psi \colon \Subposet \to \Link{\Poset}{v}$
    such that $\Retraction|_{\im(\Psi)} \circ \Psi = \Identity_{\Subposet}$.
  \end{enumerate}
  Then $\Poset \cup \{v\}$ is uniformly $(n+1)$-acyclic and its UBC constant in
  degree $q \leq n+1$ is bounded by
  \[
  	\UBCConstantTechnicalLemmaII{q}{\UBCConstant{q-1}{\Linkfunction}}
  	\coloneqq
    (q+2)
    (1 + (q+1) \UBCConstant{q-1}{\Linkfunction})
    \left(
    2
    +
    2(q+2)(q+1)
    +
    2 (q+1)
    \right).
  \]
  In particular, if $\Link{\Poset}{v}$ is $(-1)$-acyclic then
  $\Poset\cup\{v\}$ is uniformly $0$-acyclic, i.e.\ it is connected of finite diameter and satisfies the $0$-UBC with constant $\UBCConstantTechnicalLemmaII{0}{0}$.
\end{lemma}
\begin{proof}
  We first establish the case $n=-1$:
  By \autoref{item-2-UBCindegreenplusone}, each element in $\Poset$ is connected to some element of 
  $\Subposet$ by an edge (i.e.\ $\Retraction(\PosetElement) \leq \PosetElement$ for $\PosetElement \in \Poset$).
  By \autoref{item-3-UBCindegreenplusone}, each element in $\Subposet$ is connected to some element of
  $\Link{\Poset}{v}$ by an edge (i.e.\ $\PosetElement = \Retraction(\Psi(\PosetElement))) \leq \Psi(\PosetElement)$ for $\PosetElement \in \Subposet$).
  Since each element in $\Link{\Poset}{v}$ is connected to $v$ by an edge, we conclude that every element of $\Poset\cup\{v\}$
  is connected to $v$ by a path of length at most $3$. Hence, $\Poset\cup\{v\}$ is connected with diameter at most $6$ and \autoref{rem:0UBC} implies 
  that it is uniformly $0$-acyclic with UBC constant $3 \leq \UBCConstantTechnicalLemmaII{0}{0} = 16$.

  Let $n\geq 0$. We start by recording an observation: Let $\Chain \in \ChainComplex{q}{\Poset}$ be any chain with $\BoundaryChainComplex{} \Chain = 0$.
  As in the proof of \autoref{lem:PosetDeformation}, \autoref{item-2-UBCindegreenplusone} implies that there exists a chain homotopy $\Homotopy_{\Poset}$ between $\Identity_F$ and $(\iota_{*} \circ \Retraction_{*})$ of norm at most $q+1$ in degree $q$.
  Hence, it holds that
  $
    \apply{(\iota_{*} \circ \Retraction_{*})}{\Chain} - \Chain = \BoundarySimplex
    \apply{\Homotopy_{\Poset}}{\Chain}
  $
  and, since $\apply{(\iota_{*} \circ \Retraction_{*})}{\apply{\Psi_{*}}{\apply{\Retraction_{*}}{\Chain}}} = \apply{(\iota_{*} \circ \Retraction_{*})}{\Chain}$ by \autoref{item-3-UBCindegreenplusone}, that
  $
    \apply{(\iota_{*} \circ \Retraction_{*})}{\Chain}
    -
    \apply{\Psi_{*}}{\apply{\Retraction_{*}}{\Chain}}
    =
    \BoundarySimplex
    \apply{\Homotopy_{\Poset}}{\apply{\Psi_{*}}{\apply{\Retraction_{*}}{\Chain}}}.
  $
  Therefore,
  $
    \apply{\Psi_{*}}{\apply{\Retraction_{*}}{\Chain}} - \Chain 
    =
    \BoundarySimplex \apply{\Homotopy_{\Poset}}{\Chain}
    -
    \BoundarySimplex \apply{\Homotopy_{\Poset}}{\apply{\Psi_{*}}{\apply{\Retraction_{*}}{\Chain}}}.
  $
  It follows that there exists a $(q+1)$-chain $\tau = \apply{\Homotopy_{\Poset}}{\Chain}
  + \apply{\Homotopy_{\Poset}}{\apply{\Psi_{*}}{\apply{\Retraction_{*}}{\Chain}}}$ such that
  $
    \Chain + \BoundarySimplex \tau \in \ChainComplex{q}{\Link{\Poset}{v}}
  $
  and
  $
    \Norm{\tau}
    \leq
    2 (q+1)
    \Norm{\Chain}
  $.

  We now check that $\Poset \cup \{v\}$ is uniformly $(n+1)$-acyclic. Consider a chain $\Chain$ with $\BoundaryChainComplex{} \Chain = 0$ in
  $\ChainComplex{q}{\Poset \cup \{v\}}$ for $q\leq n+1$.
  We record that
  $
  (\Poset \cup \{v\})_\bullet = \Poset_\bullet \cup \Star{\Poset\cup\{v\}}{v}_\bullet
  $
  with $\Poset_\bullet \cap \Star{\Poset\cup\{v\}}{v}_\bullet = \Link{\Poset}{v}_\bullet$.
  Hence, we can write $\Chain = \Chain_{\Starfunction} + \sigma_{\Poset}$, where
  $
    \Chain_{\Starfunction}
    \in
    \ChainComplex{q}{\Star{\Poset\cup \{v\}}{v}},
  $
  $
    \Chain_{\Poset}
    \in
    \ChainComplex{q}{\Poset}
  $
  and $\Norm{\Chain} = \Norm{\Chain_{\Starfunction}} + \Norm{\Chain_{\Poset}}$.
  Then $\BoundaryChainComplex{} \Chain = 0$ implies that
  $
    \BoundarySimplex \Chain_{\Poset}
    =
    - \BoundarySimplex \Chain_{\Starfunction}
    \in
    \ChainComplex{q-1}{\Link{\Poset}{v}}.
  $
  By \autoref{item-1-UBCindegreenplusone} and since $q-1 \leq n$, we have that
  $
    \BoundarySimplex \Chain_{\Poset}
    =
    - \BoundarySimplex \Chain_{\Starfunction}
  $
  is the boundary of a chain
  $
    \BoundingChain
    \in
    \ChainComplex{q}{\Link{\Poset}{v}}
  $ 
  with $\Norm{\BoundingChain} \leq \UBCConstant{q-1}{\Linkfunction} \Norm{\BoundarySimplex \Chain_{\Poset}} \leq (q+1) \UBCConstant{q-1}{\Linkfunction} \Norm{\Chain_{\Poset}}$.
  It follows that $\Chain_{\Poset} - \BoundingChain \in \ChainComplex{q}{\Poset}$ is a closed chain, i.e.\ $\BoundaryChainComplex{}(\Chain_{\Poset} - \BoundingChain) = 0$.
  By the observation recorded above, there exist a $(q+1)$-chain $\tau_{\Poset}$ such
  that
  $
  	(\Chain_{\Poset} - \BoundingChain)
  	+
    \partial \tau_\Poset = \Chain_{\Linkfunction}
    \in \ChainComplex{q}{\Link{\Poset}{v}}
  $ 
  with
  \begin{align*}
    \Norm{\tau_{\Poset}}
    &\leq
    2 (q+1)
    \Norm{\Chain_{\Poset} - \BoundingChain}
    \leq
    2 (q+1)
    \left(
      1 + (q+1)\UBCConstant{q-1}{\Linkfunction}
    \right)
    \Norm{\Chain_{\Poset}}.
  \end{align*}
  We also note that therefore
  \begin{align*}
    \Norm{\Chain_{\Linkfunction}}\
    &\leq
    \Norm{\Chain_{\Poset}}
    + 
    \Norm{\BoundingChain}
    +
    (q+2) \Norm{\tau_\Poset}\\
    &\leq
    (1
    +
    (q+1) \UBCConstant{q-1}{\Linkfunction}
    +
    (q+2)2(q+1) (1 + (q+1)\UBCConstant{q-1}{\Linkfunction}))
    \Norm{\Chain_{\Poset}}\\
    &\leq
    (1
    +
    (q+1) \UBCConstant{q-1}{\Linkfunction})
    (1 + 2(q+2)(q+1))\Norm{\Chain_{\Poset}}.
  \end{align*}

  By \autoref{lem:ConeUBC}, $\Star{\Poset\cup \{v\}}{v}$ is uniformly $\infty$-acyclic with UBC
  constant $1$ in all degrees. Since $\Chain_{\Linkfunction}$ and $\Chain_{\Starfunction} + \BoundingChain$ both closed chains in
  $\Star{\Poset\cup \{v\}}{v}$, i.e.\ $\BoundaryChainComplex{} \Chain_{\Linkfunction} = 0$ and $\BoundaryChainComplex{}(\Chain_{\Starfunction} + \BoundingChain) = 0$,
  it follows that there exist $(q+1)$-chains $\tau_{\Linkfunction}, \tau_{\Starfunction} \in \ChainComplex{q+1}{\Star{\Poset\cup \{v\}}{v}}$ such that 
  $\BoundaryChainComplex{} \tau_{\Linkfunction} = \Chain_{\Linkfunction}$ and $\BoundaryChainComplex{} \tau_{\Starfunction} = \Chain_{\Starfunction} + \BoundingChain$ and such that their norms satisfy
  $$\Norm{\tau_{\Linkfunction}} \leq \Norm{\Chain_{\Linkfunction}} \leq (1+(q+1) \UBCConstant{q-1}{\Linkfunction})(1 + 2(q+2)(q+1))\Norm{\Chain_{\Poset}}$$
  and 
  $$\Norm{\tau_{\Starfunction}} \leq \Norm{\Chain_{\Starfunction}} + \Norm{\BoundingChain} \leq
  \Norm{\Chain_{\Starfunction}} + (q+1) \UBCConstant{q-1}{\Linkfunction} \Norm{\Chain_{\Poset}}.
  $$
  All in all this gives that $\Chain = (\Chain_{\Starfunction} + \BoundingChain) + (\Chain_{\Poset} - \BoundingChain) = (\Chain_{\Starfunction} + \BoundingChain) +  (\Chain_{\Linkfunction}  - \BoundaryChainComplex{} \tau_{\Poset}) = \BoundaryChainComplex{} \tau_{\Starfunction} + (\BoundaryChainComplex{} \tau_{\Linkfunction} - \BoundaryChainComplex{} \tau_{\Poset})$ is the boundary of the $(q+1)$-chain $\tau_{\Starfunction} + \tau_{\Linkfunction} - \tau_{\Poset}$. Using that $\Norm{\Chain_{\Starfunction}} \leq \Norm{\Chain}$ and $\Norm{\Chain_{\Poset}} \leq \Norm{\Chain}$ as well as the estimates above, it follows that
  \begin{align*}
    \Norm{\Chain} 
    &
    =
    \Norm{\BoundaryChainComplex{}(\tau_{\Starfunction} + \tau_{\Linkfunction} - \tau_{\Poset})}
    \leq
    (q+2) (\Norm{\tau_{\Starfunction}} + \Norm{\tau_{\Linkfunction}} + \Norm{\tau_{\Poset}})\\
    &\leq
    (q+2)
    (
      1 + (q+1) \UBCConstant{q-1}{\Linkfunction}
      +
      (1+(q+1) \UBCConstant{q-1}{\Linkfunction})(1 + 2(q+2)(q+1))
      +
      \\
      &\text{ }\hspace{3cm} 2 (q+1) (1 + (q+1)\UBCConstant{q-1}{\Linkfunction})
    )
    \Norm{\Chain}
    \\
    &
    =
    (q+2)
    (1 + (q+1) \UBCConstant{q-1}{\Linkfunction})
    \left(
    2
    +
    2(q+2)(q+1)
    +
    2 (q+1)
    \right)
    \Norm{\Chain}. \qedhere
  \end{align*}	
\end{proof}
\section{Warmup: A uniform Solomon--Tits theorem}
\label{scn:SolomonTits}
The goal of this section is to illustrate at a simple example
how our uniformly bounded toolbox can be applied to prove bounded acyclicity results:
We prove a uniformly bounded version of the Solomon--Tits theorem, i.e.\ \autoref{item:solomon-tits-an} and \autoref{item:solomon-tits-bncn} of \autoref{thm:general-connectivity}.
This result is not used in the remainder of the article, but might be of independent interest. 
We decided to include it as a warmup, and to showcase the flexibility of our techniques.
A topologized version of this result for continuous bounded cohomology was previously obtained and used by Monod (see \cite[Theorem~3.9]{MonodSemiSimple}).

\begin{definition}
	\label{def:tits-building}
  Let $\Field$ be a field. The \introduce{Tits building $\TitsBuilding{\Field}{n}$ of type $\mathtt{A}_{n-1}$ over $\Field$} is the poset of proper, nontrivial subspaces of $0 \neq \SubVectorSpace \subsetneq \Field^n$, where the partial order is defined by the inclusion of subspaces $\SubVectorSpace \subseteq \SubVectorSpace'$.
\end{definition}

Tits buildings play a key role in the theory of arithmetic groups, see e.g.\ \cite{borelserre1973}.
A celebrated theorem of Solomon--Tits theorem
\cite{solomon1969thesteinbergcharacterofafinitegroupwithbnpair} states that
$\TitsBuilding{\Field}{n}$ is homotopy equivalent to a bouquet of
$(n-2)$-spheres, $\TitsBuilding{\Field}{n} \simeq \bigvee S^{n-2}$.
Using the tools from
\autoref{sec:toolbox}, we prove the following uniformly bounded refinement of
this result.

\begin{theoremnum}
	\label{thm:uniform-solomon-tits-theorem}
	There exists a function $K^{T}: \NaturalNumbers \to \Reals$ such that for
	any field $\Field$ the Tits building $\TitsBuilding{\Field}{n}$ is
	uniformly $(n-3)$-acyclic and satisfies the $q$-UBC for all $q \leq n-3$
	with constant $K^{T}(n)$. In particular, $\ReducedBoundedCohomologyOfSimplicialObject{q}{\TitsBuilding{\Field}{n}}{\Reals} = 0$ if $q \neq n-2$.
\end{theoremnum}

The following argument is a uniformly bounded refinement of the proof of \cite[Theorem 5.1]{BestvinaPLMorseTheory}.

\begin{proof}
  We proceed by induction on $n \in \NaturalNumbers$.
  Let $\Field$ be any field.
  Note that $\TitsBuilding{\Field}{2}$ is a non-empty discrete space,
  hence it is uniformly $-1$-acyclic with constant
  $\UBCConstant{-1}{\TitsBuilding{\Field}{2}} = 0$.
  We therefore define $K^{T}: \{0,1,2\} \to \Reals$ to be the zero function.

  Now suppose that we have constructed a function $K^{T}: \{0, \dots, n-1\}
  \to \Reals$ such that for any field $\Field$ and any $n' < n$ the Tits
  building $\TitsBuilding{\Field}{n'}$ is uniformly $(n'-3)$-acyclic and
  satisfies the $q$-UBC for all $q \leq n' - 3$ with constant $K^{T}(n')$.

  For the induction step, we construct
  $\TitsBuilding{\Field}{n}$ by iteratively attaching vertices, along highly
  uniformly acyclic links, to certain subcomplexes of
  $\TitsBuilding{\Field}{n}$. Using the tools in
  \autoref{sec:toolbox}, we keep track of how each gluing operation
  affects the
  UBC-constants. Then we use the resulting estimates for the UBC-constants of
  $\TitsBuilding{\Field}{n}$ to define the value $K^{T}(n)$ of $K^{T}$ at
  $n$.

  Let $\Field$ be any field. Fix a line $\ell$ in $\Field^n$. The star
  $\Star{\TitsBuilding{\Field}{n}}{\ell}$ of the vertex $\ell \in
  \TitsBuilding{\Field}{n}$ is a cone over the link of $\ell$,
  $$\Star{\TitsBuilding{\Field}{n}}{\ell} = \{\ell\} \ast
  \Link{\TitsBuilding{\Field}{n}}{\ell}.$$
  Therefore \autoref{lem:ConeUBC} implies that
  $\Star{\TitsBuilding{\Field}{n}}{\ell}$ is uniformly $\infty$-acyclic and
  satisfies the $q$-UBC for all $q$ with constant $K_0^q \coloneqq 1$.

  The vertices of $\TitsBuilding{\Field}{n}$ which are not contained in
  $\Star{\TitsBuilding{\Field}{n}}{\ell}$ are proper, nontrivial subspaces $0
  \neq P \subsetneq \Field^n$ that do not contain $\ell$. Such vertices $P$
  satisfy $\apply{\dim}{P} \in \{1, \dots, n-1\}$. We now filter
  $\TitsBuilding{\Field}{n}$ by the subposets $Q_0 =
  \Star{\TitsBuilding{\Field}{n}}{\ell}$ and
  $$Q_r = Q_0 \cup \{P \in \TitsBuilding{\Field}{n}: \ell \not\subseteq P
  \text{ and } 1 \leq \apply{\dim}{P} \leq r\}.$$
  Note that this filtration is \emph{finite} since $Q_{n-1} =
  \TitsBuilding{\Field}{n}$.

  We now apply \autoref{lem:ConeUBC}, \autoref{cor:uniform-discrete-morse-theory} and the
  induction hypothesis to study how passing from $Q_r$ to $Q_{r+1}$ affects
  uniform acyclicity and the UBC-constants.

  Let $0 \leq r < n-1$ and suppose that $Q_{r}$ is uniformly
  $\infty$-acyclic and satisfies the $q$-UBC for all $q$ with constant
  $K_{r}^q$. Notice that $Q_{r+1} \setminus Q_r$ is discrete, which means that
  $Q_{r+1}$ is obtained from $Q_r$ by gluing all elements $P \in Q_{r+1}
  \setminus Q_r$ to $Q_r$ along their links $\Link{Q_r}{P}$ in $Q_r$ (see
  \autoref{def:PosetObtainedByGluing}).

  Let $P \in Q_{r+1} \setminus Q_r$ be a subspace of $\Field^n$ of dimension
  $r+1$ not containing $\ell$. Then $\Link{Q_r}{P}$ decomposes as a join
  \begin{equation}
  	\label{eqn:LinkInTitsBuilding}
    \Link{Q_r}{P}
    =
    \left\{
      P' \in Q_r : P' \supseteq \apply{\mathrm{span}}{P\cup \ell}
    \right\}
    \ast
    \left\{
      P' \in Q_r: P' \subsetneq P
    \right\}.
  \end{equation}

  If $r \neq n-2$, it follows $\apply{\mathrm{span}}{P\cup \ell} \neq \Field^n$
  and that $\apply{\mathrm{span}}{P\cup \ell} \in \Link{Q_r}{P}$ is a conepoint
  in $\Link{Q_r}{P}$. Therefore \autoref{lem:ConeUBC} implies that
  $\Link{Q_r}{P}$ is uniformly $\infty$-acyclic and satisfies the $q$-UBC for
  all $q$ with constant $1$. An application of \autoref{cor:uniform-discrete-morse-theory} to
  \[
    Q_{r+1}
    =
    Q_{r}
    \cup
    \bigcup_{\ell \not\subseteq P,\text{ } \apply{\dim}{P} = r+1}
    \Star{Q_{r+1}}{P}
  \]
  yields that $Q_{r+1}$ is uniformly $\infty$-acyclic and satisfies the
  $q$-UBC with constants
  $$K_{r+1}^q = \UBCConstantMorseTheory{0}{q}{\UBCConstant{q}{r}}{1} = 2 + q + K^{q}_{r} + (q+2)^2 K^{q}_{r}.$$

  Now, let $r = n-2$.  To obtain $Q_{n-1} = \TitsBuilding{\Field}{n}$ from
  $Q_{n-2}$, we need to glue subspaces $P \in Q_{n-1} \setminus Q_{n-2}$ of
  dimension $n-1$ that do not contain $\ell$
  to $Q_{n-2}$. Observe that $\apply{\mathrm{span}}{P\cup \ell} = \Field^n$.
  Therefore, \autoref{eqn:LinkInTitsBuilding} implies that $\Link{Q_{n-2}}{P}$
  is given by the poset of proper, nontrivial subspaces of $P$,
  \begin{equation*}
  	\Link{Q_{n-2}}{P}
  	=
  	\left\{
  	P' \in Q_{n-2}: P' \subsetneq P
  	\right\}.
  \end{equation*}
  This is evidently isomorphic to $\TitsBuilding{\Field}{n-1}$. Therefore
  the induction hypothesis implies that $\Link{Q_{n-2}}{P}$ is uniformly
  $(n-4)$-acyclic and satisfies the $q$-UBC for all $q \leq n-4$ with constant
  $K^T(n-1)$. Applying \autoref{cor:uniform-discrete-morse-theory} to
  \[
    \TitsBuilding{\Field}{n} = Q_{n-1}
    =
    Q_{n-2}
    \cup
    \bigcup_{\ell \not \subset P,\text{ } \apply{\dim}{P} = n-1}
    \Star{Q_{n-1}}{P}
  \]
  yields that $\TitsBuilding{\Field}{n} = Q_{n-1}$ is uniformly
  $(n-3)$-acyclic and satisfies the $q$-UBC condition for $q \leq n-3$ with
  constant
  \begin{align*}
    K^q_{n-1}
    &=
    \UBCConstantMorseTheory{0}{q}{\UBCConstant{q}{n-2}}{K^T(n-1)}\\
    &=
    1 + (q+1) K^T(n-1) + \UBCConstant{q}{n-2} + (q+2)(1 + (q+1) K^T(n-1))\UBCConstant{q}{n-2}.
  \end{align*}
  We may hence define $K^T(n) \coloneqq \max_{q \leq n-3} K^q_{n-1}$. This
  completes the proof.
\end{proof}

The discussion above admits the following symplectic analogue; the details of which we leave to the interested reader.

\begin{definition}
	\label{def:symplectic-tits-building}
	Let $\Field$ be a field and equip $\Field^{2n}$ with the standard symplectic form $\omega$. The \introduce{Tits building $\TitsBuilding{\omega, \Field}{n}$ of type $\mathtt{C}_{n-1}$ over $\Field$} is the poset of nontrivial isotropic subspaces $0 \neq \SubVectorSpace \subseteq \Field^{2n}$, where the partial order is defined by the inclusion of subspaces $\SubVectorSpace \subseteq \SubVectorSpace'$.
\end{definition}

\begin{theoremnum}
	There exists a function $K^{T^\omega}: \NaturalNumbers \to \Reals$ such that for 
	any field $\Field$ the symplectic Tits building $\TitsBuilding{\omega, \Field}{n}$ is 
	uniformly $(n-2)$-acyclic and satisfies the $q$-UBC for all $q \leq n-2$ 
	with constant $K^{T^\omega}(n)$. In particular, $\ReducedBoundedCohomologyOfSimplicialObject{q}{\TitsBuilding{\omega, \Field}{n}}{\Reals} = 0$ if $q \neq n-1$.
\end{theoremnum}

\begin{proof}
	Similar to the argument above, apply \autoref{lem:ConeUBC} and \autoref{cor:uniform-discrete-morse-theory} to refine the proof of \cite[Theorem 5.2]{BestvinaPLMorseTheory}.
\end{proof}
\section{Uniform acyclicity results for general linear groups}
\label{scn:BoundedAcyclicityProofs}

Throughout this section, $\Ring$ denotes a unital ring of finite stable rank, $\StableRank(\Ring) < \infty$, and we assume \autoref{con:general-linear-groups}.
Our goal is to prove that the
complex of $\Ring$-split injections $\SimplicialComplex^n_\bullet$ is
bounded $(n - \apply{\StableRank}{\Ring}  -1)$-acyclic, i.e.\
\autoref{thm:BoundedAcyclicityResultForComplexOfSplitInjections}. Note that this statement is empty if $\StableRank(\Ring) = \infty$.
Our proof refines an intricate connectivity argument due to van der Kallen
and heavily builds the ideas contained in
\cite{vanderkallen1980homologystabilityforlineargroups}:
We start by introducing the poset
of unimodular sequences $\UnimodularSequences{\Ring^n}$ studied in
\cite{vanderkallen1980homologystabilityforlineargroups}, and reduce
\autoref{thm:BoundedAcyclicityResultForComplexOfSplitInjections} to checking
that these posets are uniformly highly acyclic. Then, we combine the uniformly
bounded simplicial methods developed in \autoref{sec:toolbox} and van der
Kallen's strategy of proof in \cite[§2. An Acyclicity
Theorem]{vanderkallen1980homologystabilityforlineargroups} to
complete the argument. The structure of proofs and the notation 
in this section is intentionally chosen to be parallel to that of \cite[§2. An Acyclicity
Theorem]{vanderkallen1980homologystabilityforlineargroups}.

\subsection{Posets of unimodular sequences}

In this subsection, introduce the posets of unimodular sequences studied by van der Kallen
\cite{vanderkallen1980homologystabilityforlineargroups}.

\begin{definition}
For a set $V$, let $\OrderedSequences{V}$ denote the poset of ordered sequences $(v_1, \dots, v_n)$ of distinct elements in $V$ of finite length $n \geq 1$, with covering relation given by omission of one element in the sequence.
\end{definition}

We are interested in posets satisfying the following descending chain
condition.

\begin{definition}
A subposet $\Poset \subseteq \OrderedSequences{V}$ satisfies the
descending chain condition if every subsequence $w \leq s$ of a
sequence $s \in \Poset$ is contained in $\Poset$, i.e. $w \in \Poset$.
\end{definition}

We will consider the subposet of the poset of ordered sequences
$\OrderedSequences{\Ring^n}$ of vectors in $\Ring^n$ that consist of
\emph{unimodular sequences} in the sense of
\autoref{def:stable-rank}.

\begin{definition}
	The poset of unimodular sequences $\UnimodularSequences{\Ring^n}$ is the subposet of $\OrderedSequences{\Ring^n}$ consisting of unimodular sequences in $\Ring^n$.
\end{definition}

\subsection{Reduction to the poset of unimodular sequences}

In this subsection, we explain our strategy for proving \autoref{thm:BoundedAcyclicityResultForComplexOfSplitInjections}. Van der Kallen established the following connectivity theorem for the poset of unimodular sequences in \cite{vanderkallen1980homologystabilityforlineargroups}.

\begin{theoremnum}[van der Kallen]
	\label{VanDerKallenResultForUnimodularSequences}
 	 $\UnimodularSequences{\Ring^n}$ is $(n - \StableRank(\Ring) - 1)$-acyclic.
\end{theoremnum}

Our goal is to prove the following uniformly bounded \emph{refinement}
of this result.

\begin{theoremnum} \label{thm:BoundedAcyclicityResultForUnimodularSequences}
  $\UnimodularSequences{\Ring^n}$ is uniformly $(n - \StableRank(\Ring) - 1)$-acyclic.
\end{theoremnum}

To prove \autoref{thm:BoundedAcyclicityResultForUnimodularSequences}, we employ the tools developed in \autoref{sec:toolbox} to check that the poset of unimodular sequences $\UnimodularSequences{\Ring^n}$ satisfies the uniform boundary condition \emph{while} running van der Kallen's proof of \autoref{VanDerKallenResultForUnimodularSequences}.

Using \autoref{cor:qUBCimpliesAcyclic}, the following corollary shows
that
\autoref{thm:BoundedAcyclicityResultForUnimodularSequences} indeed
implies
\autoref{thm:BoundedAcyclicityResultForComplexOfSplitInjections}.

\begin{corollary}\label{cor:complex-of-split-injections}
	The complex of $\Ring$-split injections
	$\SimplicialComplex^n_\bullet$ is uniformly $(n - \StableRank(\Ring)
	- 1)$-acyclic.
\end{corollary}
\begin{proof}
	The poset of simplices of $\SimplicialComplex^n_\bullet$ is exactly $\UnimodularSequences{\Ring^n}$,
	$$\UnimodularSequences{\Ring^n} = \BarycentricSubdivion{\SimplicialComplex^n_\bullet}.$$
	Hence, this result follows from
	\autoref{thm:BoundedAcyclicityResultForUnimodularSequences} and
	\autoref{lem:BarycentricSubdivision}.
\end{proof}

\subsection{Uniform acyclicity result for posets of unimodular sequences}

This subsection contains the heart of the proof of
\autoref{thm:BoundedAcyclicityResultForUnimodularSequences}, but first
let us introduce some more notation from
\cite{vanderkallen1980homologystabilityforlineargroups}: Consider a
poset $\Poset \subseteq \UnimodularSequences{R^{\infty}}$
and let $(v_1, \dots, v_m) \in \Poset$, we write $F_{(v_1, \dots,
  v_m)}$ for the set of sequences $(w_1, \dots, w_n)$ such that $(w_1,
\dots, w_n, v_1, \dots, v_m) \in F$. We frequently use the following
observation: If $(w_1, \dots, w_n) \in F_{(v_1, \dots, v_m)}$, then
$$(F_{(v_1, \dots, v_m)})_{(w_1, \dots, w_n)} = F_{(w_1, \dots, w_n,
  v_1, \dots, v_m)}.$$

In this subsection we prove the following uniformly acyclic analogue
of \cite[Theorem
2.6]{vanderkallen1980homologystabilityforlineargroups}, which implies
\autoref{thm:BoundedAcyclicityResultForUnimodularSequences}.

\begin{theoremnum}
	\label{thm:BoundedAcyclicityResultTechnical}
	There exists functions $K^{(1)}: \NaturalNumbers^2 \to \Reals$, $K^{(2)}: \NaturalNumbers^3 \to \Reals$, $K^{(3)}: \NaturalNumbers^2 \to \Reals$, and $K^{(4)}: \NaturalNumbers^3 \to \Reals$ such that:
	
	If $n \geq 0$ and $\delta \in \{0,1\}$, then
	\begin{enumerate}
		\item \label{item-1-BoundedAcyclicityResultTechnical} $\OrderedSequences{R^n + \delta e_{n+1}} \cap \UnimodularSequences{R^{\infty}}$ is uniformly $(n - \StableRank(\Ring) - 1)$-acyclic and its UBC in degree $q \leq n - \StableRank(\Ring) - 1$ can be chosen as $\UBCConstantGLOne{n}{\StableRank(\Ring)}$.
		\item \label{item-2-BoundedAcyclicityResultTechnical} $\OrderedSequences{R^n + \delta e_{n+1}} \cap \UnimodularSequences{R^{\infty}}_v$ is uniformly $(n - \StableRank(\Ring) -1 - k)$-acyclic for any $v = (v_1, \dots, v_k) \in \UnimodularSequences{R^{\infty}}$ and its UBC constant in degree $q \leq n - \StableRank(\Ring) -1 - k$ can be chosen as $\UBCConstantGLTwo{n}{\StableRank(\Ring)}{k}$.
		\item \label{item-3-BoundedAcyclicityResultTechnical} $\OrderedSequences{(R^n + \delta e_{n+1}) \cup (R^n + \delta e_{n+1} + e_{n+2})} \cap \UnimodularSequences{R^{\infty}}$ is uniformly $(n - \StableRank(\Ring))$-acyclic and its UBC constant in degree $q \leq n - \StableRank(\Ring)$ can be chosen as $\UBCConstantGLThree{n}{\StableRank(\Ring)}$.
		\item \label{item-4-BoundedAcyclicityResultTechnical} $\OrderedSequences{(R^n + \delta e_{n+1}) \cup (R^n + \delta e_{n+1} + e_{n+2})} \cap \UnimodularSequences{R^{\infty}}_{v}$ is uniformly $(n - \StableRank(\Ring) - k)$-acyclic for any $v = (v_1, \dots, v_k) \in \UnimodularSequences{R^{\infty}}$ and its UBC constant in degree $q \leq n - \StableRank(\Ring) - k$ can be chosen as $\UBCConstantGLFour{n}{\StableRank(\Ring)}{k}$.
	\end{enumerate}
\end{theoremnum}

Note that \autoref{thm:BoundedAcyclicityResultForUnimodularSequences} follows 
from \autoref{item-1-BoundedAcyclicityResultTechnical} of \autoref{thm:BoundedAcyclicityResultTechnical} for $\delta = 0$,
because $\OrderedSequences{R^n} \cap \UnimodularSequences{R^{\infty}} = \UnimodularSequences{R^{n}}$.
The proof relies on uniformly acyclic analogues of two lemmas in
\cite{vanderkallen1980homologystabilityforlineargroups}.
The first lemma is an analogue of \cite[Lemma 2.12]{vanderkallen1980homologystabilityforlineargroups}.

\begin{lemma}
	\label{lem:LinkLemmaI}
	There exists a function $K^{\on{GLI}}: \NaturalNumbers^4 \times \Reals \to \Reals$ such that:

	Given $l, d, m \in \NaturalNumbers$ and $K \in \Reals$ and a subposet
	$\Poset \subseteq \UnimodularSequences{\Ring^{l}}$
	satisfying the
	descending chain condition together with an element
  $(v_1, \dots , v_m) \in \Poset$ with the following property: For all
  $(w_1, \dots , w_n) \in F^+((v_1, \dots , v_m))$
	the poset $\Poset_{(w_1, \dots , w_n)}$  is uniformly $(d-n)$-acyclic
	with UBC constant in degrees $q \leq d-n$ bounded by $K$.

	Then $\Link{\Poset}{(v_1, \dots , v_m)}$ is uniformly $(d-1)$-acyclic and its UBC constant in degree $q \leq d-1$ can be chosen as $\UBCConstantIntermediateGLI{l}{\StableRank(\Ring)}{m}{d}{K}$.
\end{lemma}

The proof of \autoref{lem:LinkLemmaI} is a refinement of the proof of \cite[Lemma 2.12]{vanderkallen1980homologystabilityforlineargroups}; its structure and the notation used in it is close to that of \cite[Proof of Lemma 2.12]{vanderkallen1980homologystabilityforlineargroups}.

\begin{proof}
	We start by observing that, because $\Poset$ satisfies the descending
	chain condition, it holds that $\Link{\Poset}{(v_1, \dots, v_m)}^-$
	is exactly the simplex poset of the boundary of an $(m-1)$-simplex
	$\partial \Delta^{m-1}$.
	Therefore, \autoref{lem:join-qubc} implies that
	$$\Link{\Poset}{(v_1, \dots, v_m)} = \Link{\Poset}{(v_1, \dots, v_m)}^- \ast \Link{\Poset}{(v_1, \dots, v_m)}^+$$
	is uniformly $(d-1)$-acyclic and that its UBC constant $q \leq d-1$ can be chosen as
	$$\UBCConstantIntermediateGLI{l}{\StableRank(\Ring)}{m}{d}{K} \coloneqq \max_{q \leq d-1} \UBCConstantSuspension{m-1}{q}{\UBCConstantIntermediateGLIplus{l}{\StableRank(\Ring)}{m}{d}{K}}$$
	if $\Link{\Poset}{(v_1, \dots, v_m)}^+$ is uniformly $(d - m)$-acyclic with UBC constant in degree $q \leq d-m$ bounded by $\UBCConstantIntermediateGLIplus{l}{\StableRank(\Ring)}{m}{d}{K}$.

	Using induction on $m$, we now construct the function
	$K^{\on{GLI^+}}: \NaturalNumbers^4 \times \Reals \to \Reals$ with the
	property that $\Link{\Poset}{(v_1, \dots, v_m)}^+$ is uniformly
	$(d-m)$-acyclic with UBC constant for $q \leq d - m$ choosable as
	$\UBCConstantIntermediateGLIplus{l}{\StableRank(\Ring)}{m}{d}{K}$.

	\textsl{Step A:} Let $$P_0 = \{(w_1, \dots , w_n) \in \Link{\Poset}{(v_1, \dots, v_m)}^+: w_n = v_m\}.$$ If $m = 1$, then, since $v_1 \in F^+(v_1)$, and $P_0 \cong F_{(v_1)}$ it follows from the assumptions that $P_0$ is uniformly $(d-1)$-acyclic and that its UBC constant for $q \leq d-1$ can be chosen as $\UBCConstantIntermediateGLIplusP{0}{l}{\StableRank(\Ring)}{1}{d}{K} \coloneqq K$.

	If $m > 1$, then $P_0 \cong \Link{\Poset_{(v_m)}}{(v_1, \dots,
	v_{m-1})}^+$ and we check that the induction hypothesis can be
	applied to the poset $P_0$. For this, consider $(w_1, \dots, w_n) \in
	\Poset^+_{(v_m)}((v_1, \dots, v_{m-1}))$, then it follows that
	$(\Poset_{(v_m)})_{(w_1, \dots, w_n)} = \Poset_{(w_1, \dots , w_n,
	v_m)}$ and $(w_1, \dots , w_n, v_m) \in \Poset^+((v_1, \dots, v_m))$.
	Hence, by assumption, $(\Poset_{(v_m)})_{(w_1, \dots, w_n)}$ is
	uniformly $((d-1) - n)$-acyclic and satisfies the q-UBC for $q \leq
	(d - 1) - n$ with constant $K$.

	By applying the induction hypothesis for $m - 1 < m$,
	$\Poset_{(v_m)}$ and for $d-1 \neq d$, we obtain a
	function
	$$K^{\on{GLI^+}}: \NaturalNumbers^2 \times \{1, \dots, m-1\} \times \NaturalNumbers \times \Reals \to \Reals$$
	such that $P_0 \cong \Link{\Poset_{(v_m)}}{(v_1, \dots, v_{m-1})}^+$
	is uniformly $((d-1)-(m-1)) = (d - m)$-acyclic and its UBC constant
	in degree $q \leq d - m$ is choosable as
	$\UBCConstantIntermediateGLIplusP{0}{l}{\StableRank(\Ring)}{m}{d}{K}
	\coloneqq
	\UBCConstantIntermediateGLIplus{l}{\StableRank(\Ring)}{m-1}{d-1}{K}$.

	\textsl{Step B:} Let $$P_1 = \{(w_1, \dots , w_r) \in \Poset : \text{ for some } 1 \leq n \leq r \text{ one has } (w_1, \dots, w_n) \in P_0\}.$$
	Then, $P_0 \hookrightarrow P_1$ is deformation retract in the sense of \autoref{lem:PosetDeformation}. Therefore, $P_1$ is uniformly $(d - m)$-acyclic and satisfies the q-UBC for $q \leq d-m$ with constant $\UBCConstantIntermediateGLIplusP{1}{l}{\StableRank(\Ring)}{m}{d}{K} \coloneqq (d-m + 1) + \UBCConstantIntermediateGLIplusP{0}{l}{\StableRank(\Ring)}{m}{d}{K}$.

	\textsl{Step C:} Consider $P_1 \hookrightarrow \Link{\Poset}{(v_1, \dots, v_m)}^+$. Observe that $$\Link{\Poset}{(v_1, \dots, v_m)}^+ \setminus P_1$$
	consists of elements $(v_1, \dots, v_m, z_1, \dots, z_j)$ with $j \geq 1$. We filter $\Link{\Poset}{(v_1, \dots, v_m)}^+$ by the subposets $Q_0 = P_1$ and $$Q_r = P_1 \cup \{(v_1, \dots, v_m, z_1, \dots, z_j) : 1 \leq j \leq r\}.$$
	Note that $$\Link{\Poset}{(v_1, \dots, v_m)}^+ = \bigcup_{r \in \mathbb{N}} Q_r.$$

	We claim that this filtration is \emph{finite}: Recall that $\StableRank(\Ring) < \infty$. In \autoref{rem:MaximalLengthOfUnimodularSequences} we showed that \autoref{lem:StableRankProperties} implies that if $\Ring^k \to \Ring^l$ is an $\Ring$-split injection then $k \leq l + \StableRank(\Ring)$. Hence, the length of a unimodular sequence in $\Ring^l$ is at most $l + \StableRank(\Ring)$. This implies that
	$$\Link{\Poset}{(v_1, \dots, v_m)}^+ = Q_{l + \StableRank(\Ring)-m} = \bigcup_{0 \leq r \leq l + \StableRank(\Ring)-m} Q_r.$$

	We now use a second induction (on $r$) as well as \autoref{lem:join-qubc} and \autoref{cor:uniform-discrete-morse-theory} to study how passing from $Q_r$ to $Q_{r+1}$ affects the uniform acyclicity and the UBC constant. Along the way, we define functions $K^{\on{GLI^+}}_{Q_r}$ controlling the UBC constant of $Q_r$ such that we may eventually complete the induction on $m$ by setting
	$$\UBCConstantIntermediateGLIplus{l}{\StableRank(\Ring)}{m}{d}{K} \coloneqq \UBCConstantIntermediateGLIplusQ{l + \StableRank(\Ring)-m}{l}{\StableRank(\Ring)}{m}{d}{K}.$$

	We begin by collecting several important observations:
	We note that $Q_{r+1} \setminus Q_r$ is discrete. Hence, $Q_{r+1}$ is obtained from $Q_r$ by gluing each
	$$(\vec v, \vec z) = (v_1, \dots, v_m, z_1, \dots, z_{r+1}) \in Q_{r+1} \setminus Q_r$$
	to $Q_r$ along $\Link{Q_r}{(\vec v, \vec z)}$ in the sense of \autoref{def:PosetObtainedByGluing}.

	Therefore, if it holds that
  \begin{itemize}
    \item $Q_r$ is uniformly $(d-m)$-acyclic and its UBC constant in degree $q
    \leq d-m$ is bounded by
    $\UBCConstantIntermediateGLIplusQ{r}{l}{\StableRank(\Ring)}{m}{d}{K}$
    \item for all $(\vec v, \vec z) \in Q_{r+1} \setminus Q_r$:
    $\Link{Q_r}{(\vec v, \vec z)}$ is uniformly $(d-m-1)$-acyclic with
    UBC constant for $q \leq d-m-1$ bounded by
    $\UBCConstantIntermediateGLIplusLkQ{r}{l}{\StableRank(\Ring)}{m}{d}{K}$
  \end{itemize}
  then \autoref{cor:uniform-discrete-morse-theory} implies
  that there exists a function
	$$K_{Q_{r+1}}^{\on{GLI^+}}: \NaturalNumbers^2 \times \{1, \dots, m\} \times \NaturalNumbers \times \Reals \to \Reals$$
	such that $Q_{r+1}$ is uniformly $(d-m)$-acyclic with UBC constant in degree $q \leq d-m$ choosable as
	\begin{align*}
	\UBCConstantIntermediateGLIplusQ{r+1}{l}{\StableRank(\Ring)}{m}{d}{K}
	&\coloneqq\\
	&\UBCConstantMorseTheory{0}{q}{\UBCConstantIntermediateGLIplusQ{r}{l}{\StableRank(\Ring)}.{m}{d}{K}}{\UBCConstantIntermediateGLIplusLkQ{r}{l}{\StableRank(\Ring)}{m}{d}{K}}
	\end{align*}

	The first item holds by the induction hypothesis (for $r$). We now work towards the second item:

	Note that it holds that
	$$\Link{Q_r}{(\vec v, \vec z)} = \Link{Q_r}{(\vec v, \vec z)}^- \ast \Link{Q_r}{(\vec v, \vec z)}^+.$$
	Therefore, if each $\Link{Q_r}{(\vec v, \vec z)}^+$ is uniformly
	$((d-r-2)-(m-1))$-acyclic and satisfies the UBC in degree $q \leq (d
	- r - 2) - (m-1)$ with constant
	$
    \UBCConstantIntermediateGLIplusLkQplus
      {r}
      {l}
      {\StableRank(\Ring)}
      {m}
      {d}
      {K}
  $, 
  then \autoref{lem:join-qubc} implies that
	$\Link{Q_r}{(\vec v, \vec z)}$ is uniformly $(d-m-1)$-acyclic and
	satisfies the q-UBC for $q \leq d-m-1$ with constant
	$$
	\UBCConstantIntermediateGLIplusLkQ{r}{l}{\StableRank(\Ring)}{m}{d}{K} \coloneqq
	\max_{q \leq d-m-1} \UBCConstantSuspension{m+r}{q}{\UBCConstantIntermediateGLIplusLkQplus{r}{l}{\StableRank(\Ring)}{m}{d}{K}}
	$$

	\textsl{Step D:} In this final step, we describe how to construct the desired function
	$$K^{\on{GLI^+}}_{\on{Lk}_{Q_r}^+}: \NaturalNumbers^2 \times \{1, \dots, m\} \times \NaturalNumbers \times \Reals \to \Reals$$
	using the induction hypothesis on $m$.

	If $m = 1$, then
	$$\{(w_1, \dots, w_s, v_1, \vec z) : (w_1, \dots, w_s) \in F_{(v_1, \vec z)}\} \cong \Poset_{(v_1, \vec z)}.$$
	Since $(v_1, \vec z) \in \Poset^+((v_1))$, it follows that
	$\Poset_{(v_1, \vec z)}$ is uniformly $(d-r-2)$-acyclic and satisfies
	the q-UBC for $q \leq (d - r - 2)$ with constant $K$ by assumption.
	Furthermore,
	$$\{(w_1, \dots, w_s, v_1, \vec z) : (w_1, \dots, w_s) \in F_{(v_1, \vec z)}\} \hookrightarrow \Link{Q_r}{(v_1, \vec z)}^+$$
	is a deformation retract in the sense of
	\autoref{lem:PosetDeformation}, because the length of $\vec z$ is
	$r+1>r$.
	Therefore,
	\autoref{lem:PosetDeformation} implies that $\Link{Q_r}{(v_1, \vec
	z)}^+$ is uniformly $(d-r-2)$-acyclic and satisfies the q-UBC for $q
	\leq (d - r - 2)$ with constant
	$$\UBCConstantIntermediateGLIplusLkQplus{r}{l}{\StableRank(\Ring)}{1}{d}{K} \coloneqq (d - r - 2) + K.$$

	If $m > 1$, then
	\begin{align*}
		&\{(w_1, \dots, w_s, v_m, \vec z) : (w_1, \dots, w_s) \in \Link{\Poset_{(v_m,\vec z)}}{(v_1, \dots, v_{m-1})}^+\}\\
		& \cong \Link{\Poset_{(v_m,\vec z)}}{(v_1, \dots, v_{m-1})}^+.
	\end{align*}
	and we check that the induction hypothesis applies to $\Link{\Poset_{(v_m,\vec z)}}{(v_1, \dots, v_{m-1})}^+$: For this, 
	let $(w_1, \dots, w_n) \in \Poset_{(v_m,\vec z)}((v_1, \dots, v_{m-1}))$, then we have $(\Poset_{(v_m,\vec z)})_{(w_1, \dots, w_n)} = \Poset_{(w_1, \dots , w_n, v_m,\vec z)}$ and $(w_1, \dots , w_n, v_m,\vec z) \in \Poset^+((v_1, \dots, v_m))$. Hence, by assumption, $(\Poset_{(v_m,\vec z)})_{(w_1, \dots, w_n)}$ is uniformly $(d - (n + r + 2)) = ((d-r-2) - n)$-acyclic and satisfies the q-UBC for $q \leq (d-r-2) - n$ with constant $K$.

	By applying the induction hypothesis for $m-1 < m$,
	$\Poset_{(v_m,\vec z)}$ and $d-r-2 \neq d$, we obtain
	a function
	$$K^{\on{GLI^+}}: \NaturalNumbers^2 \times \{1, \dots, m-1\} \times \NaturalNumbers \times \Reals \to \Reals$$
	such that $\Link{\Poset_{(v_m,\vec z)}}{(v_1, \dots, v_{m-1})}^+$ is
	uniformly $((d-r-2)-(m-1))$-acyclic and satisfies the q-UBC $q \leq
	(d-r-2)-(m-1)$ with constant
	\[\UBCConstantIntermediateGLIplus{l}{\StableRank(\Ring)}{m-1}{d-r-2}{K}.\]

	To complete the argument we use that
	$$\{(w_1, \dots, w_s, v_m, \vec z) : (w_1, \dots, w_s) \in \Link{\Poset_{(v_m,\vec z)}}{(v_1, \dots, v_{m-1})}^+\} \hookrightarrow \Link{Q_r}{(\vec v, \vec z)}^+$$
	is deformation retract in the sense of \autoref{lem:PosetDeformation}. Therefore, \autoref{lem:PosetDeformation} implies that $\Link{Q_r}{(\vec v, \vec z)}^+$ is uniformly $((d-r-2)-(m-1))$-acyclic and satisfies the q-UBC for  $q \leq (d-r-2)-(m-1)$ with constant
	$$\UBCConstantIntermediateGLIplusLkQplus{r}{l}{\StableRank(\Ring)}{m}{d}{K} \coloneqq (d - r - 2) - (m - 1) + 1 + \UBCConstantIntermediateGLIplus{l}{\StableRank(\Ring)}{m-1}{d-r-2}{K}$$
	as claimed. \qedhere
\end{proof}

The second lemma is the uniform acyclic analogue of \cite[Lemma 2.13]{vanderkallen1980homologystabilityforlineargroups}.

\begin{lemma}
	\label{lem:LinkLemmaII}
	There exists a function $K^{\on{GLII}}: \NaturalNumbers^3 \times \Reals \to \Reals$ such that:

	Given $l, d \in \NaturalNumbers$ and $K \in \Reals$ as well as a
	poset $\Poset \subseteq \UnimodularSequences{\Ring^l}$ satisfying the
	descending chain condition and a subset $X \subseteq \Ring^\infty$,
	then the following holds:
	\begin{enumerate}
		\item \label{item-1-LinkLemmaII} If the poset
		$\OrderedSequences{X} \cap \Poset$ is uniformly $d$-acyclic and
		satisfies the q-UBC for $q \leq d$ with constant $K$ and if further
   		for all $(v_1, \dots, v_k) \in F \setminus
		\OrderedSequences{X}$ the poset $F_{(v_1, \dots, v_k)} \cap
		\OrderedSequences{X}$ is uniformly $(d-k)$-acyclic and satisfies
		the q-UBC for $q \leq d-k$ with constant $K$,
		then $\Poset$ is uniformly $d$-acyclic and its UBC
		constant in degree $q \leq d$ can be chosen as
		$\UBCConstantIntermediateGLII{l}{\StableRank(\Ring)}{d}{K}$.
		\item \label{item-2-LinkLemmaII} If for all $(v_1, \dots, v_k) \in
		\Poset \setminus \OrderedSequences{X}$ the poset $F_{(v_1, \dots,
		v_k)} \cap \OrderedSequences{X}$ is uniformly $((d+1)-k)$-acyclic
		and satisfies the q-UBC for $q \leq (d+1)-k$ with constant $K$ and
		if further there is an element $(y_0) \in F$ such that $F
		\cap \OrderedSequences{X} \subseteq F_{(y_0)}$, then $F$ is
		uniformly $(d+1)$-acyclic and its UBC constant in degree $q \leq
		d+1$ can be chosen as
		$\UBCConstantIntermediateGLII{l}{\StableRank(\Ring)}{d}{K}$.
	\end{enumerate}
\end{lemma}

The proof of \autoref{lem:LinkLemmaII} is split in two parts concerning \autoref{item-1-LinkLemmaII} and \autoref{item-2-LinkLemmaII}; the second part depends on the first. The structure and the notation used in these arguments is again close to that of \cite[Proof of Lemma 2.13]{vanderkallen1980homologystabilityforlineargroups}.

\begin{proof}[Proof of \autoref{item-1-LinkLemmaII} of \autoref{lem:LinkLemmaII}]
	We start with the following observation: Let $P_0 = \{(v_0, \dots, v_k) \in F: \text{ at least one } v_i \text{ is in } X\}$. Then, $$\Poset \cap \OrderedSequences{X} \hookrightarrow P_0$$ is a deformation retract in the sense of \autoref{lem:PosetDeformation} and, hence, $P_0$ is uniformly $d$-acyclic and satisfies the q-UBC for $q \leq d$ with constant $$\UBCConstantIntermediateGLIIPOne{0}{l}{\StableRank(\Ring)}{d}{K} \coloneqq d + K.$$

	Similar to the previous lemma, we filter $\Poset$ by the subposets
	$$P_r = P_0 \cup \{(w_1, \dots, w_j) \in F : 1 \leq j \leq r\}.$$
	Note that $$\Poset = \bigcup_{r \in \mathbb{N}} P_r.$$
	We claim that this filtration is \emph{finite}: Recall that $\StableRank(\Ring) < \infty$. As explained in  \autoref{rem:MaximalLengthOfUnimodularSequences}, it follows from \autoref{lem:StableRankProperties} that if $\Ring^k \to \Ring^l$ is an $\Ring$-split injection then $k \leq l + \StableRank(\Ring)$. Hence, the length of a unimodular sequence in $\Ring^l$ is at most $l + \StableRank(\Ring)$. This implies that
	$$\Poset = P_{l+\StableRank(\Ring)} = \bigcup_{0 \leq r \leq l+\StableRank(\Ring)} P_r.$$

	Hence, we can apply \autoref{cor:uniform-discrete-morse-theory} to study how passing from $P_r$ to $P_{r+1}$ affects uniform acyclicity and the uniform boundary condition. We continue with an induction on $r$.

	For the induction step, passing from $r$ to $r+1$, we may assume that we already constructed a function
	$$K_{P_r, (1)}^{\on{GLII}}: \NaturalNumbers^3 \times \Reals \to \Reals$$
	such that $P_r$ is uniformly $d$-acyclic with q-UBC constant for $q \leq d$ choosable as $\UBCConstantIntermediateGLIIPOne{r}{l}{\StableRank(\Ring)}{d}{K}$.

	We start observe that $P_{r+1} \setminus P_r$ is discrete. Hence, $P_{r+1}$ is obtained from $P_r$ by gluing each
	$$(w_1, \dots, w_{r+1}) \in P_{r+1} \setminus P_r$$
	to $P_r$ along $\Link{P_r}{(w_1, \dots, w_{r+1})}$ in the sense of \autoref{def:PosetObtainedByGluing}.

	In order to invoke \autoref{cor:uniform-discrete-morse-theory}, we need to study $\Link{P_r}{(w_1, \dots, w_{r+1})}$. Set $Q = F \cap \OrderedSequences{X \cup \{w_1, \dots, w_{r+1}\}}$. Then, $\Link{Q}{(w_1, \dots, w_{r+1})}$ is uniformly $(d-1)$-acyclic and satisfies the q-UBC for $q \leq d-1$ with constant $\UBCConstantIntermediateGLI{l}{\StableRank(\Ring)}{r+1}{d}{K}$ by an application of the previous \autoref{lem:LinkLemmaI}. Since
	$$\Link{Q}{(w_1, \dots, w_{r+1})} \hookrightarrow \Link{P_{r}}{(w_1, \dots, w_{r+1})}$$
	is a poset deformation retract in the sense of \autoref{lem:PosetDeformation}, we can conclude that $\Link{P_{r}}{(w_1, \dots, w_{r+1})}$ is uniformly $(d-1)$-acyclic and satisfies the q-UBC for $q \leq d-1$ with constant $\UBCConstantIntermediateGLIILk{r}{l}{\StableRank(\Ring)}{d}{K} \coloneqq d + \UBCConstantIntermediateGLI{l}{\StableRank(\Ring)}{r+1}{d}{K}$.

	The induction hypothesis (on $r$) and \autoref{cor:uniform-discrete-morse-theory} therefore imply that $P_{r+1}$ is uniformly $d$-acyclic and satisfies the q-UBC condition for $q \leq d$ with constant
	\begin{align*}
		\UBCConstantIntermediateGLIIPOne{r+1}{l}{\StableRank(\Ring)}{d}{K} &\coloneqq\\
		&\max_{q \leq d} \UBCConstantMorseTheory{0}{q}{\UBCConstantIntermediateGLIIPOne{r}{l}{\StableRank(\Ring)}{d}{K}}{\UBCConstantIntermediateGLIILk{r}{l}{\StableRank(\Ring)}{d}{K}}.
	\end{align*}
	This completes the induction step. We hence showed that $\Poset$ is uniformly $d$-acyclic and satisfies the q-UBC condition for $q \leq d$ with constant $$\UBCConstantIntermediateGLIIOne{l}{\StableRank(\Ring)}{d}{K} \coloneqq \UBCConstantIntermediateGLIIPOne{l + \StableRank(\Ring)}{l}{\StableRank(\Ring)}{d}{K}.$$

	We highlight that $K^{\on{GLII}}_{(1)}$ is \emph{not} the function occurring in \autoref{lem:LinkLemmaII}. In the proof of \autoref{item-2-LinkLemmaII}, the function $K^{\on{GLII}}$ is constructed in such a way that $$\UBCConstantIntermediateGLIIOne{l}{\StableRank(\Ring)}{d}{K} \leq \UBCConstantIntermediateGLII{l}{\StableRank(\Ring)}{d}{K}.$$
	From this the claim in \autoref{item-1-LinkLemmaII} follows.
\end{proof}
\begin{proof}[Proof of \autoref{item-2-LinkLemmaII} of \autoref{lem:LinkLemmaII}]
	The key in this argument is the following observation due to van der Kallen: Since $y_0 \not\in \OrderedSequences{X}$, the assumptions imply that
	$$\OrderedSequences{X} \cap \Poset = \OrderedSequences{X} \cap \Poset_{(y_0)}$$
	is uniformly $d$-acyclic and satisfies the q-UBC for $q \leq d$ with constant $K$. Apart from one step that exploits this observation, the argument is exactly as that for \autoref{item-1-LinkLemmaII} of \autoref{lem:LinkLemmaII}. Taking the slight modification of the argument into account, we construct the function $K^{\on{GLII}}$ inductively in such a way that it always dominates the $K^{\on{GLII}}_{(1)}$ in the proof of \autoref{item-1-LinkLemmaII}. This is the reason why one can use the \emph{same} function to estimate the UBC constants in both part \autoref{item-1-LinkLemmaII} and \autoref{item-2-LinkLemmaII}.

	Define $P_r$ exactly as in the proof of \autoref{item-1-LinkLemmaII}. Then, using the key observation above and exactly as in the proof of \autoref{item-1-LinkLemmaII}, $P_0$ is uniformly $d$-acyclic and satisfies the q-UBC for $q \leq d$ with constant $\UBCConstantIntermediateGLIIPOne{0}{l}{\StableRank(\Ring)}{d}{K} = d + K$.

	Again, $P_{r+1}$ is obtained from $P_r$ by gluing $(w_1, \dots, w_{r+1}) \in P_{r+1} \setminus P_r$ to $P_r$ along the link $\Link{P_r}{(w_1, \dots, w_{r+1})}$ in the sense of \autoref{def:PosetObtainedByGluing}.

	Using the assumption in \autoref{item-2-LinkLemmaII} that for all $(v_1, \dots, v_k) \in F \setminus \OrderedSequences{X}$ the poset $F_{(v_1, \dots, v_k)} \cap \OrderedSequences{X}$ is uniformly $((d+1)-k)$-acyclic and satisfies the q-UBC for $q \leq (d+1)-k$ with constant $K$ (note that this is one degree more than in \autoref{item-1-LinkLemmaII}), exactly the same argument as in \autoref{item-1-LinkLemmaII} shows that $\Link{P_r}{(w_1, \dots, w_{r+1})}$ is uniformly $d$-acyclic and satisfies the q-UBC for $q \leq d$ with constant
	$$\UBCConstantIntermediateGLIILk{r}{l}{\StableRank(\Ring)}{d+1}{K} = (d+1) + \UBCConstantIntermediateGLI{l}{\StableRank(\Ring)}{r+1}{d+1}{K}.$$
	(note that this one degree more than in \autoref{item-1-LinkLemmaII}).

	We now invoke \autoref{lem:UBCindegreenplusone}: The previous paragraph implies that $\Link{P_0}{(y_0)}$ is uniformly $d$-acyclic and that its UBC constant in degree $q \leq d$ can be chosen as $\UBCConstantIntermediateGLIILk{1}{l}{\StableRank(\Ring)}{d+1}{K}$.
	Now, $$\psi: \OrderedSequences{X} \cap \Poset \to \Link{P_0}{(y_0)}: \vec v \mapsto (\vec v,y_0)$$
	is an embedding of posets, and $\OrderedSequences{X} \cap \Poset$ is a deformation retraction of $P_0$ in the sense of \autoref{lem:PosetDeformation} via a retraction, whose restriction to $\im(\psi)$ is inverse to $\psi$. Therefore \autoref{lem:UBCindegreenplusone} implies that $P_0 \cup \{(y_0)\}$ is uniformly $(d+1)$-acyclic and satisfies the q-UBC for $q \leq d+1$ with constant
	$$K_{P_0 \cup \{(y_0)\}}(l,\StableRank(\Ring),d,K) \coloneqq \max_{q \leq d+1} \UBCConstantTechnicalLemmaII{q}{\UBCConstantIntermediateGLIILk{1}{l}{\StableRank(\Ring)}{d+1}{K}}.$$

	Passing from $P_0 \cup \{(y_0)\}$ to $P_1$, and from $P_r$ to $P_{r+1}$ one adds cones over links which are uniformly $d$-acyclic and satisfy the q-UBC for $q \leq d$ with constants choosable as $\UBCConstantIntermediateGLIILk{0}{l}{\StableRank(\Ring)}{d+1}{K}$ and $\UBCConstantIntermediateGLIILk{r}{l}{\StableRank(\Ring)}{d+1}{K}$, respectively. Hence, \autoref{cor:uniform-discrete-morse-theory} implies that $P_1$ is uniformly $(d+1)$-acyclic and satisfies the q-UBC for $q \leq d+1$ with constant
	\begin{align*}
	\UBCConstantIntermediateGLIIPTwo{1}{l}{\StableRank(\Ring)}{d}{K} &\coloneqq \\
	&\UBCConstantMorseTheory{0}{q}{K_{P_0 \cup \{(y_0)\}}(l,\StableRank(\Ring),d,K)}{\UBCConstantIntermediateGLIILk{0}{l}{\StableRank(\Ring)}{d+1}{K}}
	\end{align*}
	Then, we continue our induction passing from $r$ to $r+1$ as in the proof of \autoref{item-1-LinkLemmaII} using $K_{P_1, (2)}^{\on{GLII}}$ instead of $K_{P_1, (1)}^{\on{GLII}}$, and conclude that $P_{r+1}$ is uniformly $(d+1)$-acyclic and satisfies the q-UBC for $q \leq d+1$ with constant
	\begin{align*}
		&\UBCConstantIntermediateGLIIPTwo{r+1}{l}{\StableRank(\Ring)}{d}{K} \coloneqq\\
		&\max_{q \leq d+1} \UBCConstantMorseTheory{0}{q}{\UBCConstantIntermediateGLIIPTwo{r}{l}{\StableRank(\Ring)}{d}{K}}{\UBCConstantIntermediateGLIILk{r}{l}{\StableRank(\Ring)}{d+1}{K}}.
	\end{align*}
	It follows that $\Poset = P_{l+\StableRank(\Ring)}$ is uniformly $(d+1)$-acyclic and that its UBC constant in degree $q \leq d+1$ can be chosen as
	$\UBCConstantIntermediateGLIITwo{l}{\StableRank(\Ring)}{d}{K} \coloneqq \UBCConstantIntermediateGLIIPTwo{l+\StableRank(\Ring)}{l}{\StableRank(\Ring)}{d}{K}$.
	To complete the proof, we simply define
	\begin{equation*}
	\UBCConstantIntermediateGLII{l}{\StableRank(\Ring)}{d}{K}
	\coloneqq
	\max\{\UBCConstantIntermediateGLIIOne{l}{\StableRank(\Ring)}{d}{K}, \UBCConstantIntermediateGLIITwo{l}{\StableRank(\Ring)}{d}{K}\}. \qedhere
	\end{equation*}
\end{proof}

With these two lemmas at hand, we are ready to prove
\autoref{thm:BoundedAcyclicityResultTechnical}.
This proof is the uniform acyclic analogue of \cite[§ 2.14.-18.]{vanderkallen1980homologystabilityforlineargroups}; its structure and the notation used in it is close to that in \cite[§ 2.14.-18.]{vanderkallen1980homologystabilityforlineargroups}. Further helpful details are given in \cite[Proof of Theorem 2.4]{friedrich2016homologicalstabilityofautomorphismgroupsofquadraticmodulesandmanifolds}.

\begin{proof}[Proof of \autoref{thm:BoundedAcyclicityResultTechnical}]
	As described in \cite[§2.14.]{vanderkallen1980homologystabilityforlineargroups}, we prove \autoref{thm:BoundedAcyclicityResultTechnical} by induction. For each poset, we either prove directly that is uniformly $(-1)$-acyclic --- in which case we only need to check that the poset is nonempty and set the UBC constant in degree $(-1)$ equal to zero --- or we invoke \autoref{lem:LinkLemmaII}. The induction is on the ``size'' of the poset: A poset as in \autoref{item-1-BoundedAcyclicityResultTechnical} of \autoref{thm:BoundedAcyclicityResultTechnical} has size $2n$; a poset as in \autoref{item-2-BoundedAcyclicityResultTechnical} of \autoref{thm:BoundedAcyclicityResultTechnical} has size $2n - k$; a poset as in \autoref{item-3-BoundedAcyclicityResultTechnical} of \autoref{thm:BoundedAcyclicityResultTechnical} has size $2n+1$; a poset as in \autoref{item-4-BoundedAcyclicityResultTechnical} of \autoref{thm:BoundedAcyclicityResultTechnical} has size $2n-k+1$. In particular, during the induction step, the induction hypothesis applies to all posets of smaller size in this sense.
	For negative sizes, all four claims are empty and hence trivially hold for the UBC constant $0$ in every case.

	\textsl{Proof of \autoref{item-1-BoundedAcyclicityResultTechnical}, i.e.\ the analogue of \cite[§2.15.]{vanderkallen1980homologystabilityforlineargroups}:} Let $\Poset \coloneqq \OrderedSequences{\Ring^{n} + \delta e_{n+1}} \cap \UnimodularSequences{\Ring^{\infty}}$ and $d  \coloneqq n - \StableRank(\Ring) - 1$. We need to define the value of the function $K^{(1)}$ at $(n, \StableRank(\Ring))$ such that $\Poset$ is uniformly $d$-acyclic and satisfies the q-UBC for $q \leq d$ with constant $\UBCConstantGLOne{n}{\StableRank(\Ring)}$.

	If $n = 0$, then $d < -1$ because $\StableRank(\Ring) > 0$  (see \autoref{def:stable-rank}). Hence, the condition is empty and trivially satisfied for $\UBCConstantGLOne{0}{\StableRank(\Ring)} \coloneqq 0$.

	For $n > 0$, we invoke the induction hypothesis and use \autoref{item-1-LinkLemmaII} of \autoref{lem:LinkLemmaII}: To see this, set $X \coloneqq (\Ring^{n-1} + e_{n+1} \delta) \cup (\Ring^{n-1} + e_n + e_{n+1} \delta)$. Then $\OrderedSequences{X} \cap \Poset = \OrderedSequences{X} \cap \UnimodularSequences{R^{\infty}}$ is uniformly $((n-1) - \StableRank(\Ring)) = d$-acyclic and satisfies the q-UBC for $q \leq d$ with constant
	$$\UBCConstantGLThree{n-1}{\StableRank(\Ring)}$$
	by \autoref{item-3-BoundedAcyclicityResultTechnical} of the induction hypothesis, since up to a change of coordinates $X$ is equal to $(\Ring^{n-1} + e_{n} \delta) \cup (\Ring^{n-1} + e_n \delta + e_{n+1})$.
	Similarly, if $(v_1, \dots, v_k) \in F \setminus \OrderedSequences{X}$ the poset $\OrderedSequences{X} \cap F_{(v_1, \dots, v_k)} = \OrderedSequences{X} \cap \UnimodularSequences{\Ring^{\infty}}_{(v_1, \dots, v_k)}$ is uniformly $((n-1) - \StableRank(\Ring) - k) = (d-k)$-acyclic and satisfies the q-UBC for $q \leq d-k$ with constant $$\UBCConstantGLFour{n-1}{\StableRank(\Ring)}{k}$$
	by \autoref{item-4-BoundedAcyclicityResultTechnical} of the induction hypothesis. Notice that, because $(v_1, \dots, v_k) \in \Poset$ and $F \subseteq \Ring^{n+1}$, it follows from \autoref{rem:MaximalLengthOfUnimodularSequences} that $k \leq n+1+\StableRank(\Ring)$. We set
	$$K^{(4)}_{\max}(n-1, \StableRank(\Ring)) \coloneqq \max_{k \in \{1, \dots, n+1+\StableRank(\Ring)\}} K^{(4)}(n-1, \StableRank(\Ring), k).$$
	Invoking \autoref{item-1-LinkLemmaII} of \autoref{lem:LinkLemmaII} for
	$$K = K(n,\StableRank(\Ring)) \coloneqq \max \{\UBCConstantGLThree{n-1}{\StableRank(\Ring)}, K^{(4)}_{\max}(n-1, \StableRank(\Ring))\} \text{ and } l = n+1,$$
	we conclude that $\Poset$ is uniformly $(n - \StableRank(\Ring) - 1) = d$-acyclic and satisfies the q-UBC for $q \leq d$ with constant
	$$\UBCConstantGLOne{n}{\StableRank(\Ring)} \coloneqq \UBCConstantIntermediateGLII{n+1}{\StableRank(\Ring)}{n - \StableRank(\Ring) - 1}{K(n,\StableRank(\Ring))}.$$

	\textsl{Proof of \autoref{item-2-BoundedAcyclicityResultTechnical}, i.e.\ the analogue of \cite[§2.16.]{vanderkallen1980homologystabilityforlineargroups}:}
	In \cite{vanderkallen1980homologystabilityforlineargroups} this argument is omitted since it is similar to the previous one; we present it to clarify the construction of $K^{(2)}: \NaturalNumbers^3 \to \Reals$.

	Let $\Poset \coloneqq \OrderedSequences{\Ring^{n} + \delta e_{n+1}} \cap \UnimodularSequences{\Ring^{\infty}}_{(v_1, \dots, v_k)}$ and $d  \coloneqq n - \StableRank(\Ring) - 1 - k$. We need to define the value of the function $K^{(2)}$ at $(n, \StableRank(\Ring),k)$ such that $\Poset$ is uniformly $d$-acyclic and satisfies the q-UBC for $q \leq d$ with constant $\UBCConstantGLTwo{n}{\StableRank(\Ring)}{k}$.

	For $k = 0$, we set $\UBCConstantGLTwo{n}{\StableRank(\Ring)}{k} \coloneqq \UBCConstantGLOne{n}{\StableRank(\Ring)}$. Now, assume $k > 0$.

	If $n \leq \StableRank(\Ring)$, then $d < -1$ because $\StableRank(\Ring) > 0$ and $k > 0$. Hence, if $n \leq \StableRank(\Ring)$ the condition is empty and trivially satisfied for $\UBCConstantGLTwo{n}{\StableRank(\Ring)}{k} \coloneqq 0$.

	For $n > \StableRank(\Ring)$, we apply the induction hypothesis and \autoref{item-1-LinkLemmaII} of \autoref{lem:LinkLemmaII}: To see this, set $X \coloneqq (\Ring^{n-1} + e_{n+1} \delta) \cup (\Ring^{n-1} + e_n + e_{n+1} \delta)$. Then, $\OrderedSequences{X} \cap \Poset = \OrderedSequences{X} \cap \UnimodularSequences{\Ring^{\infty}}_{(v_1, \dots, v_k)}$ is uniformly $((n-1) - \StableRank(\Ring) - k) = d$-acyclic and satisfies the q-UBC for $q \leq d$ with constant
	$$\UBCConstantGLFour{n-1}{\StableRank(\Ring)}{k}$$
	by \autoref{item-4-BoundedAcyclicityResultTechnical} of the induction hypothesis, since up to a change of coordinates $X$ is equal to $(\Ring^{n-1} + e_{n} \delta) \cup (\Ring^{n-1} + e_n \delta + e_{n+1})$. Similarly, if $(w_1, \dots, w_{k'}) \in F \setminus \OrderedSequences{X}$ the poset $\OrderedSequences{X} \cap F_{(w_1, \dots, w_{k'})} = \OrderedSequences{X} \cap \UnimodularSequences{\Ring^{\infty}}_{(w_1, \dots, w_{k'}, v_1, \dots, v_{k})}$ is uniformly $((n-1) - \StableRank(\Ring) - k - k') = (d-k')$-acyclic and satisfies the q-UBC for $q \leq d-k'$ with constant $$\UBCConstantGLFour{n-1}{\StableRank(\Ring)}{k + k'}$$
	by \autoref{item-4-BoundedAcyclicityResultTechnical} of the induction hypothesis. Note that, because $(w_1, \dots, w_{k'}) \in \Poset$ and $F \subseteq \Ring^{n+1}$, it follows from \autoref{rem:MaximalLengthOfUnimodularSequences} that $k' \leq n+1+\StableRank(\Ring)$. We set
	$$K^{(4)}_{\max}(n-1, \StableRank(\Ring), k) \coloneqq \max_{k' \in \{1, \dots, n+1+\StableRank(\Ring)\}} K^{(4)}(n-1, \StableRank(\Ring), k + k').$$
	Invoking \autoref{item-1-LinkLemmaII} of \autoref{lem:LinkLemmaII} for
	$$K = K(n,\StableRank(\Ring),k) \coloneqq \max \{K^{(4)}(n - 1, \StableRank(\Ring), k), K^{(4)}_{\max}(n-1, \StableRank(\Ring), k)\} \text{ and } l = n+1,$$
	we conclude that $\Poset$ is uniformly $(n - \StableRank(\Ring) - 1 - k) = d$-acyclic and satisfies the q-UBC for $q \leq d$ with constant
	$$\UBCConstantGLTwo{n}{\StableRank(\Ring)}{k} \coloneqq \UBCConstantIntermediateGLII{n+1}{\StableRank(\Ring)}{n - \StableRank(\Ring) - 1 - k}{K(n,\StableRank(\Ring),k)}.$$

	\textsl{Proof of \autoref{item-3-BoundedAcyclicityResultTechnical}, i.e.\ the analogue of \cite[§2.17.]{vanderkallen1980homologystabilityforlineargroups}:}
	Let $\Poset = \OrderedSequences{(\Ring^{n} + e_{n+1} \delta) \cup (\Ring^{n} + e_{n+1} \delta + e_{n+2})} \cap \UnimodularSequences{\Ring^{\infty}}$ and $d  = n - \StableRank(\Ring) - 1$.  We need to define the value of the function $K^{(3)}$ at $(n, \StableRank(\Ring))$ such that $\Poset$ is uniformly $(d+1)$-acyclic and satisfies the q-UBC for $q \leq d+1$ with constant $\UBCConstantGLThree{n}{\StableRank(\Ring)}$.

	The next construction is exactly as in \cite[§2.17.]{vanderkallen1980homologystabilityforlineargroups}: We set $X \coloneqq \Ring^n + e_{n+1} \delta$ and $y_0 \coloneqq e_{n+1} \delta + e_{n+2}$. For $(v_1, \dots, v_k) \in F \setminus \OrderedSequences{X}$, we study $\OrderedSequences{X} \cap F_{(v_1, \dots, v_k)} = \OrderedSequences{X} \cap \UnimodularSequences{\Ring^{\infty}}_{(v_1, \dots, v_k)}$: We may assume $v_1 \notin X$; otherwise permute the $v_i$. The definition of $F$ implies that the $(n + 2)$-nd co-ordinate of $v_1$ equals 1. Setting $v_i' = v_i - \alpha_i v_1$ for $\alpha_i \in \Ring$, we have that $\UnimodularSequences{\Ring^{\infty}}_{(v_1, v_2 \dots, v_k)} = \UnimodularSequences{\Ring^{\infty}}_{(v_1, v_2' \dots, v_k')}$. For each $2 \leq i \leq k$, we can choose $\alpha_i$ such that the $(n + 2)$-nd co-ordinate of $v_i'$ vanishes. As a consequence,
	$\OrderedSequences{X} \cap F_{(v_1, \dots, v_k)} = \OrderedSequences{X} \cap \UnimodularSequences{\Ring^{\infty}}_{(v_2' \dots, v_k')}.$
	The poset $\OrderedSequences{X} \cap \UnimodularSequences{\Ring^{\infty}}_{(v_2' \dots, v_k')}$ is uniformly $(n - \StableRank(\Ring) -1 - (k-1))=((d+1)-k)$-acyclic and satisfies the q-UBC for $q \leq (d+1)-k$ with constant $\UBCConstantGLTwo{n}{\StableRank(\Ring)}{k-1}$ by \autoref{item-2-BoundedAcyclicityResultTechnical} of the induction hypothesis.

	Because $(v_1, \dots, v_k) \in \Poset$ and $F \subseteq \Ring^{n+2}$, it follows from \autoref{rem:MaximalLengthOfUnimodularSequences} that $k \leq n+2+\StableRank(\Ring)$. We set
	$$K^{(2)}_{\max}(n, \StableRank(\Ring)) \coloneqq \max_{k \in \{1, \dots, n+2+\StableRank(\Ring)\}} K^{(2)}(n, \StableRank(\Ring), k).$$
	Therefore, we can invoke \autoref{item-2-LinkLemmaII} of \autoref{lem:LinkLemmaII} using
	$$K = K(n,\StableRank(\Ring)) \coloneqq K^{(2)}_{\max}(n, \StableRank(\Ring)) \text{ and } l = n+2,$$
	to conclude that $\Poset$ is uniformly $(n - \StableRank(\Ring)) = (d+1)$-acyclic and satisfies the q-UBC for $q \leq d+1$ with constant
	$$\UBCConstantGLThree{n}{\StableRank(\Ring)} \coloneqq \UBCConstantIntermediateGLII{n+2}{\StableRank(\Ring)}{n - \StableRank(\Ring) - 1}{K(n,\StableRank(\Ring))}.$$

	\textsl{Proof of \autoref{item-4-BoundedAcyclicityResultTechnical}, i.e.\ the analogue of \cite[§2.18.]{vanderkallen1980homologystabilityforlineargroups}:}
	Let $\Poset = \OrderedSequences{(\Ring^{n} + e_{n+1} \delta) \cup (\Ring^{n} + e_{n+1} \delta + e_{n+2})} \cap \UnimodularSequences{\Ring^{\infty}}_{(v_1, \dots, v_k)}$ and $d  = n - \StableRank(\Ring) - k$. We need to define the value of the function $K^{(4)}$ at $(n, \StableRank(\Ring), k)$ such that $\Poset$ is uniformly $d$-acyclic and satisfies the q-UBC for $q \leq d$ with constant $\UBCConstantGLFour{n}{\StableRank(\Ring)}{k}$.

	For $k = 0$, we set $\UBCConstantGLFour{n}{\StableRank(\Ring)}{k} \coloneqq \UBCConstantGLThree{n}{\StableRank(\Ring)}$. Now, assume $k > 0$.

	If $n < \StableRank(\Ring)$, then $d < -1$ because $\StableRank(\Ring) > 0$ and $k > 0$. Hence, if $n < \StableRank(\Ring)$ the condition is empty and trivially satisfied for $\UBCConstantGLFour{n}{\StableRank(\Ring)}{k} \coloneqq 0$.

	If $n = \StableRank(\Ring)$, the only nontrivial case is $k = 1$. For this case, we need to check that $\Poset$ is non-empty. This can be done exactly as in \cite[§2.18.]{vanderkallen1980homologystabilityforlineargroups} using the stable range condition. In all of these cases, we set $\UBCConstantGLFour{n}{\StableRank(\Ring)}{k} \coloneqq 0$.

	Finally, let $n > \StableRank(\Ring)$. This case follows from \autoref{item-1-LinkLemmaII} of \autoref{lem:LinkLemmaII} similar to previous cases: Exactly as in \cite[§2.18.]{vanderkallen1980homologystabilityforlineargroups}, one can use the stable range condition to change coordinates such that
	$(\Ring^{n} + e_{n+1} \delta) \cup (\Ring^{n} + e_{n+1} \delta + e_{n+2})$ is preserved and that the first coordinate of $v_1$ is equal to 1. Then, one sets
	\begin{align*}
		X &= \{v \in (\Ring^{n} + e_{n+1} \delta) \cup (\Ring^{n} + e_{n+1} \delta + e_{n+2}): \text{ first coordinate of } v \text{ vanishes}\}.
	\end{align*}
	Then, $\OrderedSequences{X} \cap \Poset = \OrderedSequences{X} \cap \UnimodularSequences{\Ring^{\infty}}_{(v_1, \dots, v_k)}$. As in the proof of \autoref{item-3-BoundedAcyclicityResultTechnical}, we may choose $\alpha_i \in \Ring$ such that $v_i' = v_i - \alpha_i v_1$ has vanishing first coordinate. Again, it holds that $\UnimodularSequences{\Ring^{\infty}}_{(v_1, v_2 \dots, v_k)} = \UnimodularSequences{\Ring^{\infty}}_{(v_1, v_2', \dots, v_k')}$ and that
	$$\OrderedSequences{X} \cap \Poset =  \OrderedSequences{X} \cap \UnimodularSequences{\Ring^{\infty}}_{(v_1, v_2', \dots, v_k')} = \OrderedSequences{X} \cap \UnimodularSequences{\Ring^{\infty}}_{(v_2', \dots, v_k')}.$$
	As a consequence, $\OrderedSequences{X} \cap \Poset$ is uniformly $((n-1)-\StableRank(\Ring) - (k -1)) = d$-acyclic and satisfies the q-UBC for $q \leq d$ with constant
	$$\UBCConstantGLFour{n-1}{\StableRank(\Ring)}{k-1}$$
	by \autoref{item-4-BoundedAcyclicityResultTechnical} of the induction hypothesis, since up to a change of coordinates $X$ is equal to $(\Ring^{n-1} + e_{n} \delta) \cup (\Ring^{n-1} + e_n \delta + e_{n+1})$.
	Similarly, if $(w_1, \dots, w_{k'}) \in F \setminus \OrderedSequences{X}$ the poset $F_{(w_1, \dots, w_{k'})} \cap \OrderedSequences{X} = \OrderedSequences{X} \cap \UnimodularSequences{\Ring^{\infty}}_{(w_1', \dots, w_{k'}', v_2', \dots, v_{k}')}$ is uniformly $((n-1) - \StableRank(\Ring) - (k'+k-1)) = (d-k')$-acyclic and satisfies the q-UBC for $q \leq d-k'$ with constant $$\UBCConstantGLFour{n-1}{\StableRank(\Ring)}{k'+k-1}$$
	by \autoref{item-4-BoundedAcyclicityResultTechnical} of the induction hypothesis. Note that, because $(w_1, \dots, w_{k'}) \in \Poset$ and $F \subseteq \Ring^{n+2}$, it follows from \autoref{rem:MaximalLengthOfUnimodularSequences} that $k' \leq n+2+\StableRank(\Ring)$. We set
	$$K^{(4)}_{\max}(n-1, \StableRank(\Ring), k-1) \coloneqq \max_{k' \in \{1, \dots, n+2+\StableRank(\Ring)\}} K^{(4)}(n-1, \StableRank(\Ring), k' + k-1).$$
	Invoking \autoref{item-1-LinkLemmaII} of \autoref{lem:LinkLemmaII} for
	$$K = K(n,\StableRank(\Ring),k) \coloneqq \max \{K^{(4)}(n - 1, \StableRank(\Ring), k-1), K^{(4)}_{\max}(n-1, \StableRank(\Ring), k-1)\}$$
	and $l = n+2$, we conclude that $\Poset$ is uniformly $d = (n - \StableRank(\Ring) - k)$-acyclic and satisfies the q-UBC for $q \leq d = n - \StableRank(\Ring) - k$ with constant
	\begin{equation*}
		\UBCConstantGLFour{n}{\StableRank(\Ring)}{k} \coloneqq \UBCConstantIntermediateGLII{n+2}{\StableRank(\Ring)}{n - \StableRank(\Ring) - k}{K(n,\StableRank(\Ring),k)}. \qedhere
	\end{equation*}
\end{proof}

\section{Uniform acyclicity results for automorphism groups of quadratic
modules}
\label{scn:BoundedAcyclicityProofsForUnitaryGroups}

Throughout this section, we assume \autoref{con:quadratic-module-setting} and use the notation of \autoref{sec:automorphism-groups-of-quadratic-modules}. Our goal is to prove \autoref{thm:BoundedAcyclicityResultForComplexOfHyperbolicSplitInjections}, i.e.\ \autoref{item:stability-complex-sp} of \autoref{thm:general-connectivity}. For this we establish a bounded cohomology analogue of \cite[Section 3 and Section 4]{GalatiusRandalWilliamsStability}. The structure of proofs and the notation
in this section is intentionally chosen to be close to that in \cite[Section 3 and Section 4]{GalatiusRandalWilliamsStability}.

\begin{definition}
	Let $\QuadraticModule$ be an $\FormParameter$-quadratic module and
	let $\HyperbolicModule$ be the $\FormParameter$-hyperbolic module. We
	denote
	by $\SimplicialComplexS^\QuadraticModule$ the simplicial complex whose vertex set is given by set of
	$\FormParameter$-quadratic module morphisms
	\[
	h: \HyperbolicModule \to \QuadraticModule.
	\]
	A $k$-simplex in $\SimplicialComplexS^\QuadraticModule$ is a set $\{h_0, \dots, h_k\}$ of $k+1$ vertices with the property that their images $h_i(\HyperbolicModule)$ are pairwise orthogonal subspaces of $\QuadraticModule$. We call the $\SimplicialComplexS^\QuadraticModule$ the complex of unordered hyperbolic $\FormParameter$-split injections into $\QuadraticModule$.
\end{definition}

Notice that the complex of hyperbolic $\FormParameter$-split injections $\SimplicialComplex_\bullet^\QuadraticModule$ introduced in \autoref{def:ComplexOfHyperbolicSplitInjections} is related to the unordered version $\SimplicialComplexS^\QuadraticModule$ by
\[
(\SimplicialComplexS^\QuadraticModule)^{\text{ord}}_\bullet = \SimplicialComplex_\bullet^\QuadraticModule,
\]
where $(\SimplicialComplexS^\QuadraticModule)^{\text{ord}}_\bullet$ is obtained by applying the construction in \autoref{def:ord-complex} to $\SimplicialComplexS^\QuadraticModule$. In particular, the next theorem, \autoref{lem:ord-complex-connectivity-lemma} and \autoref{cor:qUBCimpliesAcyclic}	 imply \autoref{thm:BoundedAcyclicityResultForComplexOfHyperbolicSplitInjections}.

\begin{theoremnum}
	\label{thm:UniformAcyclicityResultForHyperbolicSplitInjections}
	There exists a function $\UBCUnitaryGroups: \NaturalNumbers \to \Reals$
	such that for any quadratic module $\QuadraticModule$ satisfying
	$\StableWittIndex{\QuadraticModule} \geq g$ it holds that the complex of unordered hyperbolic $\FormParameter$-split injections $\SimplicialComplexS^\QuadraticModule$ is uniformly $\lfloor \frac{g-4}{2} \rfloor$-acyclic and $\lCMub(\SimplicialComplexS^\QuadraticModule) \geq \lfloor \frac{g-1}{2} \rfloor$ with UBC-constant $\UBCUnitaryGroups(g)$, respectively. In particular, $\SimplicialComplexS^\QuadraticModule$ is uniformly Cohen--Macaulay of level
	$\lfloor \frac{g-2}{2} \rfloor$ with UBC constant $\UBCUnitaryGroups(g)$.
\end{theoremnum}

The proof of \autoref{thm:UniformAcyclicityResultForHyperbolicSplitInjections} is by induction on $g$, i.e.\ the stable Witt index of $\QuadraticModule$. The following \cite[Proposition 4.3]{GalatiusRandalWilliamsStability} of Galatius--Randal-Williams is the induction beginning.

\begin{proposition}
	\label{prop:UnitaryGroupsUniformAcyclicityLowDegrees}
	Let $\SimplicialComplexS^\QuadraticModule$ be the complex of unordered hyperbolic $\FormParameter$-split injections into a $\FormParameter$-quadratic module $\QuadraticModule$.
	\begin{itemize}
		\item If $\StableWittIndex{\QuadraticModule} \geq 2$, then
		$\SimplicialComplexS^\QuadraticModule \neq \emptyset$.
		\item If $\StableWittIndex{\QuadraticModule} \geq 4$, then
		$\SimplicialComplexS^\QuadraticModule$ is path-connected with diameter less or equal than $4$.
	\end{itemize}
\end{proposition}

\begin{proof}
	For $\Ring = \Integers$ this is exactly the content of \cite[Proposition
	4.3]{GalatiusRandalWilliamsStability}; the bound on the diameter is
	established in its proof.
	If $\Ring$ is a field of characteristic zero, one can run exactly the same argument using that: the analogue of \cite[Proposition 3.3]{GalatiusRandalWilliamsStability} holds by the same proof; the analogue of \cite[Proposition 3.4]{GalatiusRandalWilliamsStability} also holds by the same proof using the previous item; the analogue of \cite[Proposition 4.1]{GalatiusRandalWilliamsStability} holds by e.g.\ \cite[Lemma 6.6]{mirzaiivanderkallen2002} (use $k = 1$ and that $\on{usr}(\Ring) = 1$ for local rings by \cite[Example 6.5]{mirzaiivanderkallen2002}); the analogue of \cite[Corollary 4.2]{GalatiusRandalWilliamsStability} holds by the same proof using the previous item.
\end{proof}

With these preparation, we are now ready to discuss the proof of \autoref{thm:UniformAcyclicityResultForHyperbolicSplitInjections}. The argument is the uniformly bounded analogue of \cite[Proof of Theorem 3.2]{GalatiusRandalWilliamsStability}; its structure and the notation used in it is parallel to that of \cite[Proof of Theorem 3.2]{GalatiusRandalWilliamsStability}.

\begin{proof}[Proof of \autoref{thm:UniformAcyclicityResultForHyperbolicSplitInjections}]
	The function $\UBCUnitaryGroups: \NaturalNumbers \to \Reals$ is
	constructed to be monotonous; i.e.\ $\UBCUnitaryGroups(g_1) \leq
	\UBCUnitaryGroups(g_2)$ if $g_1 \leq g_2$. Using this property, one
	can check the local uniform acyclicity statement
	$\lCMub(\SimplicialComplexS^\QuadraticModule) \geq \lfloor
	\frac{g-1}{2} \rfloor$ at each step of the proof of the global
	uniform acyclicity statement by invoking the induction hypothesis
	exactly as described in the first paragraph of \cite[Proof of Theorem
	3.2]{GalatiusRandalWilliamsStability} (The main idea here is that the
	links of a simplex, will again be isomorphic to complexes of
	unordered hyperbolic $\FormParameter$-split injections).

	We now use induction on $g$ to check the global statement, i.e.\ that $\SimplicialComplexS^\QuadraticModule$ is uniformly $\lfloor \frac{g-4}{2} \rfloor$-acyclic:	For $g \in \{0,1\}$, the claim is void. Hence we define $\UBCUnitaryGroups(0) \coloneqq 0$ and $\UBCUnitaryGroups(1) \coloneqq 0$. For $g \in \{2,3\}$, the claim is that $\SimplicialComplexS^\QuadraticModule$ is non-empty. This follows from the first part of \autoref{prop:UnitaryGroupsUniformAcyclicityLowDegrees}. The UBC part of the claim is void, therefore we may set $\UBCUnitaryGroups(2) \coloneqq 0$ and $\UBCUnitaryGroups(3) \coloneqq 0$. For $g \in \{4,5\}$, we claim that $\SimplicialComplexS^\QuadraticModule$ is path-connected and the UBC part asserts that $\SimplicialComplexS^\QuadraticModule$ has finite diameter. This follows from the second part of \autoref{prop:UnitaryGroupsUniformAcyclicityLowDegrees}. Using \autoref{rem:0UBC}, we set $\UBCUnitaryGroups(4) \coloneqq 2$ and $\UBCUnitaryGroups(5) \coloneqq 2$.

	For the induction step, assume that $g > 5$ and that we have constructed a monotonous function $\UBCUnitaryGroups: \{0, \dots, g-1\} \to \Reals$ such that for any quadratic module
	$\QuadraticModule$ satisfying $\StableWittIndex{\Module} = g'$ for
	some
	$g' \in \{0, \dots, g-1\}$ it holds that $\SimplicialComplexS^\QuadraticModule$ is
	uniformly $\lfloor \frac{g'-4}{2} \rfloor$-acyclic and satisfies $\lCMub(\SimplicialComplexS^\QuadraticModule) \geq \lfloor \frac{g'-1}{2} \rfloor$ with UBC constant $\UBCUnitaryGroups(g')$, respectively.

	Let $\Module$ be a $\FormParameter$-quadratic module with
	$\StableWittIndex{\Module} \geq g$. Since $g > 5$,
	\autoref{prop:UnitaryGroupsUniformAcyclicityLowDegrees}
	implies that $\SimplicialComplexS^\QuadraticModule \neq \emptyset$. Therefore we find a vertex $h: \HyperbolicModule \to \Module$, which is kept fixed throughout the argument. Let us denote by $\Module' = h(\HyperbolicModule)^\perp \subset \Module$ the orthogonal complement of the image of $h$ in $\Module$. Recall from \autoref{expl:hyperbolicmodule}, that $\HyperbolicModule$ has a basis $\{e,f\}$. Then $\Module \cong \HyperbolicModule \oplus \Module'$ and $h(e)^\perp \cong \Module' \oplus
	\Ring\{e\}$. The inclusion $\Module' \hookrightarrow \Module$ factors
	through the $\FormParameter$-module $\Module' \oplus \Ring\{e\}$, and this factorization
	induces inclusions of simplicial complexes
	\begin{equation}
		\SimplicialComplexS^{\QuadraticModule'}
		\xhookrightarrow{\textcircled{1}}
		\SimplicialComplexS^{\QuadraticModule' \oplus \Ring\{e\}}
		\xhookrightarrow{\textcircled{2}}
		\SimplicialComplexS^{\QuadraticModule}.
	\end{equation}
	Our first aim is to prove that the inclusions $\textcircled{1}$ and $\textcircled{2}$ are both uniformly
	$\lfloor \frac{g-4}{2} \rfloor$-acyclic using \autoref{lem:FullSubcomplex}.

	\textsl{Claim 1: The map $\textcircled{1}$ is uniformly
	$\lfloor \frac{g-4}{2} \rfloor$-acyclic with UBC constant} $$K^{\textcircled{1}}(g) \coloneqq \max_{q \leq \lfloor \frac{g-4}{2} \rfloor} \UBCConstantTechnicalLemmaI{q}{\lfloor \frac{g-4}{2} \rfloor}{\UBCUnitaryGroups(g - 1)}.$$
	As in \cite[Proof of Theorem 3.2]{GalatiusRandalWilliamsStability}, we use that the projection $\pi: \Module' \oplus \Ring\{e\} \to \Module'$ is a morphism of $\FormParameter$-modules, which induces a retraction $\pi: \SimplicialComplexS^{\Module' \oplus \Ring\{e\}} \to \SimplicialComplexS^{\Module'}$. For any $p$-simplex $\Simplex{p}$ of $\SimplicialComplexS^{\Module' \oplus \Ring\{e\}}$ having no vertices in $\SimplicialComplexS^{\Module'}$, it then holds that
	\[
	\SimplicialComplexS^{\Module'} \cap \Link{\SimplicialComplexS^{\Module' \oplus \Ring\{e\}}}{\Simplex{p}}
	 = \Link{\SimplicialComplexS^{\Module'}}{\pi(\Simplex{p})}.
	\]
	To invoke \autoref{lem:FullSubcomplex}, we need to check that this complex is uniformly $(\lfloor \frac{g-4}{2} \rfloor-p-1)$-acyclic with a suitable UBC constant. The splitting $\QuadraticModule \cong \Module' \oplus \HyperbolicModule$ implies that $\StableWittIndex{\Module'} \geq g-1$, hence the induction hypothesis implies that $\lCMub(\Module') \geq \lfloor \frac{g'-1}{2} \rfloor$ with UBC constant $\UBCUnitaryGroups(g-1)$. Therefore, $\Link{\SimplicialComplexS^{\Module'}}{\pi(\Simplex{p})}$ is uniformly $(\lfloor \frac{g - 2}{2} \rfloor - p - 2) = (\lfloor \frac{g-4}{2} \rfloor - p - 1)$-acyclic with constant UBC-constant $\UBCUnitaryGroups(g - 1)$ and the claim follows from \autoref{lem:FullSubcomplex}.

	\textsl{Claim 2: The map $\textcircled{2}$ is uniformly $\lfloor \frac{g-4}{2} \rfloor$-acyclic with UBC constant}
	$$K^{\textcircled{2}}(g) \coloneqq \max_{q \leq \lfloor \frac{g-4}{2} \rfloor} \UBCConstantTechnicalLemmaI{q}{\lfloor \frac{g-4}{2} \rfloor}{\max_{p \leq g-2} \UBCUnitaryGroups(g - p - 2)}.$$
	As in \cite[Proof of Theorem 3.2]{GalatiusRandalWilliamsStability}, we note that $\Module' \oplus \Ring\{e\} = h(e)^\perp$ is the orthogonal complement of $e$ in $\HyperbolicModule$, consider a $p$-simplex $\Simplex{p} = \{h_0, \dots, h_p\}$ in $\SimplicialComplexS^{\QuadraticModule}$ having no vertices in $\SimplicialComplexS^{\QuadraticModule' \oplus \Ring\{e\}}$ and denote by $\Module'' = (\oplus_0^p h_i(\HyperbolicModule))^\perp \subset \QuadraticModule$. Then
	\[
	\SimplicialComplexS^{\Module' \oplus \Ring\{e\}} \cap \Link{\SimplicialComplexS^{\Module'}}{\Simplex{p}}
	= \SimplicialComplexS^{\Module'' \cap h(e)^\perp}.
	\]
	As in \cite[Proof of Theorem 3.2]{GalatiusRandalWilliamsStability}, $\StableWittIndex{\Module'' \cap h(e)^\perp} \geq g - p - 2$. The induction hypothesis therefore implies that $\SimplicialComplexS^{\Module'' \cap h(e)^\perp}$ is at least uniformly $\lfloor \frac{(g- p - 2) - 4}{2} \rfloor \geq (\lfloor \frac{g-4}{2} \rfloor - p - 1)$-acyclic with constant UBC constant $\UBCUnitaryGroups(g - p - 2)$. In particular, it follows that for every $p$ the UBC constant of $\SimplicialComplexS^{\Module' \oplus \Ring\{e\}} \cap \Link{\SimplicialComplexS^{\Module'}}{\Simplex{p}}$ in degree $q \leq \lfloor \frac{g-4}{2} \rfloor-p-1$ can be chosen as $\max_{p \leq g-2} \UBCUnitaryGroups(g - p - 2)$. Therefore \autoref{lem:FullSubcomplex} implies the claim.

	\textsl{Completing the proof.} \autoref{lem:compositions-of-highly-acyclic-maps}, Claim 1 and Claim 2 imply that the inclusion $\SimplicialComplexS^{\QuadraticModule'} \hookrightarrow	\SimplicialComplexS^{\QuadraticModule}$ is uniformly $\lfloor \frac{g-4}{2} \rfloor$-acyclic and that its UBC constant in degree $q \leq \lfloor \frac{g-4}{2} \rfloor$ can be chosen as the value $\UBCConstantFactorThrough{q}{K^{\textcircled{2}}(g)}{K^{\textcircled{1}}(g)}$. The image of this composition is contained in the star of the vertex $h$ in $\SimplicialComplexS^{\QuadraticModule}$. This is a cone and uniformly $\infty$-acyclic with UBC-constant 1 by \autoref{lem:ConeUBC}. Therefore, \autoref{lem:factoring-through-highly-uniformly-acyclic-subcomplex} implies that $\SimplicialComplexS^{\QuadraticModule}$ is uniformly $\lfloor \frac{g-4}{2} \rfloor$-acyclic and that its UBC constant in degree $q \leq \lfloor \frac{g-4}{2} \rfloor$ can be chosen as $\UBCConstantFactorThrough{q}{\UBCConstantFactorThrough{q}{K^{\textcircled{2}}(g)}{K^{\textcircled{1}}(g)}}{1}$. Removing the dependence on $q$, we may choose
	$$K^{\on{Aut}}_{\on{help}}(g) \coloneqq \max_{q \leq \lfloor \frac{g-4}{2} \rfloor} \UBCConstantFactorThrough{q}{\UBCConstantFactorThrough{q}{K^{\textcircled{2}}(g)}{K^{\textcircled{1}}(g)}}{1}$$
	as the UBC constant of $\SimplicialComplexS^{\QuadraticModule}$ in degree $q \leq \lfloor \frac{g-4}{2} \rfloor$. To ensure the monotonicity of the function $\UBCUnitaryGroups: \NaturalNumbers \to \Reals$, we may hence define
	\begin{equation*}
	\UBCUnitaryGroups(g) \coloneqq \max \{K^{\on{Aut}}_{\on{help}}(g), \UBCUnitaryGroups(g-1), \dots, \UBCUnitaryGroups(0)\}. \qedhere
	\end{equation*}
\end{proof}

\emergencystretch=10em
\printbibliography
\end{document}